\documentclass[reqno]{memo-l}

\usepackage[dvips]{graphicx}
\usepackage[all,arc,curve,color,frame,line]{xypic}

\usepackage{amsmath,graphics}
\usepackage{amsfonts,amssymb,color}

\theoremstyle{plain}
\newtheorem*{theorem*}{Theorem}
\newtheorem*{lemma*}{Lemma}
\newtheorem*{corollary*}{Corollary}
\newtheorem*{corollary-1}{Corollary 1}
\newtheorem*{corollary-2}{Corollary 2}

\newtheorem*{proposition*}{Proposition}
\newtheorem*{proposition-1}{Proposition 1}
\newtheorem*{proposition-2}{Proposition 2}
\newtheorem*{proposition-3}{Proposition 3}
\newtheorem*{proposition-4}{Proposition 4}

\newtheorem{conjecture*}{Conjecture}
\newtheorem{theorem}{Theorem}[chapter]
\newtheorem{lemma}[theorem]{Lemma}
\newtheorem{corollary}[theorem]{Corollary}
\newtheorem{proposition}[theorem]{Proposition}

\theoremstyle{remark}

\newtheorem*{remark}{Remark}
\newtheorem*{remarks}{Remarks}
\newtheorem*{definition}{Definition}
\newtheorem*{notation}{Notation}
\newtheorem*{example}{Example}
\newtheorem*{examples}{Examples}
\newtheorem*{claim}{Claim}

\theoremstyle{definition}
\newtheorem{defn}[theorem]{Definition}

\def\scrg{\mathcal{G}}

\def\scry{\mathcal{Y}}
\def\scrx{\mathcal{X}}
\def\scrh{\mathcal{H}}
\def\scrn{\mathcal{N}}
\def\scrz{\mathcal{Z}}

\def \Spec{\operatorname{Spec}}
\def \Reg{\operatorname{Reg}}

\def \chr{\operatorname{char}}

\def\pp{\mathfrak{p}}

\def\mm{\mathfrak{m}}
\def\nn{\mathfrak{n}}
\def\pp{\mathfrak{p}}

\def\GL{\operatorname{GL}} \def\SL{\operatorname{SL}} \def\Q{\mathbb{Q}} \def\F{\mathbb{F}} \def\Z{\mathbb{Z}} \def\R{\mathbb{R}} \def\C{\mathbb{C}}
\def\N{\mathbb{N}}  \def\l{\lambda} \def\ll{\langle} \def\rr{\rangle}
 \def\a{\alpha}   \def\bp{\begin{pmatrix}}
\def\sm{\setminus} \def\ep{\end{pmatrix}} \def\bn{\begin{enumerate}} \def\Hom{\operatorname{Hom}}
 \def\rank{\operatorname{rank}}  \def\en{\end{enumerate}}
\def\ba{\begin{array}} \def\ea{\end{array}}  
   \def\a{\alpha} \def\b{\beta} 
\def\id{\operatorname{id}} \def\Aut{\operatorname{Aut}} \def\im{\operatorname{Im}} 
  
\def\ker{\operatorname{Ker}}\def\be{\begin{equation}} \def\ee{\end{equation}} \def\tr{\operatorname{tr}}

 \def\dim{\operatorname{dim}}

     \def\GL{\operatorname{GL}}

    \def\fr12{\frac{1}{2}} \def\z12{\Z[\fr12]}

\def\ol{\overline}

\def\G{\Gamma}

\def\chr{\operatorname{char}}

\newcommand\ad{\operatorname{ad}}

\renewcommand\epsilon{\varepsilon}
\DeclareMathAlphabet{\mathbf}{OML}{cmm}{b}{it}

\numberwithin{equation}{chapter}

\makeindex

\begin{document}

\frontmatter

\title{$3$-Manifold Groups are Virtually Residually $p$}
\author{Matthias Aschenbrenner}
\address{University of California, Los Angeles, California, USA}
\email{matthias@math.ucla.edu}

\author{Stefan Friedl}
\address{Mathematisches Institut\\ Universit\"at zu K\"oln\\   Germany}
\email{sfriedl@gmail.com}


\begin{abstract}
Given a prime $p$, a group is called residually $p$ if the intersection of its $p$-power index normal subgroups is trivial. A group is called virtually residually $p$ if it has a finite index subgroup which is residually $p$.  It is well-known  that finitely generated linear groups over fields of characteristic zero  are virtually residually $p$ for all but finitely many $p$. In particular, fundamental groups of hyperbolic $3$-manifolds are virtually residually $p$. It is also well-known that  fundamental groups of $3$-manifolds are residually finite. In this paper we prove a common generalization of these results: every $3$-manifold group is virtually residually $p$ for all but finitely many~$p$. This gives evidence for the conjecture (Thurston) that fundamental groups of $3$-manifolds are linear groups.
\end{abstract}

\maketitle

\cleardoublepage
\thispagestyle{empty}
\vspace*{13.5pc}
\begin{center}
  Dedicated to the memory of \\[2pt]

 {\large Sr.~M.~Benigna~Fehlner (1922--2008)}\\[2pt]

   who, as their teacher in kindergarten, first introduced the authors to each other. 
\end{center}
\cleardoublepage

\setcounter{page}{7}

\tableofcontents

\mainmatter

\chapter*{Introduction}

\section*{The main result}

\noindent
Let $p$ be a prime. In this paper, by a \emph{$p$-group} we mean a finite group whose order is a power of $p$.
Recall that every $p$-group is nilpotent, and in particular solvable.
Given a property $\mathfrak P$ of groups, a group $G$ is called \emph{residually $\mathfrak P$} if for every $g\in G$ with $g\neq 1$ there exists a morphism $\phi\colon G\to H$ onto a group $H$ enjoying property $\mathfrak P$ such that $\phi(g)\neq 1$.
For example, a group $G$ being residually a $p$-group (residually $p$, for short) simply means that the intersection of all $p$-power index normal subgroups of~$G$ is trivial. 
Being residually $p$ is a very strong property: e.g., every finite subgroup of a residually $p$ group is a $p$-group, and
every non-abelian subgroup of a residually $p$ group has a quotient isomorphic to $\F_p\times\F_p$.  Free groups  and finitely generated torsion free nilpotent groups
are well-known to be residually~$p$ for all $p$.

\index{group!residually $\mathfrak P$|textbf}
\index{group!residually finite}
\index{group!residually $p$}

In this paper we study $3$-manifold groups (i.e., groups which are fundamental groups of $3$-manifolds).
Throughout this paper we assume that the manifolds are connected and compact, but unless otherwise stated, we make no assumptions on orientability or on the type of boundary.
Compactness is not a serious restriction:  every finitely generated group which is the fundamental group of a non-compact $3$-manifold also appears as the fundamental group of a compact $3$-manifold (and hence is finitely presentable), cf.~\cite{RS90, Sc73}.
It is well-known that every finitely presentable group can be
realized as the fundamental group of a compact, closed, connected, smooth manifold of dimension $4$ (or higher).
However, there are severe restrictions on which groups may occur as fundamental groups of $3$-dimensional
manifolds; for example, it is well-known that
every $3$-manifold group is  residually finite. (This was first shown for Haken manifolds \cite{Th82,He87}, and since Perelman's proof of the Geometrization Conjecture, the proof now extends to include all $3$-manifolds.) Residual finiteness of $3$-manifold groups leads to
easy proofs for the decidability of the word problem in $3$-manifold groups, and for the decidability of knot triviality.
In general, however, $3$-manifold groups tend to be far from being residually $p$  or even just being residually finite solvable.
Indeed, let $K\subseteq S^3$ be a  knot and $G=\pi_1(S^3\sm \nu K)$, where $\nu K$ is a tubular neighborhood of $K$ in $S^3$. It  follows from Dehn's Lemma that $G\cong\Z$ if and only if $K$ is the trivial knot.
But it is a consequence of Stallings' theorem \cite[Theorem~3.4]{St65} that
any map from $G$ to a nilpotent group necessarily factors through the abelianization $G/[G,G]\cong\Z$. Furthermore, if $K$ is a knot with trivial Alexander polynomial, then in fact
any map from $G$ to a solvable group necessarily factors through $\Z$.

\index{Hempel}
\index{Perelman}

The purpose of this paper is to show that virtually the picture is very different.
One says that a group $G$ is \emph{virtually $\mathfrak P$} if there exists a finite index subgroup of $G$ which is $\mathfrak P$.
(Subgroups of  residually $p$ groups are again residually $p$, so virtually residually $p$ is equivalent to residually $p$-by-finite.) 
We call a finitely generated group \emph{linear} if it admits an embedding into $\GL(n,K)$ for some $n$ and some field~$K$ of characteristic zero (which may always chosen to be $K=\mathbb C$).
Many linear groups are not residually solvable (for example, $\operatorname{SL}(n,\Z)$ for $n > 2$, since these groups are perfect). On the other hand, it is well-known that
linear groups are, for all but finitely many $p$, virtually residually $p$; see, e.g., \cite[Theorem~4.7~(i)]{We73} or \cite[Window~7,~Proposition~9]{LS03}.
(This is essentially due to Mal$'$cev \cite{Mal40}; however, he only showed the weaker statement that linear groups are residually finite.)
Therefore, each $3$-manifold group which is linear is virtually residually $p$ for all but finitely many $p$. As is  well-known,  the fundamental group of each hyperbolic $3$-manifold and of each Seifert fibered manifold is linear (cf.~Sections~\ref{sec:hyperbolic} and \ref{sec:seifert} below), but it is  not known whether this is true for all $3$-manifold groups.

\index{group!virtually $\mathfrak P$}
\index{group!virtually residually $p$}
\index{group!residually solvable}
\index{group!linear}
\index{Mal$'$cev}

\medskip

Our main result now is the following.

\begin{theorem*} 
Let $N$ be a  $3$-manifold. Then for all but finitely many primes $p$ the group $\pi_1(N)$ is virtually residually $p$.
\end{theorem*}

Our proof of this theorem is substantially more involved than the proof of residual finiteness in \cite{He87}. Indeed, in order to show that the fundamental group~$G=\pi_1(N)$ of a given $3$-manifold $N$ is residually finite, one first easily reduces to the case that $N$ is closed, orientable and prime. Then $N$ admits a JSJ decomposition, hence may be viewed as the topological realization of a finite graph of based CW-complexes whose vertex spaces $N_v$ are either Seifert fibered or hyperbolic $3$-manifolds, and whose edge spaces $N_e$ are tori. By virtue of the Seifert-van Kampen Theorem, this gives rise to a description of $G$ as the fundamental group~$\pi_1(\scrg)$ of a (finite) graph of groups $\scrg$ whose vertex groups $G_v=\pi_1(N_v)$ are linear groups, and whose edge groups $G_e=\pi_1(N_e)$ are free abelian of rank~$2$. (See Section~\ref{sec:Graphs of Groups} below for a definition of a graph of groups.) As remarked, finitely generated linear groups are residually finite (in fact, virtually residually $p$ for all but finitely many $p$). So one may assume that the JSJ decomposition of $N$ is non-trivial. Let now $g\in G$, $g\neq 1$. By choosing suitable filtrations of the vertex groups~$G_v$, one then constructs a group morphism $\phi\colon G\to \widetilde{G}$ where $\widetilde{G}=\pi_1(\widetilde{\scrg})$ is the fundamental group of a graph $\widetilde{\scrg}$ of \emph{finite} groups (with the same underlying graph as $\scrg$), such that $\phi(g)\neq 1$. It now simply remains to prove that such a graph of finite groups has residually finite fundamental group \cite[proof of Theorem~3.1]{He87}. (For the special cases of amalgamated products or HNN extensions, this was known much earlier \cite{Ba63, BT78, Co77}.)
In contrast to this, amalgamated products and HNN extensions of $p$-groups, although always residually finite, are rarely residually~$p$~\cite{Hi64, Mo07}, making it necessary to deviate from the flow of the argument in \cite{He87} for the proof of our main theorem.
Another subtle difficulty is posed by the fact that for an infinite group $G$, the intersection of all $p$-power index subgroups of $G$ being trivial does not guarantee that the intersection of all $p$-power index {\it normal}\/ subgroups of $G$ is trivial (i.e., that $G$ is residually $p$):
It is easy to see that for every fibered $3$-manifold $N$ the intersection of all $p$-power index subgroups of the fundamental group $G$ of $N$ is trivial (cf.~Proposition~\ref{prop:fibered implies weakly residually p}); but there are fibered $3$-manifolds which are not residually nilpotent (e.g., as mentioned above, the fundamental group of a non-trivial fibered knot complement).

\index{Hempel}
\index{$3$-manifold!fibered}
\index{amalgamated product}
\index{HNN extension}
\index{JSJ decomposition}

\medskip

We outline the overall strategy of our proof after discussing a few applications, and some commonalities between $3$-manifold groups and linear groups.

\section*{Applications}

\noindent
Our theorem may be used to significantly simplify the  proof \cite{FV08} that twisted Alexander polynomials detect fibered knots.  More precisely, our main theorem replaces the ad hoc argument given in Sections 5 and 6 of \cite{FV08}.
This is explained in detail in \cite{FV10}.
Here, we restrict ourselves to discussing two more immediate applications.

First, combining the theorem above with the fact (proved in  \cite{Lu80}) that the automorphism group of every finitely generated virtually residually $p$ group is also virtually residually $p$, we obtain:

\begin{corollary-1}\label{cor:Corollary 1}
The automorphism group $\Aut(G)$ of every $3$-manifold group $G$ is virtually residually $p$ for all but finitely many $p$, and hence virtually torsion free.
Therefore, every torsion subgroup of $\Aut(G)$ is finite, and there is an upper bound on the orders of the finite subgroups of $\Aut(G)$.
\end{corollary-1}

This corollary complements \cite{Ko84} where it was shown that if $N$ is a Haken $3$-manifold which is not Seifert fibered, then there is an upper bound on the orders of the finite subgroups of the group $\operatorname{Out}(G)$ of outer automorphisms of~$G=\pi_1(N)$. However, since not every finite subgroup of $\operatorname{Out}(G)$ lifts to a finite subgroup of~$\Aut(G)$ (see~\cite{Zi05}), this result is not immediately subsumed by Corollary~1. It seems not to be known whether the group of outer automorphisms of a $3$-manifold group is virtually residually $p$.

\begin{corollary-2}
Suppose $N$ is a closed aspherical $3$-manifold. Then there is a bound on the size
of finite groups of self-homeomorphisms of $N$ having a common fixed point.
\end{corollary-2}

This follows immediately from Corollary~1 and the fact that under the assumptions  made here, for every effective action of a finite group $\G$ on $N$ with a fixed point, the natural morphism $\G\to\Aut(\pi_1(N))$ is injective, cf.~\cite{CR72}.
(Questions similar to the one addressed in Corollary~2 are posed in Kirby's list \cite{Ki93}, see Problem~3.39.)

\section*{Properties of linear groups and $3$-manifold groups}

\noindent
The main theorem of this paper stated above can also been seen as further evidence to the conjecture (due to Thurston) that $3$-manifold groups are linear.
(This is Problem~3.33~(A) in \cite{Ki93}.)
It is known that every  linear group $G$ has the following group-theoretic properties. (Recall from above that in our parlance, ``linear'' includes the condition of being finitely generated.)
\begin{enumerate}
\item $G$ is virtually residually $p$ for all but finitely many primes $p$ (in particular, $G$ is residually finite and virtually torsion free).
\item $G$ is either  virtually solvable of finite rank (hence has polynomial subgroup growth) or the subgroup growth of $G$ is at least of type $2^{(\log n)^2/\log\log n}$.
\item $G$ satisfies the Tits alternative: either $G$ is virtually solvable or $G$ contains a non-abelian free group.
\item $G$ satisfies the Lubotzky alternative: either $G$ is virtually solvable, or $G$ has ``large finite quotients.'' (See \cite[Window~9]{LS03} for a precise statement.)
\item $G$ has property $\nu$: if $K$ is a non-nilpotent normal subgroup (not necessarily finitely generated) of $G$, then there is a normal subgroup $L$ of $G$ such that $L\leq K$ and $K/L$ is finite and non-nilpotent.
\end{enumerate}
We refer to   \cite{Ti72}, \cite{We73} and \cite[Windows~7~and~9]{LS03} for proofs of statements (1)--(5) above.
Property $\nu$ was isolated and studied in \cite{BR91}.
Examples  for groups with property $\nu$ (besides all linear groups) include all polycyclic-by-finite groups. This class of groups is of interest  since for a finitely presented group with property $\nu$, one can algorithmically decide whether a given finitely generated normal subgroup is nilpotent. Moreover,
every group $G$ with property $\nu$ is an FGH (`Frattini-Gasch\"utz-Hall') group in the sense of \cite{Le73}, meaning that various important characteristic subgroups of $G$ (among them the Fitting subgroup of $G$ and the Hirsch-Plotkin radical of $G$) agree and are nilpotent.

\index{Thurston}
\index{Lubotzky}
\index{group!linear}
\index{group!with property $\nu$}

It is a consequence of the proof of the Geometrization Conjecture and the work of various authors that  statements (1)--(4) also hold for all $3$-manifold groups $G$:
Property (1) is our main theorem. 
(Virtual torsion freeness also follows from the fact that irreducible $3$-manifolds with infinite fundamental group $G$ are Eilenberg-Mac~Lane spaces of type $K(G,1)$.)
Property (1) also yields property (2) by \cite[Theorem~8.3]{LS03}. However, a stronger estimate is proved in \cite[Theorem~1.1]{La05}: if $G$ is a $3$-manifold group which is not virtually solvable, then the subgroup growth of $G$ is at least of type $2^{n/\sqrt{\log n}\log\log n}$.
The fact that (3) holds for $3$-manifold groups is well-known to the experts.
It is also known (and implicit in \cite[Theorem~4.1]{FV11c}) that every $3$-manifold group $G$ satisfies a `weak Lubotzky alternative': either $G$ is virtually solvable or $vb_1(G,\F_p) = \infty$ for every prime $p$.
Here,  given a ring $R$ the virtual first  $R$-Betti number $vb_1(G,R)$ of a group $G$ is the defined as the supremum (possibly $\infty$) of all first $R$-Betti numbers of normal, finite index subgroups of $G$.
(The weak Lubotzky alternative for linear groups is a consequence of the Lubotzky alternative, cf.~\cite[Theorem 1.3]{La05}.)
Property (5) appears not to be known for arbitrary $3$-manifold groups; in a future publication we plan address this issue.

In \cite{Lub88}, Lubotzky gave a group-theoretic characterization of linear groups among finitely generated virtually residually $p$ groups, and
it is tempting to employ it in combination with the methods in the present paper to prove the linearity of $3$-manifold groups. However, in general it seems to be difficult to verify that a given group satisfies this criterion.
(For some rare cases where this has been done successfully see \cite{CCLP, Ta93}.)

\section*{Outline of the proof strategy}

\noindent
Let $N$ be a $3$-manifold. In order to show that the finitely presentable group $G=\pi_1(N)$ is virtually residually $p$ for all but finitely many $p$, we proceed in five steps.

\subsection*{Step 1. First reductions.} We first show that it suffices to consider only $3$-manifolds which are closed, orientable and prime.
Suppose $N$ satisfies these requirements. Then $N$ admits a JSJ decomposition, so $G=\pi_1(\scrg)$ for a graph of groups $\scrg$ whose vertex groups $G_v=\pi_1(N_v)$ are fundamental groups of hyperbolic or Seifert fibered $3$-manifolds (and hence are linear groups), and whose edge groups~$G_e=\pi_1(N_e)$ are free abelian of rank $2$.
(Below  we identify the fundamental group $G_e$ of every boundary torus $N_e$ of a JSJ component $N_v$ of $N$ with a subgroup of $G_v=\pi_1(N_v)$ in the natural way.) Since finitely generated linear groups are virtually residually $p$ for all but finitely many $p$, we may assume from now on that the JSJ decomposition of $N$ is non-trivial. We say that $N$ is {\it Seifert simple}\/ if every Seifert fibered component in its JSJ decomposition is of the form $S^1\times F$ for some orientable surface $F$ which is either closed or has at least two boundary components.
After passing to a finite cover of $N$, we can further assume that $N$ is Seifert simple. (This step of the proof is analogous to the first part of the proof of residual finiteness in \cite{He87}.)

\index{JSJ decomposition}

\subsection*{Step 2. Existence of suitable filtrations.}
Next, we show that with the exception of finitely many primes $p$, after passing to another finite cover of $N$, we can construct what we call a {\it complete boundary compatible $p$-filtration $\mathbf G_v=\{G_{v,n}\}_{n\geq 1}$}\/ on  each vertex group $G_v=\pi_1(N_v)$. This is a descending sequence of normal subgroups $G_{v,n}$ of $G_v=G_{v,1}$ such that $G_{v,n}/G_{v,n+1}$ is an elementary abelian $p$-group, which exhausts $G_v$ (i.e., $\bigcap_n G_{v,n}=1$),
subject to further conditions, which we refrain from spelling out completely here (see Chapter~\ref{ch:proof of the main results} for the precise definition); for example, we require that $\mathbf G_v$ separates the fundamental group $G_e=\pi_1(N_e)$ of every boundary torus $N_e$ of $N_v$ (i.e., for each $g\in G_v\setminus G_e$ there is some $n\geq 1$ with $g\notin G_{v,n}\cdot G_e$),
and that it intersects to the lower central $p$-filtration of $G_e$ (i.e., $G_e^{p^{n-1}} = G_{v,n}\cap G_e$ for all $n\geq 1$).
The construction of these filtrations is inspired by the proof of the fact that linear groups are virtually residually $p$ for all but finitely many $p$,
 but needs  further ingredients: a localization theorem for finitely generated integral domains (which we prove by using
a combination of elementary commutative algebra and model theory) and some facts about congruence subgroups. From now on until the end of this discussion we assume that on each $G_v$ we are given such a  complete boundary compatible $p$-filtration $\mathbf G_v$. (In some cases, e.g., for graph manifolds, we can arrange for the lower central $p$-filtration $\gamma^p(G_v)=\{\gamma^p_n(G_v)\}_{n\geq 1}$ of $G_v$ to be a complete boundary compatible $p$-filtration.)

\index{filtration}


\subsection*{Step 3. Residual properties of  fundamental groups of graphs of groups.}
A general criterion for the fundamental group of a graph of groups to be residually $p$ (Proposition~\ref{Hempel criterion residually p}) reduces our task to showing that for all $n\geq 1$, the graph of groups $\scrg_n$, obtained by replacing each vertex group $G_v$  and each edge group $G_e$ by their ($p$-power) quotients $G_v/G_{v,n}$ respectively $G_e/G_e^{p^{n-1}}$, has residually $p$ fundamental group. The case $n=2$ plays an important role.
If $\mathbf G_v=\gamma^p(G_v)$ for each vertex $v$, then $\scrg_2$ is the ``mod $p$ homology graph  of $\scrg$'': its vertex and edge groups are the first homology groups with $\F_p$-coefficients $H_1(G_v;\F_p)=G_v/\gamma^p_2(G_v)$ and $H_1(G_e;\F_p)=\F_p\oplus\F_p$ of the vertex respectively edge groups of $\scrg$.

\subsection*{Step 4. Reduction to $\scrg_2$.}
Let $T$ be a maximal subtree of the graph underlying the graph of groups $\scrg$.
The fundamental group $\pi_1(\scrg_2)$ of $\scrg_2$ can be described as an iterated HNN extension of the fundamental group $\pi_1(\scrg_2|T)$  of the restriction~$\scrg_2|T$ of $\scrg_2$ to $T$. Replacing $\pi_1(\scrg_2|T)$ in this iterated HNN extension by the
 fiber sum  $\Sigma=\Sigma(\scrg_2|T)$ (colimit in the category of abelian groups) of the tree~$\scrg_2|T$ of elementary abelian $p$-groups, we obtain a ``partial abelianization'' of~$\pi_1(\scrg_2)$ along~$T$, which we denote by  $\pi_1^*(\scrg_2,T)$. In this step we show that, possibly excluding finitely many more primes, if $\pi_1^*(\scrg_2,T)$ is residually $p$, then  all the fundamental groups of the quotient graphs $\scrg_n$ are residually $p$ (and hence $G=\pi_1(N)$ is residually $p$ by Step~3); cf.~Theorem~\ref{thm:reduction theorem}. This requires two main ingredients: an amalgamation theorem for {\it filtered}\/ $p$-groups (Theorem~\ref{thm:higman}) which sharpens a result of Higman~\cite{Hi64}, the proof of which requires some facts about wreath products and group rings; and analogous criteria for iterated HNN extensions of $p$-groups to be residually $p$ (Corollary~\ref{cor:chatzidakis}), related to work of Chatzidakis~\cite{Ch94}. 

\index{graph!of groups}
\index{graph of groups}
\index{partial abelianization}
\index{HNN extension}
\index{graph!mod $p$ homology}
\index{Higman}
\index{Chatzidakis}

\subsection*{Step 5. Unfolding a graph of groups.}
A sufficient condition for $\pi_1^*(\scrg_2,T)$ to be residually $p$ is the existence of a group morphism $\pi_1(\scrg_2,T)\to P$ to a finite abelian $p$-group $P$ which is injective on each vertex groups of $\scrg_2$.
Hence, in order to finish the proof, we show that after passing to a suitable finite cover of $N$, we can achieve exactly this.
Now, the identification morphisms  between the fundamental groups of boundary tori yield a collection of isomorphisms $\varphi_e$ between subgroups of~$\Sigma$ which provide obstructions to the existence of such a morphism.
Extending each $\varphi_e$ to an automorphism of $\Sigma$ yields a morphism $\psi\colon\pi_1(\scrg)\to \Aut(\Sigma)$ which is trivial on each vertex group of $\scrg$. To obtain the desired finite cover of $N$, we introduce a construction, which we call the {\it unfolding of $\scrg$ along $\psi$,}\/ which produces a morphism of graphs of groups $\phi\colon\widetilde{\scrg}\to\scrg$ and complete boundary compatible $p$-filtrations on the vertex groups of $\widetilde{\scrg}$ (which give rise to the associated quotient graph  of elementary abelian $p$-groups $\widetilde{\scrg}_2$) such that the induced morphism $\phi_*\colon \pi_1(\widetilde{\scrg})\to\pi_1(\scrg)$ is injective with image $\ker(\psi)$, and such that for some maximal subtree $\widetilde{T}$ of the underlying graph of $\widetilde{\scrg}$, we do have a morphism $\pi_1(\widetilde{\scrg}_2,\widetilde{T})\to \Sigma$
to a $p$-group which is injective on the vertex groups of $\widetilde{\scrg}_2$ as required.

\index{unfolding}

\section*{A more general theorem?}

\noindent
It is legitimate to wonder whether our main theorem is but the echo of a more general fact. For example, one may ask: {\it Is every finitely presented residually finite torsion free group, for all but finitely many $p$, virtually residually $p$?}\/ If we replace ``finitely presented'' by ``finitely generated'' and drop the condition ``torsion free,'' then the answer to this question is ``no'': if $G$ is a group which is virtually residually~$p$ for two distinct primes~$p$, then $G$ is virtually torsion free; however, for every prime $p$ there are examples of $p$-torsion groups which are
infinite, finitely generated, and residually finite, constructed by Golod \cite{Golod}. (On the other hand, it seems to be unknown whether there exist infinite {\it finitely presented}\/ periodic groups.) If we replace ``finitely presented'' by ``finitely generated,'' then the answer is also negative:
It is a well-known theorem of Grigorchuk \cite{Gri89} (see also \cite[Interlude~E]{DdSMS}) that
if $G$ is a finitely generated residually $p$ group which has word growth $\prec e^{\sqrt{n}}$, then $G$ is virtually nilpotent (and hence has polynomial growth).
Grigorchuk \cite{GriMaki, Gri85} also constructed a finitely generated residually finite torsion free group $\Gamma$ whose word growth is intermediate (i.e., between polynomial and exponential) and which is therefore not virtually residually $p$ for {\it any}\/ $p$. (
For a while, the question whether every finitely generated group of intermediate growth is residually finite was open \cite[VI.64]{dlH}, but this turned out to be false \cite{Er}.)

\medskip

We also note that it might be difficult to find a $3$-manifold $N$ for which the set of exceptional primes in our main theorem is non-empty: it has been conjectured (by Thurston) that every irreducible $3$-manifold which is not a graph manifold has a finite cover which is fibered; on the other hand, our techniques easily yield:

\index{Thurston}

\begin{proposition-1}
Let $N$ be a fibered $3$-manifold.  Then for \emph{every} prime $p$ the group $\pi_1(N)$ is virtually residually $p$.
\end{proposition-1}

\index{$3$-manifold!fibered}

Topics of this kind are also discussed, employing other arguments, in \cite{Ko11}.

\index{Grigorchuk}

\section*{Graph manifolds}

\noindent
The method for proving residual properties of $3$-manifold groups introduced in this paper is quite flexible. For example, it yields the following improvement of our main theorem for graph manifolds:

\begin{proposition-2}
Let $N$ be a graph manifold. Then for \emph{every} prime $p$ the group $\pi_1(N)$ is virtually residually $p$.
\end{proposition-2}

\index{graph manifold}

In \cite{PS99}, Perron and Shalen introduce the notion of a `weakly residually $p$-nilpotent' group, and  show that all graph manifold groups have a finite index subgroup which is weakly residually $p$-nilpotent for every $p$ \cite[Proposition~0.3]{PS99}. However, the proof for that claim  has a gap \cite[p.~36]{PS99}: in fact, one may show that a group is weakly residually $p$-nilpotent if and only if it is residually $p$, and  if $N$ is a Sol-manifold, then $\pi_1(N)$ does {\it not}\/ have a finite index subgroup which is
residually $p$ for every $p$. These matters will be discussed further in \cite{AF2}.

\index{Perron}
\index{Shalen}

\medskip

Proposition~2 may be refined to show that virtually, the pro-$p$ topology on the fundamental group  of a graph manifold $N$ interacts well with the structure of the graph of groups arising from the JSJ decomposition of $N$. To make this precise, suppose we are interested in a property $\mathfrak P$ of groups. The pro-$\mathfrak P$ topology  on a group~$G$ is the group topology with fundamental system of neighborhoods of the identity given by the normal subgroups $H\trianglelefteq G$ such that $G/H$ is a $\mathfrak P$-group. Inspired by a notion introduced in \cite{WZ98},
we say that a  $3$-manifold $N$ is \emph{$\mathfrak P$-efficient} if $N$ is closed prime orientable (so the JSJ decomposition of $N$ exists), its fundamental group $G=\pi_1(N)$ is residually $\mathfrak P$, and for each JSJ component~$N_v$ of~$N$ and each component $N_e$ of the boundary of $N_v$, the fundamental groups~$G_v=\pi_1(N_v)$ and $G_e=\pi_1(N_e)$, identified with subgroups of $G$ in the natural way, the following conditions hold:

\index{topology!pro-$\mathfrak P$}
\index{topology!pro-$p$}
\index{$p$-efficient}
\index{$\mathfrak P$-efficient}

\begin{enumerate}
\item $G_v$ and $G_e$ are closed in the pro-$\mathfrak P$ topology on $G$;
\item the pro-$\mathfrak P$ topology on $G$ induces the pro-$\mathfrak P$ topology on $G_v$ and on $G_e$.
\end{enumerate}
Wilton and Zalesskii \cite{WZ10} showed that all closed prime orientable $3$-mani\-folds are $\mathfrak P$-efficient, for $\mathfrak P$ the property of being finite.
We show:

\index{Wilton}
\index{Zalesskii}

\begin{proposition-3}
Let $N$ be a  closed graph manifold. Then for every $p$, $N$ has a finite cover which is $p$-efficient. 
\end{proposition-3}

It seems plausible that this result might play the same role in studying the (virtual) conjugacy $p$-separability of graph manifold groups as finite efficiency does in the proof of their conjugacy separability, which was also established in \cite{WZ10}. We do not pursue this issue further in the present paper; however, we note that Proposition~3 also implies that virtually, the fundamental group of a closed graph manifold group and its pro-$p$ completion have the same mod $p$ cohomology in a certain sense, explained in Section~\ref{sec:coh p-completeness}.

\medskip

We do not know whether every closed prime $3$-manifold is, for all but finitely many $p$ (or even just for infinitely many $p$), virtually $p$-efficient. This question will be discussed further in Section~\ref{sec:p-efficiency in general}. We note here that if $N$ is a  $p$-efficient $3$-manifold with fundamental group $G$, then
condition~(2) of $p$-efficiency implies that there exists some $n$ such that for every JSJ component $N_v$ of $N$ with fundamental group $G_v$ (identified in the natural way with a subgroup of $G$), we have $\gamma^p_2(G_v)\geq G_v\cap\gamma^p_n(G)$. As an indication that every closed prime $3$-manifold  might indeed be virtually $p$-efficient for all but finitely many $p$, we show that we can achieve this condition on $\gamma^p_2$ in a finite cover. In fact, we may take $n=2$, and so our result may be phrased as a statement about the mod $p$ homology of finite covers of $3$-manifolds as follows:

\begin{proposition-4}
Every closed prime $3$-manifold $N$ has, for all but finitely many~$p$, a finite cover $\widetilde{N}$
such that for every JSJ component $\widetilde{N}_{\widetilde{v}}$ of $\widetilde{N}$, the natural morphism $H_1(\widetilde{N}_{\widetilde{v}};\F_p)\to H_1(\widetilde{N};\F_p)$ is injective.
\end{proposition-4}

\section*{Guide for the reader}

\noindent
As should be apparent from the outline of the strategy given above, the proof of the main result of this paper  involves going through several technical prerequisites; some of those can be assumed as black boxes in a first reading. We recommend that the reader should first only skim the preliminary Chapter~\ref{ch:preliminaries}, and  take the main results of Chapter~\ref{ch:embedding theorems} for granted before immediately moving to Chapters~\ref{ch:residual properties of graphs of groups} and \ref{ch:proof of the main results}, which contain the essence of Steps~3--5 in the proof strategy indicated above. The various subsections labeled ``digression'' may also be skipped upon a first reading of the paper.
In the final Chapter~\ref{ch:graph manifolds} we give proofs of Propositions~2--4 from above.

\section*{Conventions and notations}

\noindent
Throughout this paper, $m$, $n$ range over the set $\N=\{0,1,2,\dots\}$ of natural numbers. Unless otherwise noted, $p$ denotes a prime number. By a $p$-group we mean a finite group whose order is a power of $p$.

\index{$p$-group}

\section*{Acknowledgments}

\noindent
The Fields Institute in Toronto provided kind hospitality to the first-named author while a part of the work on this paper was done.
He was also partially supported by National Science Foundation grant DMS-0556197.
The second-named author was partially supported by a CRM-ISM Fellowship and by CIRGET.

We would like to thank Steve Boyer for many insightful conversations. 
We are also grateful to Ian Agol, Christian Haesemeyer, Marc Lackenby, Yehuda Shalom, and Stefano Vidussi for helpful comments.
We also thank the referee for a careful reading of the manuscript.

\chapter{Preliminaries}\label{ch:preliminaries}

\section{Filtrations of groups}
\label{sec:filtrations of groups}

\noindent
We begin by introducing basic terminology and fundamental facts concerning filtrations of groups. Facts stated without proof or reference are standard; for more information on filtrations (in particular the connection to Lie algebras, not explicitly exploited here), see \cite{BL94}. Throughout this section we let $G$ be a group.

\subsection{Filtrations}\label{subsec:filtrations}
We call a descending sequence
$$G \geq G_1 \geq G_2 \geq \cdots \geq G_n \geq G_{n+1}\geq \cdots$$
of subgroups of $G$ a {\it filtration} of $G$.
We say that a filtration ${\mathbf G}=\{G_n\}$ as above is {\it complete} if $G_1=G$. 
We call a filtration $\mathbf G$ of $G$ {\it normal} if
each subgroup $G_n$ is normal in $G$; in this case, we write $L_n({\mathbf G}):=G_n/G_{n+1}$ for the $n$th ``layer'' of $\mathbf G$.

\index{filtration|textbf}
\index{filtration!normal}
\index{filtration!complete}
\index{filtration!layer}
\index{$L_n(\mathbf G)$}

\medskip
A {\it filtered group} is a pair $(G,\mathbf G)$ where $G$ is a group and $\mathbf G$ is a filtration of~$G$. The class of filtered groups are the objects of a category, where a morphism of filtered groups $(G,\mathbf G)\to (H,\mathbf H)$ is a group morphism $\varphi\colon G\to H$ such that $\varphi^{-1}(H_n)=G_n$ for all $n$. Note that if $\varphi\colon (G,\mathbf G)\to (H,\mathbf H)$ is a morphism between normal filtered groups, then for each $n$, $\varphi$ induces an injective group morphism $L_n(\varphi)\colon L_n(\mathbf G)\to L_n(\mathbf H)$.
The category of filtered groups has direct products: if $\mathbf G=\{G_n\}$ and $\mathbf H=\{H_n\}$ are filtrations of $G$ respectively $H$, then $\mathbf G\times\mathbf H=\{G_n\times H_n\}$ is a filtration of $G\times H$ (the {\it product}\/ of the filtrations $\mathbf G$ and $\mathbf H$). If both $\mathbf G$ and $\mathbf H$ are normal, then so is $\mathbf G\times\mathbf H$, with $L_n(\mathbf G\times\mathbf H)=L_n(\mathbf G)\times L_n(\mathbf H)$ for all $n$.

\index{group!filtered}
\index{morphism!of filtered groups}

\medskip
Let $\mathbf G=\{G_n\}$ be a filtration of $G$, and let $K\leq G$.
We call the filtration $\mathbf G\cap K:=\{G_n\cap K\}$ of $K$ the {\it intersection}\/ of $\mathbf G$ with $K$.
If $\mathbf K=\{K_n\}$ is a filtration of a subgroup of $G$
with  $\mathbf G\cap K=\mathbf K$, or equivalently, if the natural inclusion $K\to G$ is a morphism of filtered groups $(K,\mathbf K)\to (G,\mathbf G)$, then we say that $\mathbf G$ {\it intersects to $\mathbf K$} on $K$.

\index{filtration!intersection}

\medskip

We say that filtrations $\mathbf G=\{G_n\}$ and $\mathbf G^*=\{G_n^*\}$ of $G$ are \emph{equivalent} if
$$\{G_n: n\geq 1\}=\{G_n^*: n\geq 1\};$$
notation: $\mathbf G\sim \mathbf G^*$.
Clearly $\sim$ is an equivalence relation on filtrations of $G$.
We say that $\mathbf G$ {\it induces}\/ a filtration $\mathbf K$ on a subgroup $K$ of $G$ if $\mathbf G\cap K\sim\mathbf K$.
The following is obvious:

\begin{lemma}\label{lem:refine}
Let $K$ be a  common subgroup of $G$ and another group $H$. Let $\mathbf G$ be a filtration of $G$ and $\mathbf H$ be a filtration of $H$ with $\mathbf G\cap K\sim\mathbf H\cap K$. \textup{(}We say that $\mathbf G$, $\mathbf H$ \emph{induce the same filtration on $K$}.\textup{)} Then there are
filtrations $\mathbf G^*\sim\mathbf G$ of $G$ and $\mathbf H^*\sim\mathbf H$ of $H$ with $\mathbf G^*\cap K=\mathbf H^*\cap K$.
\end{lemma}

\index{filtration!induced}
\index{filtration!equivalent}

Let $\mathbf G=\{G_n\}$ and $\mathbf G^*=\{G_n^*\}$ be filtrations of $G$.
We say that \emph{$\mathbf G^*$ refines~$\mathbf G$} (and $\mathbf G^*$ is called a \emph{refinement} of $\mathbf G$) if there exists a strictly increasing map $\iota\colon\N^{>0}\to\N^{>0}$ such that $\iota(1)=1$ and
$G_n = G^*_{\iota(n)}$ for all $n\geq 1$.
If in addition $$G_n = G^*_{\iota(n)} = G^*_{\iota(n)+1} = \cdots = G^*_{\iota(n+1)-1}\qquad\text{for all $n\geq 1$,}$$
then we say that \emph{$\mathbf G^*$ stretches $\mathbf G$} (and $\mathbf G^*$ is called a \emph{stretching} of $\mathbf G$).   Note that in this case $\mathbf G^*\sim\mathbf G$, and the layers of ${\mathbf G}^*$ are given by
$$L_{m}({\mathbf G^*}) =
\begin{cases}
L_{\iota^{-1}(m+1)-1}({\mathbf G}) & \text{if $m+1\in\im\iota$} \\
1 & \text{if $m+1\notin\im\iota$.}
\end{cases}
$$
Given a collection $\mathcal G$ of groups, and for each $G\in\mathcal G$ a filtration
${\mathbf G}$ of $G$ and a stretching ${\mathbf G}^*$ of $\mathbf G$, we say that the  ${\mathbf G}^*$ (where $G$ ranges over $\mathcal G$) are {\it compatible}\/ if~$\iota$ can be chosen uniformly for all pairs $(\mathbf G,\mathbf G^*)$.

\index{filtration!refinement}
\index{filtration!stretching|textbf}
\index{stretching|textbf}

\medskip

Given $n\geq 0$, we say that a filtration $\mathbf G=\{G_n\}$ of $G$ has {\it length $\leq n$} if $G_{n+1}=1$; the smallest such $n$ is called the length of $\mathbf G$, and if no such $n$ exists, then $\mathbf G$ is said to have infinite length. We sometimes write a finite-length filtration of $G$ simply in the form
$$G\geq G_1 \geq G_2 \geq \cdots \geq G_n \geq G_{n+1}=1,$$
with the understanding that $G_m=1$ for all $m>n$.

\index{filtration!length}

\medskip

Given a subgroup $H$ of $G$
we say that a filtration $\mathbf G=\{G_n\}$ of $G$  {\it separates $H$} if $\bigcap_n HG_n=H$, and we simply say that $\mathbf G$ is {\it separating} if it separates the trivial subgroup $1$ of $G$. We say that a filtration of $G$  separates a collection  of subgroups of $G$ if it separates each member.

\index{filtration!separating|textbf}

\medskip

Given two filtrations $\mathbf G=\{G_n\}$ and $\mathbf G^*=\{G_n^*\}$ of $G$, we say that {\it $\mathbf G$ descends faster than $\mathbf G^*$}\/ (in symbols: $\mathbf G\leq\mathbf G^*$) if $G_n\leq G_n^*$ for all $n$.

\subsection{Central filtrations}
Let $\mathbf G=\{G_n\}$ be a complete filtration of $G$. We say that $\mathbf G$ is
a {\it central filtration} of $G$ if $[G,G_n]\leq G_{n+1}$ for all $n$. 
Every such central filtration $\mathbf G$ of $G$ is automatically normal, and  the successive quotients $L_n(\mathbf G)=G_n/G_{n+1}$  are central, that is,  $G_n/G_{n+1}\leq Z(G/G_{n+1})$ for all $n$. In particular, the groups $G_n/G_{n+1}$ are abelian, and every refinement of a central filtration is central.

\index{filtration!central}

\medskip

The {\it lower central filtration} of $G$   is the fastest descending central filtration $\gamma(G)=\{\gamma_n(G)\}$ of $G$, defined by
$$\gamma_1(G):=G,\qquad \gamma_{n+1}(G):=[G,\gamma_{n}(G)]\text{ for $n\geq 1$.}$$
The $\gamma_n(G)$ are totally invariant subgroups of $G$. The first quotient $G/\gamma_2(G)=G/[G,G]$ is the abelianization $G_{\operatorname{ab}}=H_1(G;\Z)$ of $G$. For each $n\geq 1$, the commutator map $$(g_n,g)\mapsto [g_n,g]\colon \gamma_n(G)\times G\to\gamma_{n+1}(G)$$ yields a surjective morphism $$(\gamma_n(G)/\gamma_{n+1}(G))\otimes_\Z G_{\operatorname{ab}}\to \gamma_{n+1}(G)/\gamma_{n+2}(G).$$
Thus
if $G_{\operatorname{ab}}$ is finitely generated (in particular, if $G$ is finitely generated), then so is each quotient $\gamma_n(G)/\gamma_{n+1}(G)$. (See, e.g., \cite[(4.6)]{BL94}.)

\index{filtration!lower central}
\index{$\gamma_n(G)$}

\medskip

Given $n$,
there exists a central separating filtration of $G$ of length $\leq n$ if and only if $\gamma_{n+1}(G)=1$; in this case, $G$ is said to be nilpotent of class $\leq n$. Moreover, the group~$G$ is residually nilpotent if and only if there exists a separating central filtration of $G$, and in this case the lower central filtration of $G$ is separating.

\medskip

A {\it strongly central filtration} of $G$ is a complete filtration $\{G_n\}$ of $G$ with the property that $[G_n,G_m]\leq G_{n+m}$ for all $n$ and $m$. (This is called an ``$N$-series'' in \cite{La54}, and sometimes called a ``Lazard filtration.'') The lower central filtration of $G$ is indeed strongly central.

\index{filtration!strongly central}


\subsection{$p$-Filtrations}
Let $\mathbf G=\{G_n\}$ be a filtration of $G$.
We say that $\mathbf G$ is a {\it $p$-filtration} if $(G_n)^p\leq G_{n+1}$ for all $n\geq 1$. Thus if $\mathbf G$ is a normal $p$-filtration of~$G$ then each non-trivial element of the quotient $L_n(\mathbf G)=G_n/G_{n+1}$ has order $p$. A filtration of $G$ equivalent to a $p$-filtration of $G$ is itself a $p$-filtration. Similarly, the intersection of a $p$-filtration of $G$ with any subgroup $K$ of $G$ is a $p$-filtration of~$K$, and the product of two $p$-filtrations is a $p$-filtration. (All these statements, of course, also hold with ``central filtration'' respectively ``strongly central filtration'' in place of ``$p$-filtration'' everywhere.)

\index{filtration!$p$-filtration}
\index{$p$-filtration}

\medskip

An important role in this paper is played by central $p$-filtrations: $\mathbf G$ is a central $p$-filtration of $G$ precisely if
$$G_1=G\quad\text{and}\quad [G, G_n] (G_n)^p \leq G_{n+1}\text{ for all $n\geq 1$.}$$
In this case, each successive quotient $L_n(\mathbf G)$ has a natural structure of an $\F_p$-linear space.
The {\it lower central $p$-filtration} $\gamma^p(G)=\{\gamma^p_n(G)\}$ of $G$   is  defined by
$$\gamma^p_1(G):=G,\qquad \gamma^p_{n+1}(G):=[G,\gamma^p_n(G)] \gamma^p_n(G)^p\text{ for $n\geq 1$.}$$
The $\gamma^p_n(G)$ are totally invariant: if $\alpha\colon G\to H$ is a group morphism, then $\alpha(\gamma^p_n(G))\leq \gamma^p_n(H)$ for all $n\geq 1$. We have $$G/\gamma^p_2(G)=H_1(G;\mathbb F_p)=H_1(G;\Z)\otimes_\Z \F_p,$$ and if $G$ is a $p$-group, then  $\gamma^p_2(G)=[G, G]\, G^p$ is the Frattini subgroup of $G$.
The lower central $p$-filtration is the fastest descending central $p$-filtration of $G$. It was introduced in \cite{Sk50, La54}, and some of its basic properties are described in detail in \cite[Chapter~VIII]{HB82} (under the name ``$\lambda$-series''), for example:

\begin{proposition}\label{prop:lower p}
For all $m,n\geq 1$:
\begin{enumerate}
\item $[\gamma^p_m(G),\gamma^p_n(G)]\leq \gamma^p_{m+n}(G)$;
\item $\gamma^p_n(G) = \gamma_1(G)^{p^{n-1}} \gamma_2(G)^{p^{n-2}}\cdots\gamma_n(G)$.
\end{enumerate}
\end{proposition}

For a proof see \cite[Chapter~VIII, Theorem~1.5]{HB82}.

\index{filtration!lower central $p$-filtration}
\index{$p$-filtration!lower central}
\index{$\gamma^p_n(G)$}

\begin{lemma}\label{lem:canonical filtration}
Let $T\leq G$.
Suppose there exists a central $p$-filtration $\mathbf G=\{G_n\}$ of $G$ which intersects to the lower central $p$-filtration on $T$.
Then every central $p$-filtration $\mathbf G^*$ of $G$ with $\mathbf G^*\leq\mathbf G$ intersects to the lower central $p$-filtration of~$T$.
In particular, the lower central $p$-filtration of $G$ intersects to the lower central $p$-filtration of $T$.
\end{lemma}

\begin{proof}
Suppose $\mathbf G^*$ is a central $p$-filtration of $G$ with $\mathbf G^*\leq\mathbf G$. Then $T\cap \mathbf G^*\leq T\cap \mathbf G=\gamma^p(T)$. On the other hand, $T\cap\mathbf G^*$  is a central $p$-filtration of $T$, hence $\gamma^p(T)\leq T\cap \mathbf G^*$.
\end{proof}

The following is shown by induction on $m$ (simultaneously for all $n$); we leave the details to the reader.

\begin{lemma}\label{lem:iterated gamma}
Let $\mathbf G=\{G_n\}$ be a central $p$-filtration of $G$. Then
$$\gamma^p_m(G_n) \leq G_{m+n-1}\qquad\text{for all $m,n\geq 1$.}$$
\end{lemma}

If there exists a central separating $p$-filtration  of $G$ of length $\leq n$, then we have $\gamma^p_{n+1}(G)=1$, and in this case we say that $G$ has \emph{lower $p$-length $\leq n$;}\/ the smallest such $n$ (if it exists) is the \emph{lower $p$-length}\/ of $G$.
So for example, $G$ has lower $p$-length~$1$ if and only if $G$ is a non-trivial elementary abelian $p$-group. In general,
a finite group is a $p$-group if and only if it has finite lower
$p$-length.
A finitely generated group $G$ is residually $p$ if and only if there exists a separating central $p$-filtration of $G$; in this case the lower central $p$-filtration of $G$ is separating.

\index{lower $p$-length}

\subsection{Dimensional $p$-filtrations}
A strongly central filtration~$\{G_n\}$ of $G$ is called a {\it dimensional $p$-filtration} if $(G_n)^p\leq G_{np}$ for all $n\geq 1$. (So every dimensional $p$-filtration is indeed a $p$-filtration.)
The fastest descending dimensional $p$-filtration~$\{D_n\}$ of $G$ is given by
$$D_1=G,\qquad D_n = (D_{\lceil n/p \rceil})^p \prod_{i+j=n} [D_i, D_j]\quad \text{for $n>1$.}$$
The $D_n$ are called the {\it dimension subgroups}\/ of $G$ in characteristic $p$. If we want to stress the dependence of $D_n$ on $G$ and $p$, we write $D_n^p(G)=D_n$. We call $\{D_n\}$ the {\it lower dimensional $p$-filtration}\/ of $G$. (This is also variously called the Zassenhaus, Jennings or Lazard series of $G$ in the literature. Dimensional $p$-filtrations are called ``$N_p$-series'' in \cite{Pa77}.) Its length  will be called the {\it lower dimensional $p$-length} of the group $G$.
Note that the $D_n$ are totally invariant, and clearly $D_n \geq \gamma^p_n(G) \geq \gamma_n(G)$ for all $n\geq 1$. A closed formula for $D$ is due to Lazard (cf.~\cite[Theorem~11.2]{DdSMS}):
$$D_n(G) = \prod_{ip^j\geq n} \gamma_i(G)^{p^j} \qquad\text{for each $n\geq 1$.}
$$
(The groups $D_n$ will only appear in Chapter~\ref{ch:embedding theorems}.)

\index{$p$-filtration!dimensional}
\index{dimension subgroup}
\index{$p$-filtration!lower dimensional}
\index{$D_n^p(G)$}

\subsection{Chief filtrations}
We say that a complete normal finite-length filtration $\mathbf G=\{G_n\}$ of $G$ is a {\it chief filtration}\/ of $G$ if each  quotient $L_n(\mathbf G)=G_n/G_{n+1}$ is either trivial or  a minimal non-trivial subgroup of $G/G_{n+1}$.  Every chief filtration of $G$ induces a chief filtration on any subgroup of $G$.
The quotients $L_n(\mathbf G)$ in a chief filtration of $G$ are called the {\it chief factors}\/ of the filtration. Every finite group has a chief filtration of finite length. 
 Every normal filtration of finite length of a finite group may be refined to a chief filtration of finite length. In particular, if $G$ admits a  finite-length chief filtration and $N$ is any normal subgroup of $G$, then one can always find a finite-length chief filtration in which $N$ is one of the terms. Every chief factor of a finite nilpotent group is central, and every chief factor of a $p$-group has order $p$. (Hence every chief filtration of a $p$-group is a central $p$-filtration.)

\index{filtration!chief}

\subsection{Filtrations and split morphisms}
Let $A\mapsto \Sigma(A)$ be an operation which associates to a group $A$ a totally invariant subgroup $\Sigma(A)$ of $A$, i.e., every group morphism $A\to B$ restricts to a group morphism $\Sigma(A)\to\Sigma(B)$. 
Recall that a subgroup $H$ of $G$ is called a \emph{retract} of $G$ if  the natural inclusion $H\to G$ left splits. Note that if $H$ is a free factor of $G$, i.e., if there exists a subgroup $H'$ of $G$ such that $G$ is (isomorphic to) the free product $H\ast H'$, then $H$ is a retract of $G$.

\begin{lemma}\label{lem:lower central for semidirect products}
If $H\leq G$ is a retract of $G$, then $$\Sigma(H)=\Sigma(G)\cap H.$$
If $G=B\rtimes H$ where $B\trianglelefteq G$ and $H\leq G$, then $$\Sigma(G) = \Sigma(H)(\Sigma(G)\cap B).$$
\end{lemma}
\begin{proof}
Let $H\leq G$.
We have $\Sigma(H)\leq\Sigma(G)\cap H$, and if $\phi\colon G\to H$ is a morphism which is the identity on $H$, then $\Sigma(G)\cap H = \phi(\Sigma(G)\cap H)\leq\Sigma(H)$. This shows the first statement. For the second statement,
assume $G=B\rtimes H$, where $H\leq G$, and let $\pi\colon G\to H$ be the natural morphism. Then
$\pi(\Sigma(G))\leq \Sigma(H)\leq \Sigma(G)$, thus $\Sigma(G) = \Sigma(H) (\Sigma(G)\cap B)$ as required.
\end{proof}

\index{retract}

This lemma applies in particular to the operation $\Sigma$ which, for given $n\geq 1$, yields the $n$th term in the lower central filtration $\gamma$, lower central $p$-filtration $\gamma^p$, or lower dimensional $p$-filtration $D$, respectively.
We also note:

\begin{lemma}\label{lem:retract}
Let $H\leq G$ be a retract of $G$, and suppose $G$ is residually $p$. Then $\gamma^p(G)$ separates $H$.
\end{lemma}
\begin{proof}
Let $\phi\colon G\to H$ be a morphism which is the identity on $H$.
Let $g\in\bigcap_{n\geq 1} \gamma^p_n(G)\cdot H$. Let $n\geq 1$, and take $h\in H$ with $g\equiv h\bmod \gamma^p_n(G)$. Then $\phi(g)\equiv \phi(h)\bmod\gamma^p_n(H)$ and hence $g\equiv\phi(g)\bmod\gamma^p_n(G)$. Since this holds for every~$n\geq 1$, and $\bigcap_{n\geq 1}\gamma^p_n(G)=1$, we obtain $g=\phi(g)\in H$ as required.
\end{proof}


\subsection{The layers of the lower central $p$-filtration}
For every $n\geq 1$ we write $L^p_{n}(G):=L_n(\gamma^p(G))$. Recall that each $L^p_n(G)$ is in a natural way an $\F_p$-linear space. If $G_{\operatorname{ab}}$ is finitely generated, then each of the layers $L^p_n(G)$ is finite (cf.~\cite[Ch.~II]{La54}; see also \cite[\S{}14]{BL94}).
Note that if $G$ is abelian and written additively, then $\gamma^p_n(G)=p^{n-1}G$ and so $L^p_n(G)=p^{n-1}G/p^n G$, for all $n$.
If $H$ denotes the abelianization $G_{\operatorname{ab}}=G/[G,G]$ of $G$, with natural morphism $\pi\colon G\to H$, then $\pi(\gamma^p_n(G)) = \gamma^p_n(H)=p^{n-1}H$  for every $n\geq 1$, and the restriction of $\pi$ to $\gamma^p_n(G)$ induces a surjective morphism
$$L^p_n(G)=\gamma^p_n(G)/\gamma^p_{n+1}(G) \to  p^{n-1}H/p^n H=L^p_n(H),$$
which is an isomorphism for $n=1$. 

\index{$L^p_n(G)$}

In the rest of this subsection we develop some further basic facts about the groups $L^p_n(G)$. The main tool is a precise version of the Hall-Petrescu formula \cite[Theo\-rem~1.1.30]{LGMcK}. In the theorem and its corollaries below,  we let~$x$,~$y$ range over~$G$.

\begin{theorem}\label{thm:HP}
For all $m,n\geq 1$,
$$(xy)^{m} = x^{m}\, y^{m}\, [y,x]^{m\choose 2}\, r_{mn}\mod \gamma_{n+1}(G)$$
where $r_{mn}$ is a product of powers $c^e$ with $c\in\gamma_i(G)$ and $e\in\N$ is a $\Z$-linear combination of ${m\choose 1}, {m\choose 2},\dots,{m\choose i}$,
for some $i$ with $3\leq i\leq m$.
\end{theorem}

This yields a sharpening of \cite[Lemma~11.9]{DdSMS}:

\begin{corollary}\label{cor:HP}
Let $m\geq 1$. Then
$$(xy)^{2^m} \equiv x^{2^m}\, y^{2^m}\, [x,y]^{2^{m-1}} \mod \gamma_2(G)^{2^m}\gamma_3(G)^{2^{m-1}}\prod_{d=2}^m \gamma_{2^d}(G)^{2^{m-d}},$$
and if $p$ is odd, then
$$(xy)^{p^m} \equiv x^{p^m}\, y^{p^m} \mod \gamma_2(G)^{p^m}\prod_{d=1}^m \gamma_{p^d}(G)^{p^{m-d}}.$$
\end{corollary}
\begin{proof}
We apply the previous theorem with $n=p^m-1$ and $p^m$ in place of~$m$.
If $p$ is odd then ${p^m\choose 2} \geq p^m$ and hence $[y,x]^{p^m\choose 2}\in\gamma_2(G)^{p^m}$.
For $p=2$ note that $[y,x]^{2^m\choose 2}\equiv [x,y]^{2^{m-1}}\bmod\gamma_2(G)^{2^m}$ since
$$[y,x]^{2^m\choose 2} = [y,x]^{2^{2m-1}-2^{m-1}}=[y,x]^{2^{2m-1}}[x,y]^{2^{m-1}}.$$
It is well-known that if $j=p^{d_1} j_1$ where $d_1,j_1\in\N$, $j_1>0$  not divisible by $p$, then  ${p^m\choose j}$ is divisible by $p^{m-d_1}$. Therefore every $\Z$-linear combination of the binomial coefficients ${p^m\choose 1}, {p^m\choose 2},\dots,{p^m\choose i}$ is divisible by $p^{m-d}$ where $d=\lfloor \log_p i \rfloor$. So each of the powers $c^e$ in Theorem~\ref{thm:HP} lies in $\gamma_i(G)^{p^{m-d}}$. Therefore each $c^e$ lies in $\gamma_{p^d}(G)^{p^{m-d}}$, since $p^d\leq i$. We might have $d=0$, but in this case  $c^e$ actually lies in $\gamma_2(G)^{p^m}$, since $i\geq 3$ and so $\gamma_i(G)^{p^m}\subseteq\gamma_2(G)^{p^m}$.
Similarly, we have that if $p=2$ and $d=1$, then $c_e\in\gamma_3(G)^{2^{m-1}}$.
Therefore the corollary follows from Theorem~\ref{thm:HP}.
\end{proof}

We note some consequences of this corollary:

\begin{corollary}\label{cor:HP, 2}
Let $m\geq 1$. Then
$$(xy)^{2^m} \equiv x^{2^m}\, y^{2^m}\, [x,y]^{2^{m-1}} \mod \gamma^2_{m+2}(G),$$
and if $p$ is odd, then
$$(xy)^{p^m} \equiv x^{p^m}\, y^{p^m} \mod \gamma^p_{m+2}(G).$$
\end{corollary}

\begin{proof}
We have $\gamma_2(G)^{p^{m}}\leq \gamma^p_{m+2}(G)$ as well as
$\gamma_3(G)^{2^{m-1}}\leq\gamma^2_{m+2}(G)$
and $$\gamma_{p^d}(G)^{p^{m-d}}\leq \gamma^p_{p^d+m-d}(G)\qquad\text{for $d=1,\dots,m$.}$$
If $p\geq 3$ and $d\geq 1$,  or if $p=2$ and $d\geq 2$,
then we have $p^d+m-d\geq m+2$ and thus $\gamma_{p^d}(G)^{p^{m-d}}\leq \gamma^p_{m+2}(G)$. Hence the corollary follows from Corollary~\ref{cor:HP}.
\end{proof}

If $G$ is abelian (written additively), then $L^p_n(G)=p^{n-1}G/p^n G$ for every $n\geq 1$, and multiplication by $p$ on $G$ induces a group morphism $L^p_n(G)\to L^p_{n+1}(G)$. (If in addition $G$ does not have $p$-torsion, then for every $n\geq 1$, this morphism is actually an isomorphism.)
More generally, we have:

\begin{corollary}\label{cor:phi_n}
Let $\mathbf G=\{G_n\}$ be a central $p$-filtration of $G$, and let $n\geq 1$.
Suppose $p>2$ or $[G_n,G_n]\leq G_{n+2}$. Then for every $m\geq 1$ the map
\begin{equation}\label{eq:phi_n}
x\mapsto x^{p^m}\, G_{n+m+1}\colon G_n \to G_{n+m}/G_{n+m+1}
\end{equation}
is a group morphism \textup{(}whose kernel contains $G_{n+1}$\textup{)}.
\end{corollary}

This follows from Corollary~\ref{cor:HP, 2} (applied to $G_n$ in place of $G$)
and Lemma~\ref{lem:iterated gamma}.
For $\mathbf G=\gamma^p(G)$, the previous corollary shows that
unless $p=2$, $n=1$, and $\gamma_2(G)\not\leq\gamma^2_3(G)$, the map in \eqref{eq:phi_n} induces group morphisms
$$\Phi_{n,m}\colon L^p_n(G)\to L^p_{n+m}(G) \qquad (m\geq 1).$$
In general, the morphisms $\Phi_{n,m}$ are not injective. However, we do have:

\begin{lemma}\label{lem:phin injective}
Suppose $p>2$, $n>1$, or $\gamma^2_3(G)\geq\gamma_2(G)$. Also assume that for $i=1,\dots,n+1$, the group $G/\gamma_{i}(G)$ is $p$-torsion free \textup{(}or equivalently, the sections $\gamma_i(G)/\gamma_{i+1}(G)$ are $p$-torsion free for $i=1,\dots,n$\textup{)}.
Then the group morphisms~$\Phi_{n,m}$~\textup{(}$m\geq 1$\textup{)} are injective.
\end{lemma}

This lemma is a consequence of the following fact, elucidating the structure of the elementary abelian $p$-groups $L^p_n(G)$. We set $$\Gamma_n := \gamma_n(G)/\gamma_n(G)^p\gamma_{n+1}(G)$$
for $n\geq 1$.

\begin{theorem}
For every $n\geq 1$ the map
\begin{align*}
\Gamma_1\times\cdots\times\Gamma_n & \to L^p_n(G)=\gamma^p_n(G)/\gamma^p_{n+1}(G) \\
(\ol{g_1},\dots,\ol{g_n}) &\mapsto g_1^{p^{n-1}}g_2^{p^{n-2}}\cdots g_n\cdot \gamma^p_{n+1}(G)
\end{align*}
where $g_i\in\gamma_i(G)$ and $\ol{g_i}=g_i\gamma_i(G)^p\gamma_{i+1}(G)$, is a surjection, and a group morphism unless $p$ is odd and $n=1$.
If for $i=1,\dots,n+1$, the group $G/\gamma_{i}(G)$ is $p$-torsion free, then
this map is also injective.
\end{theorem}

For a proof of this theorem see \cite[Chapter~VIII, Theorems~1.8 and 1.9]{HB82}. (Theorem~1.9 of loc.~cit.~has the more restrictive assumption that $G/\gamma_i(G)$ is torsion free for $i=1,\dots,n+1$, but inspection of the proof shows that only $p$-torsion needs to be excluded.) Note that this theorem in particular implies that if $G_{\operatorname{ab}}$ is finitely generated, then for each central $p$-filtration $\mathbf G$ of $G$, all layers $L_n(\mathbf G)$ are finite.

\medskip

In the situation of Lemma~\ref{lem:phin injective}, we now have a commutative diagram
$$\xymatrix{
\Gamma_1\times\cdots\times\Gamma_n \ar[d] \ar[r] & L^p_n(G) \ar[d]_{\Phi_{n,m}} \\
\Gamma_1\times\cdots\times\Gamma_n \times\Gamma_{n+1}\times\cdots\times\Gamma_m \ar[r] & L^p_{n+m}(G)
}$$
where the top horizontal map is a bijection, and the left vertical map is given by
$$(\ol{g_1},\dots,\ol{g_n})\mapsto (\ol{g_1},\dots,\ol{g_n},1,\dots,1)$$
and hence injective.
This yields Lemma~\ref{lem:phin injective}.

\begin{example}
Suppose that $p$ is odd or $\gamma_2(G)\leq\gamma^2_3(G)$. Then for each $m\geq 1$ we have a group morphism
$$\Phi_m\colon L^p_1(G)=H_1(G;\F_p)\to L^p_{m+1}(G)$$
which is given by
$$x\, \gamma^p_2(G)\mapsto x^{p^{m}}\, \gamma^p_{m+2}(G).$$
If the abelian group $H_1(G;\Z)=G_{\operatorname{ab}}=G/\gamma_2(G)$ does not have $p$-torsion, then for each $m\geq 1$, the morphism $\Phi_m$ is injective.
\end{example}

\subsection{$p$-potent filtrations}\label{sec:p-potent}
Let now $\mathbf G=\{G_n\}$ be an arbitrary central $p$-filtration of $G$.
In Corollary~\ref{cor:phi_n} we saw that if $p$ is odd or  $[G,G]\leq G_3$, then for each $n\geq 1$, exponentiation to the power $p^{n}$ on $G$ induces a group morphism $L_1(\mathbf G)\to L_{n+1}(\mathbf G)$, and above we gave a criterion, in the case $\mathbf G=\gamma^p(G)$, for these morphisms to be injective.
The following definition captures this phenomenon:

\begin{definition}
We say that $\mathbf G$ is {\it $p$-potent}\/ if for every $n\geq 1$, the map
$$x\mapsto x^{p^{n}}G_{n+2}\colon G\to G_{n+1}/G_{n+2}$$ is
a group morphism whose kernel \emph{equals} $G_2$.
In this case, we have an induced injective group morphism
$L_1(\mathbf G) \to L_{n+1}(\mathbf G)$, which we denote
by $\Phi_n$.
\end{definition}

It will be useful to have terminology for a few variants of this definition available:
We say that $\mathbf G$ is {\it strongly $p$-potent}\/ if for every $n\geq 1$, the map
$$x\mapsto x^p G_{n+2}\colon G_n\to G_{n+1}/G_{n+2}$$ is
a morphism with kernel $G_{n+1}$. In this case we have induced injective morphisms $\phi_n\colon L_n(\mathbf G)\to L_{n+1}(\mathbf G)$.
We call $\mathbf G$ \emph{uniformly $p$-potent} if $\mathbf G$ is strongly $p$-potent and each $\phi_n$ ($n\geq 1$) is an isomorphism. We simply say that $G$ is {\it $p$-potent} if $\gamma^p(G)$ is $p$-potent.  (Similarly for strongly $p$-potent respectively uniformly $p$-potent.) Clearly
uniformly $p$-potent $\Rightarrow$ strongly $p$-potent $\Rightarrow$ $p$-potent.

\index{filtration!$p$-potent}
\index{filtration!strongly $p$-potent}
\index{filtration!uniformly $p$-potent}

\begin{examples}\mbox{}

\begin{enumerate}
\item  If $G$ is abelian $p$-torsion free, then $G$ is uniformly $p$-potent.
\item  Every  uniformly powerful pro-$p$ group is (almost by definition) uniformly $p$-potent; see \cite[Chapter~4]{DdSMS}. (This explains our choice of terminology.) Examples may be obtained by considering the congruence subgroups of $\GL(d,\Z_p)$. For simplicity assume $p$ is odd, and for each~$n\geq 1$ define
$$\GL^n(d,\Z_p) := \ker\big(\GL(d,\Z_p)\to\GL(d,\Z_p/p^n\Z_p)\big).$$
Then the groups $\GL^n(d,\Z_p)$ form a uniformly $p$-potent separating filtration of the pro-$p$ group $G=\GL^1(d,\Z_p)$, and in fact $\gamma^p_n(G) = \GL^n(d,\Z_p)$ for every $n\geq 1$; see \cite[Theorem~5.2]{DdSMS}, and also Section~\ref{sec:csg} below.
\end{enumerate}
\end{examples}

We say that a central $p$-filtration $\mathbf G$ of $G$ has finite rank if there exists an integer~$r$ such that $\dim_{\F_p} L_n(\mathbf G) \leq r$ for all $n\geq 1$.

\begin{remarks}\mbox{}

\begin{enumerate}
\item If $G$ admits a separating strongly $p$-potent filtration, then $G$ is torsion~free.
\item If $\mathbf G=\{G_n\}$ is a strongly $p$-potent filtration of $G$ having finite rank, then for some $m\geq 1$, the filtration $\{G_{m+n-1}\}_{n\geq 1}$ of $G_m$ is uniformly $p$-potent.
\end{enumerate}
\end{remarks}

Our chief interest in uniformly $p$-potent filtrations is in the case of example~(1) above. For sake of completeness, we note:

\begin{lemma}\label{lem:linearity}
Suppose $G$ is finitely generated. If $G$ admits a separating strongly $p$-potent filtration of finite rank, then $G$ is necessarily linear \textup{(}over $\Z_p$\textup{)}. Conversely, if $G$ is linear, then for all but finitely many $p$, some finite index normal subgroup of $G$ admits a separating uniformly $p$-potent filtration.
\end{lemma}
\begin{proof}
This is essentially Lubotzky's linearity criterion \cite{Lub88} (see also Interlude~B in \cite{DdSMS}): Suppose $\mathbf G$ is a separating strongly $p$-potent filtration of finite rank; then $\mathbf G$ is a $p$-congruence system of finite rank in the sense of \cite[Definition~B.2]{DdSMS}, hence $G$ is linear over $\Z_p$ by Proposition~B.3 of loc.~cit. Conversely, suppose $G$ is linear. Then for all but finitely many $p$, the group $G$ admits an embedding into $\GL(d,\Z_p)$ for some $d$ \cite[Lemma~B.4]{DdSMS}. Assuming now that $p$ is odd and $G\leq\GL(d,\Z_p)$, the filtration $\mathbf G=\{G_n\}$ of $G$ given by $G_n := G\cap\GL^n(d,\Z_p)$ is separating and, as filtration of $G_1$, strongly $p$-potent of finite rank. Now use the preceding Remark~(2).
\end{proof}

\begin{remark}
Every normal filtration $\mathbf G=\{G_n\}$ of $G$ gives rise to a topology on $G$ making $G$ into a topological group with a fundamental system of neighborhoods of the identity given by the groups $G_n$. We denote by ${\widehat G}_{\mathbf G}$ the
 associated completion of $G$, that is, the topological group
$${\widehat G}_{\mathbf G} = \lim_{\longleftarrow} G/G_n.$$
We have a natural continuous morphism $G\to\widehat{G}_{\mathbf G}$, which is injective precisely if $\mathbf G$ is separating. If all $G_n$ are of finite index in $G$, then ${\widehat G}_{\mathbf G}$ is a profinite group. In particular, if $G_{\operatorname{ab}}$ is finitely generated and $\mathbf G=\gamma^p(G)$, then
we obtain the pro-$p$ topology on $G$, and $\widehat{G}_{\mathbf G}$ is the pro-$p$ completion of $G$. The proof of  \cite[Proposition~B.3]{DdSMS} shows that
if $\mathbf G$ is a strongly $p$-potent filtration of $G$ having finite rank, then ${\widehat G}_{\mathbf G}$ is isomorphic (qua topological group) to a $p$-adic analytic group. (This was originally shown in \cite[Th\'eor\`eme~3.1.7]{Laz65}.)
\end{remark}

\index{group!$p$-adic analytic}
\index{filtration!completion}
\index{pro-$p$ completion}

If $G_{\operatorname{ab}}$ is $p$-torsion free, then $G$
 is $p$-potent for all odd primes $p$ (by Corollary~\ref{cor:phi_n}), but not in general for $p=2$. However, assuming that both $G_{\operatorname{ab}}=G/\gamma_2(G)$ and $\gamma_2(G)/\gamma_3(G)$ are $2$-torsion free ensures that $G$ always has a fully characteristic  subgroup (of $2$-power-index if $G$ is finitely generated) carrying a $2$-potent filtration. More precisely, and more generally:

\begin{lemma}\label{lem:2-potent}
Let $G^*:=\gamma^p_m(G)$, where $m\geq 2$, and for $n\geq 1$ define $G^*_n:=\gamma^p_{n+m-1}(G)$. Then
$\mathbf G^*=\{G^*_n\}$ is a strongly central $p$-filtration of $G^*$, and
for every $n\geq 1$, the map $x\mapsto x^p$ induces a group morphism
$L_n(\mathbf G^*)\to L_{n+1}(\mathbf G^*)$.
If $G/\gamma_n(G)$ is $p$-torsion free for $n=1,\dots,m+1$, then $\mathbf G^*$ is $p$-potent. If $G/\gamma_n(G)$ is $p$-torsion free for
every $n\geq 1$, then  $\mathbf G^*$ is strongly $p$-potent.
\end{lemma}
\begin{proof}
It is clear that $\mathbf G^*$ is a strongly central $p$-filtration of $G^*$ (cf.~Proposition~\ref{prop:lower p}).
By Corollary~\ref{cor:phi_n}, $x\mapsto x^p$ induces a group morphism
$L_n(\mathbf G^*)\to L_{n+1}(\mathbf G^*)$. The rest follows from
Lemma~\ref{lem:phin injective}.
\end{proof}

We are mainly interested in this lemma when $G$ is free, in which case $G/\gamma_n(G)$ actually is torsion free for every $n\geq 1$. We note that the assumption that $G/\gamma_n(G)$ is torsion free for every $n\geq 1$ is also satisfied in other interesting situations, e.g., for primitive $1$-relator groups \cite{Lab70}, such as surface groups,  and primitive link groups \cite[Theorems~2 and 1]{Lab90}.

\section{Graphs of groups}\label{sec:Graphs of Groups}

\noindent
In this section we summarize some basic definitions and facts concerning graphs of groups and their fundamental groups. We refer to \cite{Ba93, Co89, Se80} for missing details.

\subsection{Graphs}
A {\it graph}\/ $Y$ consists of a set $V=V(Y)$ of {\it vertices}\/ and a set $E=E(Y)$ of {\it edges}, and two maps $E\to V\times V\colon e\mapsto (o(e),t(e))$ and $E\to E\colon e\mapsto\ol{e}$, subject to the following condition: for each $e\in E$ we have $\ol{\ol{e}}=e$, $\ol{e}\neq e$, and $o(e)=t(\ol{e})$.
We call $o(e)=t(\ol{e})$ and  $t(e)=o(\ol{e})$ the {\it origin}\/ respectively the {\it terminus}\/ of $e$. The edge $\ol{e}$ is called the {\it inverse edge}\/ of $e$. We call the set $\{e,\ol{e}\}$ a \emph{topological edge} of $Y$.

\index{graph}
\index{topological edge}

\medskip

Given a graph $Y$ as above, a {\it subgraph of $Y$}\/ is a graph $Y'$ whose vertex and edge sets $V(Y')$ and $E(Y')$ are subsets of $V$ respectively $E$ such that $o(e),t(e)\in V(Y')$ and $\ol{e}\in E(Y')$ for all $e\in E(Y')$, and such that the maps $o$, $t$ and $\ol{\,\cdot\,}$ on $Y'$ are the restrictions of the corresponding maps from $E$ to $E(Y')$. 
If $v_0$ is a vertex of~$Y$ we denote by $Y\setminus\{v_0\}$ the subgraph of $Y$ with vertex set $V\setminus\{v_0\}$ and edge set
$\{e\in E:o(e),t(e)\neq v_0\}$.
Similarly, if $e_0\in E$ we denote by $Y\setminus\{e_0\}$ the subgraph of $Y$ with vertex set $V$ and edge set $E\setminus\{e_0,\ol{e_0}\}$.

\index{subgraph}

\medskip

Given graphs $Y$, $Y'$, a morphism $\phi\colon Y\to Y'$ consists of two maps $V(Y)\to V(V')$ and $E(Y)\to E(Y')$, both also denoted by $\phi$, with the property that $$\phi(\ol{e})=\ol{\phi(e)}, \quad \phi(o(e))=o(\phi(e)),\quad \phi(t(e))=t(\phi(e))\qquad\text{ for all $e\in E(Y)$.}$$
We say that a morphism $Y\to Y'$ is an embedding if the corresponding maps $V(Y)\to V(Y')$ and  $E(Y)\to E(Y')$ are injective. If $Y$ is a subgraph of $Y'$, then the natural inclusions $V(Y)\to V(Y')$ and $E(Y)\to E(Y')$ form an embedding $Y\to Y'$.

\index{morphism!of graphs}

\medskip

Let $Y$ be a graph, and let $E$ denote a subset of the edge set $E(Y)$ of  $Y$ which is closed under taking inverses.
An {\it orientation}\/ of $E$ is a subset $E_+$ of $E$ such that $E$ is the disjoint union of $E_+$ and $\ol{E_+}$. By an orientation of $Y$ we mean an orientation of $E(Y)$.
An {\it oriented graph}\/ is a graph together with one of its orientations. Given two sets $V$ and $E_+$ and a map $E_+\to V\times V\colon e\mapsto (o(e),t(e))$ there exists (up to isomorphism) a unique graph $Y$ with vertex set $V$ such that~$E_+$ is an orientation of $Y$ and such that the origin and terminus maps on $E_+$ are given by $o$ and $t$ respectively. (The set of edges of $Y$ is the disjoint union $E(Y)=E_+\cup \ol{E_+}$ where $\ol{E_+}$ is a copy of $E_+$.)
Note that given graphs $Y$ and $Y'$ and an orientation $E_+$ of $Y$, in order to specify a morphism $Y\to Y'$ it suffices to give maps
$V(Y)\to V(Y')$ and $E_+\to E(Y')$, both here denoted by $\phi$, such that $\phi(o(e))=o(\phi(e))$ and $\phi(t(e))=t(\phi(e))$ for all $e\in E_+$.

\index{graph!orientation}
\index{graph!oriented}

\medskip

Let $P_n$ denote the graph with vertex set $V(P_n)=\{0,\dots,n\}$ and orientation given by $E(P_n)_+=\{(0,1),(1,2),\dots,(n-1,n)\}$  with $o(i,i+1)=i$, $t(i,i+1)=i+1$. 
A {\it path}\/ of length $n$ in the graph $Y$ from a vertex $v\in V$ to a vertex $w\in V$ is a morphism $p\colon P_n\to Y$ such that $p(0)=v$ and $p(n)=w$.
Such a path $p$ is determined by the sequence $(e_1,\dots,e_n)$ of edges $e_i=p(i-1,i)$ with $o(e_1)=v$,
$t(e_i)=o(e_{i+1})$ for $i=1,\dots,n-1$, and $t(e_n)=w$; we also write $p=(e_1,\dots,e_n)$.
A {\it circuit}\/ of length $n$ in $Y$ (where $n>0$) is a path $c=(e_1,\dots,e_n)$ such that the vertices $t(e_1),\dots,t(e_n)$ are pairwise distinct, and $t(e_n)=o(e_1)$. A {\it tree}\/ is a graph which is connected (for all vertices $v$, $w$  there is a path from $v$ to $w$) and without circuits. If $Y$ is a tree, then for any pair $(v,w)$ of vertices of $Y$ there exists a unique path of shortest length from $v$ to $w$, called the {\it geodesic} from $v$ to $w$ in $Y$.

\index{tree}

\medskip

{\it In the rest of this paper, unless otherwise noted, we assume that all graphs are finite \textup{(}i.e., their vertex sets
 and edge sets are finite\textup{)} and connected.}

\subsection{The fundamental group of a graph of groups}\label{sec:pi1}
Let $Y$ be a graph. A \emph{graph $\mathcal G$ of groups based on $Y$} consists
of families $\{G_v\}_{v\in V(Y)}$ and $\{G_e\}_{e\in E(Y)}$ of groups satisfying $G_e=G_{\ol{e}}$ for every $e\in E(Y)$, together with a family $\{f_e\}_{e\in E(Y)}$ of group embeddings $f_e\colon G_e\to G_{t(e)}$ ($e\in E(Y)$).
(We also call $Y$ the \emph{underlying graph} of $\scrg$.)

\medskip

\index{graph of groups|textbf}
\index{graph!of groups|textbf}
\index{graph!underlying a graph of groups}

Let $\scrg$ be a graph of groups based on a graph $Y$. We recall the construction of the fundamental group $G=\pi_1(\scrg)$ of $\scrg$ from \cite[I.5.1]{Se80}. First, consider the group $\pi(\scrg)$ (the {\it path group} of $\scrg$) generated by the groups $G_v$ ($v\in V(Y)$) and the elements $e\in E(Y)$ subject to the relations
$$e f_e(g) \ol{e} = f_{\ol{e}}(g) \qquad (e\in E(Y), g\in G_e).$$
By a \emph{path} of length $n$ in $\scrg$ from a vertex $v\in V$ to a vertex $w\in V$ we mean a sequence
$$\gamma=(g_0, e_1, g_1, e_2, \dots, e_n, g_n),$$
where $(e_1,\dots,e_n)$ is a path of length $n$ in $Y$ from $v$ to $w$ and
where $g_i\in G_{v_i}$ for $i=0,\dots,n$, with $v_0:=v$ and $v_i:=t(e_i)$ for $i>0$.
If $g_i=1$ for every $i$ then $\gamma$ is called an \emph{edge path}.
We say that the path $\gamma$ {\it represents}\/ the element
$$g=g_0 e_1 g_1 e_2 \cdots e_n g_n$$
of $\pi(\scrg)$.
Such a path $\gamma$ in $\scrg$ is {\it reduced}\/ if it satisfies the following conditions:
\begin{enumerate}
\item if $n=0$ then $g_0\neq 1$;
\item if $n>0$ then $g_i\notin f_{e_i}(G_{e_i})$ for each index $i$ such that $e_{i+1}=\overline{e_i}$.
\end{enumerate}
It is easy to see that every $g\neq 1$ in $\pi(\scrg)$ is represented by a reduced path in $\scrg$ (namely, a path of minimal length representing $g$). Conversely,
if $g$ is represented by a reduced path in $\scrg$, then $g\neq 1$ \cite[I.5.2, Theorem~11]{Se80}. In particular, the natural morphisms $G_v\to\pi(\scrg)$ are injective \cite[I.5.2, Corollary~1]{Se80}.

\index{path group}
\index{graph of groups!path group}
\index{graph of groups!fundamental group}

\medskip

Let now $v_0$ be a fixed vertex of $Y$. The  fundamental group $\pi_1(\scrg,v_0)$ of $\scrg$ (with base point $v_0$) is defined to be the subgroup of $\pi(\scrg)$ consisting of the elements represented by paths in $\scrg$ from $v_0$ back to itself.
If $v_1\in V(Y)$ is another base point, and $g$ is an element of $\pi(\scrg)$ represented by a path from $v_0$ to $v_1$, then $\pi_1(\scrg,v_0)\to\pi_1(\scrg,v_1)\colon t\mapsto g^{-1}tg$ is an isomorphism; by abuse of notation we write~$\pi_1(\scrg)$ to denote $\pi_1(\scrg,v_0)$ if the particular choice of base point is irrelevant. Given a subgraph $Y'$ of $Y$, we denote by $\scrg|Y'$ the restriction of the graph of groups~$\scrg$ to $Y'$, defined in the obvious way; given a base point $v_0\in V(Y')$, the natural inclusion $\pi(\scrg|Y')\to\pi(\scrg)$ induces an injective morphism $\pi_1(\scrg|Y',v_0)\to\pi_1(\scrg,v_0)$
\cite[Corollary~1.14]{Ba93}.

\medskip

An alternative description of $\pi_1(\scrg,v_0)$, which is often useful, is as follows: Let~$T$ be a maximal subtree of $Y$. The fundamental group $\pi_1(\scrg,T)$ of $\scrg$ relative to $T$ is defined by contracting $T$ to a point:
\begin{equation}\label{eq:maximal subtree description of pi1}
\pi_1(\scrg,T)  := \pi(\scrg)/\text{(relations $e=1$ for all $e\in E(T)$)}.
\end{equation}
The natural projection $\pi(\scrg) \to \pi_1(\scrg,T)$ restricts to an isomorphism $\pi_1(\scrg,v_0)\to \pi_1(\scrg,T)$, cf.~\cite[I.5.1, Proposition~20]{Se80}.
Identifying $\pi_1(\scrg,v_0)$ with $\pi_1(\scrg,T)$ via this isomorphism, for each $v\in V$ we obtain a group morphism
$G_v\to \pi_1(\scrg,v_0)$. This group morphism is injective  and given by $t\mapsto g_v t g_v^{-1}$ where $g_v\in\pi(\scrg)$ is represented by the reduced edge path in $\scrg|T$ from $v_0$ to $v$.
We often tacitly identify the vertex groups $G_v$ with subgroups of $\pi_1(\scrg,v_0)$ in this way.
This depends on the choice of maximal subtree $T$ of $Y$, but for a different choice of $T$ we obtain a conjugate subgroup of $\pi_1(\scrg,v_0)$, so this dependence is generally harmless.

\medskip

For every $e\in E(Y)$ let
$$\varphi_e=f_{\ol{e}}\circ f_e^{-1}\colon A_e:=f_e(G_e) \to B_e:=f_{\ol{e}}(G_{\ol{e}}),$$
and let $E:=E(Y)\setminus E(T)$. Fix an orientation $E_+$ of $E$.
The identification of $G:=\pi_1(\scrg,v_0)$ with $\pi_1(\scrg,T)$ means that $G$ may be described as the iterated HNN extension
$$G = \big\langle \pi_1(\scrg|T),\ e\in E_+: \text{$e a e^{-1} = \varphi_e(a)$ for all $e\in E_+$, $a\in A_e$}\big\rangle$$
of $\pi_1(\scrg|T)$, where the latter group is itself an iterated amalgamated product
(of the vertex groups of $\scrg|T$, with identifications
$a=\varphi_e(a)$ for $e\in E(T)$, $a\in A_e$).

\index{amalgamated product}
\index{HNN extension}

\subsection{Graphs of based CW-complexes}\label{sec:graphs of cw-complexes}
Let $Y$ be a graph, $V=V(Y)$, $E=E(Y)$.
A \emph{graph of based CW-complexes} $\scrx$ with underlying graph $Y$ consists of a collection of connected CW-complexes
$\{X_v\}_{v\in V}$ and $\{X_e \}_{e\in E}$
such that $X_{\ol{e}}=X_e$ for every $e\in E$ together with a collection
$\{ f_e\colon X_e\to X_{t(e)}\}_{e\in E}$ of cellular embeddings.
Furthermore we require  a choice of base points for each
$X_e$ ($e\in E$)
and $X_v$ ($v\in V$), with $X_e$ and $X_{\ol{e}}$ being assigned the same base point,
and for each $e\in E$ a choice of path in $X_{t(e)}$ connecting the base points of $f_e(X_e)$ and $X_{t(e)}$.
Finally we demand  that the following hold:
 \bn
 \item For each $e\in E$ the map $f_e$ induces an embedding $\pi_1(X_e)\to \pi_1(X_{t(e)})$ of groups (henceforth also denoted by $f_e$);
 \item given $v\in V$ the images of the maps $f_e\colon X_e\to X_{t(e)}$ with $t(e)=v$ are disjoint; and
 \item given $v\in V$ the set $X_v\sm \bigcup_{t(e)=v} f_e(X_e)$ is non-empty.
 \en
(Conditions (2) and (3) can always be achieved by attaching products, without changing the homotopy types of the complexes $X_v$ and $X_e$).

\index{graph of based CW-complexes}
\index{graph!of based CW-complexes}

\medskip

Let $\scrx$ be a graph of based CW-complexes  with underlying graph $Y$.
The \emph{topological realization} of $\scrx$ is  the following CW-complex:
  \[ |\scrx| =\left( \bigcup\limits_{v\in V} X_v \, \cup \,  \bigcup\limits_{e\in E} X_e\times [-1,1] \right) \bigg/ \sim \qquad \text{(disjoint union)}\]
where for $x\in X_e$ we set $(x,1)\sim f_e(x)\in X_{t(e)}$ and for $(x,t)\in X_e\times [-1,1]$ we set $(x,t)\sim (x,-t)\in X_{\ol{e}}\times [-1,1]$.
Note that we have canonical inclusion maps of $X_v$ and $X_e\times 0$ into $|\scrx|$.
We identify $X_v$ and $X_e=X_e\times 0$ with their images in $|\scrx|$.

The collections $\{\pi_1(X_v)\}_{v\in V}$ and $\{\pi_1(X_e)\}_{e\in E}$ together with the group embeddings $f_e\colon\pi_1(X_e)\to \pi_1(X_{t(e)})$ form a graph of groups based on $Y$, whose fundamental group (as defined above) is naturally isomorphic to
the fundamental group  of the topological space $X:=|\scrx|$; see \cite[Theorem~2.1]{He87}.
Conversely, every graph of groups $\scrg$  based on $Y$ gives rise to a graph $\scrx$ of based CW-complexes with underlying graph $Y$ such that $\pi_1(X_v)=G_v$, $\pi_1(X_e)=G_e$, and with the embeddings $X_e\to X_{t(e)}$ inducing the edge morphisms $f_e\colon G_e\to G_{t(e)}$ of $\scrg$. (By choosing for each $X_v$ and $X_e$ classifying spaces $BG_v$ respectively $BG_e$.)

\subsection{Morphisms of graphs of groups}\label{sec:Morphisms of graphs of groups}
Let now $\scrg$ and $\scrg'$ be graphs of groups, based on the graphs $Y$ respectively $Y'$. A {\it morphism} $\phi\colon \scrg\to\scrg'$ of graphs of groups consists of the following data:
a morphism $Y\to Y'$ of graphs, also denoted by $\phi$;
families of group morphisms
$$\{\phi_v \colon G_v\to G'_{\phi(v)}\}_{v\in V(Y)},\qquad
\{\phi_e \colon G_e\to G'_{\phi(e)}\}_{e\in E(Y)};$$
and families of elements $\{g_v\}_{v\in V(Y)}$ and $\{g_e\}_{e\in E(Y)}$ of $\pi(\scrg')$, satisfying the following conditions, with $v$ ranging over $V(Y)$ and $e$ over $E(Y)$:
\begin{enumerate}
\item  $\phi_{e}=\phi_{\ol{e}}$;
\item $g_v\in \pi_1(\scrg',\phi(v))$;
\item for $v=t(e)$ we have $\delta_e:=g_v^{-1} g_e \in G'_{\phi(v)}$, and
\item the diagram
\begin{equation}\label{eq:morphism of graphs of groups}\begin{split}
\xymatrixcolsep{5pc}\xymatrix{
G'_{\phi(e)} \ar[r]^{\ad(\delta_e)\circ f'_{\phi(e)}} & G'_{\phi(v)} \\
G_e \ar[u]^{\phi_e} \ar[r]^{f_e} & G_{v} \ar[u]^{\phi_{v}}
}\end{split}
\end{equation}
commutes. Here, given a group $G$ and $h\in G$, $\ad(h)$ denotes the inner automorphism $g\mapsto h g h^{-1}$ of $G$.
\end{enumerate}
Every such morphism  of graphs of groups as above
induces a group morphism $\phi_*\colon\pi(\scrg)\to\pi(\scrg')$ between path groups satisfying
$$\phi_*(g) = g_v\phi_v(g)g_v^{-1}, \quad \phi_*(e) = g_e\phi(e)g_{\ol{e}}^{-1}
\qquad\text{for all $v\in V(Y)$, $g\in G_v$, $e\in E(Y)$.}$$
Given a base point $v_0\in V(Y)$ for $G=\pi_1(\scrg,v_0)$ and taking the base point $v_0'=\phi(v_0)$ for $G'=\pi_1(\scrg',v_0')$, this morphism restricts to a morphism $G\to G'$. (Cf.~\cite[Proposition~2.4]{Ba93}.)
Given an integer $d\geq 0$ we say that $\phi$ has degree $d$ if $[G':\phi_*(G)]=d$.
Every finite index subgroup of the fundamental group of a graph of groups arises in this way:

\index{morphism!of graphs of groups}
\index{morphism!of finite degree}


\begin{theorem}\label{thm:bass-serre}
Let $\scrg$ be a graph of groups with underlying graph $Y$, and let~$H$ be a
 finite index subgroup of $G=\pi_1(\scrg)$. Then there exists a
graph of groups $\scrh$ and a morphism $\phi\colon\scrh\to\scrg$ such that $\phi_*\colon\pi_1(\scrh)\to G$ is injective with image $H$.
\end{theorem}

This may be seen as a consequence of the main structure theorem of the Bass-Serre theory of groups acting on trees (cf.~\cite[Section~4]{Ba93}). Alternatively, one may employ the topological interpretation of $\pi_1(\scrg)$ sketched in Section~\ref{sec:Morphisms of graphs of groups} (cf.~\cite{SW79}):  if $\scrx$ is a graph of based CW-complexes based on a graph $Y$, with topological realization $X$, and $\pi\colon\widetilde{X}\to X$ is a finite covering map, then $\widetilde{X}$ is in a natural way the topological realization of a graph of based CW-complexes $\widetilde{\scrx}$: The vertex spaces of $\widetilde{\scrx}$ are given by the components of $\pi^{-1}(X_v)$, $v\in V(Y)$, and the edge spaces of $\widetilde{\scrx}$ are given by the components of $\pi^{-1}(X_e)$, $e\in E(Y)$.
Let $\widetilde{Y}$ be the underlying graph of $\widetilde{\scrx}$; the covering map $\pi\colon \widetilde{\scrx}\to \scrx$ induces a morphism of graphs $\widetilde{Y}\to Y$ which we also denote by $\pi$.
The edge spaces and vertex spaces of $\widetilde{\scrx}$ are then endowed with base points covering the base points of $\scrx$; and for $\widetilde{e}\in E(\widetilde{Y})$, $e=\pi(\widetilde{e})$, we pick a path connecting the base points of $f_{\widetilde{e}}(\widetilde{X}_{\widetilde{e}})$ and $\widetilde{X}_{t(\widetilde{e})}$. This data gives rise to a morphism $\phi\colon\widetilde\scrg\to\scrg$ between the graphs of groups
$\widetilde\scrg$ and $\scrg$ associated to $\widetilde\scrx$ respectively $\scrx$ such that $\phi_*$ agrees with the group morphism $\pi_1(\widetilde\scrx)\to\pi_1(\scrx)$ induced by $\pi$.

\begin{remark}
In the context of the previous theorem, one can choose the vertex groups of $\scrh$ to be $H\cap gG_vg^{-1}$, where $v$ ranges over the vertices of $Y$ and $g$ over a suitable set of $(H,G_v)$-double coset representatives, and the edge groups to be $H\cap gf_e(G_e)g^{-1}$, where $e$ ranges over the edges of $Y$ and $g$ over a suitable set of $(H,f_e(G_e))$-double coset representatives \cite[Section~8.5, Theorem~27]{Co89}.
In particular,
if $H$ is a \emph{normal} subgroup of $G$, then one can choose $\scrh$ such that its vertex groups are of the form $H\cap G_v$, and its edge groups of the form $H\cap f_e(G_e)$.
\end{remark}

\subsection{Compatible collections of subgroups}
Let $\scrg$, $\scrg'$  be graphs of groups based on the graphs $Y$, $Y'$, respectively. Then $\scrg'$ is a \emph{subgraph of subgroups} of $\scrg$ if
\begin{enumerate}
\item $Y'$ is a subgraph of $Y$;
\item $G'_v\leq G_v$ for each $v\in V(Y')$;
\item $G_e'\leq G_e$ and $f_e'=f_e|G_e'$ for each $e\in E(Y')$;
\item $f_e(G_e)\cap G_v'=f_e(G_e')$ for all $e\in E(Y')$, $v=t(e)$.
\end{enumerate}
In this case, the natural inclusions $G'_v\to G_v$ ($v\in V(Y')$) and $G'_e\to G_e$ ($e\in E(Y')$) form a morphism $\scrg'\to\scrg$ of graphs of groups, and given $v_0\in V(Y')$, the induced morphism $\pi_1(\scrg',v_0)\to\pi_1(\scrg,v_0)$ is injective \cite[Corollary~1.14]{Ba93}.
We say that a subgraph of subgroups $\scrg'$ of $\scrg$ is \emph{normal} if $G'_v\trianglelefteq G_v$ for each $v\in V(Y')$. (Then $G_e'\trianglelefteq G_e$ for each $e\in E(Y')$ by condition (4).)
For example, given a subgraph $Y'$ of $Y$, the graph of groups $\scrg|Y'$ is a normal subgraph of subgroups of $\scrg$.

\index{subgraph of subgroups}

\medskip

A subgraph of subgroups $\scrg'$ of $\scrg$ is called a \emph{graph of subgroups} of $\scrg$ if the underlying graph $Y'$ of $\scrg'$ is the same as the underlying graph $Y$ of $\scrg$. Let $\scrg'$ be a normal graph of subgroups of $\scrg$.
Then we obtain a graph of groups based on the graph $Y$, which we denote by $\scrg/\scrg'$, with vertex groups $G_v/G_v'$ ($v\in V(Y)$), edge groups $G_e/G_e'$ and edge morphisms $G_e/G_e'\to G_{t(e)}/G_{t(e)}'$ induced by $f_e$ $(e\in E(Y)$). The natural surjections $G_v\to G_v/G_v'$ and $G_e\to G_e/G_e'$ yield a morphism $\scrg\to\scrg/\scrg'$ of graphs of groups, which in turn yields a group morphism
$\pi_1(\scrg,v_0) \to \pi_1(\scrg/\scrg',v_0)$,
which we denote by $\pi_{\scrg'}$.

\index{graph!of subgroups}

\medskip

A graph of subgroups of $\scrg$ is completely determined by its family of vertex groups: Suppose we are given a family $\scrg'=\{G'_v\}_{v\in V(Y)}$ of subgroups $G'_v\leq G_v$ ($v\in V(Y)$) which is \emph{compatible} in the sense that
$$f_e^{-1}(G'_{t(e)}) = f_{\ol{e}}^{-1}(G'_{t(\ol{e})})\qquad\text{for all $e\in E(Y)$.}$$
Then there is a unique graph of subgroups of $\scrg$ with vertex groups $G'_v$, which we continue to denote by the same symbol $\scrg'$.
(In particular, we may write $G'_e=f_e^{-1}(G'_{t(e)})$.)

\index{compatible collection of subgroups}

\begin{example}
Let $H$ be a normal subgroup of $G=\pi_1(\scrg)$. Then the collection $\scrh=\{H\cap G_v\}_{v\in V(Y)}$ of normal subgroups is compatible; moreover, there is a unique morphism $\pi_1(\scrg/\scrh)\to G/H$ such that the diagram
\begin{equation}\label{eq:factorization}
\xymatrix{ &G \ar[dl]_{\pi_{\scrh}}\ar[dr]^{\text{natural projection}} \\ \pi_1(\scrg/\scrh)  \ar[rr]&& G/H}
\end{equation}
commutes.
\end{example}

\begin{proposition}\label{prop:commoncover}
Let $\{H_v\}_{v\in V(Y)}$
be  a compatible collection of finite index normal subgroups
of $\scrg$.
Then there exists a graph of groups $\widetilde\scrg$ based on a graph $\widetilde{Y}$ and a morphism $\phi\colon\widetilde\scrg\to\scrg$ with the following properties:
\begin{enumerate}
\item $\phi$ has finite degree;
\item for each $\widetilde{v}\in V(\widetilde{Y})$,  $\phi_{\widetilde{v}}$ is an isomorphism $\widetilde{G}_{\widetilde{v}}\xrightarrow{\cong} H_{\phi(\widetilde{v})}$;
\item for each $\widetilde{e}\in E(\widetilde{Y})$,  $\phi_{\widetilde{e}}$ is an isomorphism $\widetilde{G}_{\widetilde{e}}\xrightarrow{\cong} H_{\phi(\widetilde{e})}$.
\end{enumerate}
\end{proposition}

\begin{proof}
This is sketched in \cite[Theorem~2.2]{He87};
for the convenience of the reader, we give a proof.  We employ covering space theory. Thus, let $\scrx$ be a graph of based CW-complexes based on $Y$, with topological realization $X:=|\scrx|$, such that $\pi_1(X)=\pi_1(\scrg)$, as described in
Section~\ref{sec:graphs of cw-complexes}.
The  collections $\{H_v\}$ and $\{H_e\}$ of subgroups of the vertex respectively edge groups of $\scrg$ give
rise to collections
\[ \big\{ \widetilde{X}_v\xrightarrow{\pi_v} X_v\big\}_{v\in V(Y)} \text{ and } \big\{ \widetilde{X}_e\xrightarrow{\pi_e} X_e\big\}_{e\in E(Y)}\]
of coverings of the vertex respectively edge spaces of $\scrx$ which are also compatible in the following sense:
\bn
\item for each $e\in E(Y)$ we have $\pi_e=\pi_{\ol{e}}$, and
\item for each $e\in E(Y)$ and for each component $\widetilde{Z}$ of $\pi_{t(e)}^{-1}(f_e(X_{e}))$ there exists an embedding $\widetilde{f}_e\colon\widetilde{X}_e\to \widetilde{X}_{t(e)}$ with image $\widetilde{Z}$ which covers
$f_e\colon X_e\to X_{t(e)}$.
\en
We will show that the various covering spaces can be glued together to give a connected covering $\pi\colon\widetilde{X}\to X$ such that
\bn
\item for each $v\in V(Y)$ and each component $\widetilde{X}_v$ of  $\pi^{-1}(X_v)$ the covering $\widetilde{X}_v\xrightarrow{\pi} X_v$ corresponds to $H_v$, and
\item for each $e\in E(Y)$ and each component $\widetilde{X}_e$ of  $\pi^{-1}(X_e)$ the covering $\widetilde{X}_e\xrightarrow{\pi} X_e$ corresponds to $H_e$.
\en
Then every morphism of graph of groups $\widetilde{\scrg}\to\scrg$ associated to $\pi$ as in Section~\ref{sec:graphs of cw-complexes} has the required properties.
We proceed by induction on the number of topological edges of $Y$.

The case where $Y$ has no topological edge (and hence only one vertex) being obvious, let us first consider the case that the graph $Y$ has a single topological edge $\{e,\ol{e}\}$ and hence either one or two vertices.
First assume that  $Y$ has a single vertex $v$. In that case $\pi_v^{-1}(f_e(X_e))$ and $\pi_v^{-1}(f_{\ol{e}}(X_e))$ have the same number $n$ of components.
Matching up these components in pairs and connecting them via $n$ copies of $\widetilde{X}_e$,
the resulting complex $\widetilde{X}$ has the desired property.

Now assume $Y$ has two distinct vertices $v_1$ and $v_2$ such that $t(e)=v_1$ and $t(\ol{e})=v_2$.
Let $n_i$ be the number of components of $\pi_{v_i}^{-1}(f_e(X_e))$, for $i=1,2$.
 We get the cover $\widetilde{X}\to X$ by taking $\operatorname{lcm}(n_1,n_2)/n_1$ copies of $\widetilde{X}_{v_1}$ and $\operatorname{lcm}(n_1,n_2)/n_2$ copies of $\widetilde{X}_{v_2}$ and
by connecting each copy of a component of $\pi_{v_1}^{-1}(f_e(X_e))$ with a copy of
a component of $\pi_{v_2}^{-1}(f_{\ol{e}}(X_e))$ via a copy of $\widetilde{X}_e$.

We now turn to the case of graphs with more than one topological edge. Since~$Y$ is connected, either there is an edge $e$ such that the graph $Y \setminus\{e\}$ obtained by deleting the edge $e$ from $Y$ is connected, or there is an edge $e$ such that one of the components of
$Y \sm \{e\}$ consists of a single vertex with no topological edge.

First consider the case that we have an edge $e$ such that $Y'=Y \sm \{e\}$ is connected.
We denote by $\scrx'$ the subgraph of based CW-complexes of $\scrx$ with underlying graph $Y'$, and write $X'=|\scrx'|$. We identify $X'$ with a topological subspace of $X=|\scrx|$ in the natural way. By inductive hypothesis there exists a covering $\widetilde{X}'\to X'$ with the desired properties.
We now consider the graph of CW-complexes $\scrz$ whose underlying graph has a single vertex $w$ and a single topological edge $\{f,\ol{f}\}$, vertex space
$X_w=X'$, edge spaces $X_{f}=X_{\ol{f}}=X_e$, and edge maps $X_f\to X_w$, $X_{\ol{f}}\to X_w$ given by $f_e\colon X_e\to X_{t(e)}\subseteq X'$ respectively
$f_{\ol{e}}\colon X_{\ol{e}}\to X_{t(\ol{e})}\subseteq X'$. Writing $Z=|\scrz|$ we note that $Z=X$.
Clearly the finite index subgroup
$\pi_1(\widetilde{X}')$ of $\pi_1(X')=\pi_1(X_w)$ forms a compatible collection of subgroups (of the graphs of groups associated to $\scrz$).
Since we already showed that the conclusion of the proposition holds for graphs with one edge we now get
a cover $\widetilde{Z}\to Z=X$ which has the desired properties.

A similar argument deals with the case that we have an edge $e$ such that one component of $Y\sm \{e\}$ consists of a single vertex with no topological edge.
\end{proof}

\begin{remarks} \mbox{}

\begin{enumerate}
\item Assume that  the indices of the normal subgroups $H_v\trianglelefteq G_v$ are powers of~$p$. In that case the indices of the normal subgroups $H_e\trianglelefteq  G_e$ are powers of~$p$. With the notation introduced in the proof above we have
 \[  n=\frac{[G_v:H_v]}{[G_e:H_e]} \text{ and }   n_i=\frac{[G_{v_i}:H_{v_i}]}{[G_e:H_e]} \text{ for $i=1,2$.} \]
In particular $n$, $n_1$ and $n_2$ are powers of $p$. It is now clear that the degree of the cover $\pi\colon\widetilde{X}\to X$ (and hence the degree of the associated morphism $\widetilde{\scrg}\to\scrg$ of graphs of groups) constructed in the proof is a power of $p$.
\item There is no guarantee that the resulting cover constructed in the proof will be regular (i.e., that $\phi_*(\pi_1(\widetilde{\scrg}))$ will be a normal subgroup of $\pi_1(\scrg)$).
\end{enumerate}
\end{remarks}

\subsection{Filtrations of graphs of groups}
Let $\scrg$ be a graph of groups with underlying graph $Y$. We say that a collection $\mathbf G=\{\mathbf G_{v}\}_{v\in V(Y)}$ of
filtrations $\mathbf G_{v}=\{G_{v,n}\}_{n\geq 1}$ of $G_v$ (one for each vertex $v$) is a {\it filtration} of $\scrg$ if for each $n\geq 1$, the collection $\mathbf G_n:=\{G_{v,n}\}_{v\in V(Y)}$ of subgroups of the vertex groups is compatible.

\index{filtration!of a graph of groups}

Let $\mathbf G=\{\mathbf G_{v}\}_{v\in V(Y)}$ be a filtration of $\scrg$. Then
for each $e\in E(Y)$ we obtain a filtration $\mathbf G_e=\{G_{e,n}\}_{n\geq 1}$ of $G_e$ where $G_{e,n}:=f_e^{-1}(G_{t(e),n})$ for every $n\geq 1$. Note that $f_e$ is a morphism $(G_e,\mathbf G_e)\to (G_{t(e)},\mathbf G_{t(e)})$ of filtered groups, and $\mathbf G_e = \mathbf G_{\ol{e}}$ for all $e\in E(Y)$.

Suppose that $\mathbf G$ is normal (i.e., each $\mathbf G_v$ is normal). Then for each $n\geq 1$ we have three graphs of groups with same underlying graph $Y$ associated to $\scrg$:
\begin{enumerate}
\item The graph of subgroups $\mathbf G_n$ of $\scrg$;
\item the
quotient graph of groups $\scrg/\mathbf G_n$ associated to $\mathbf G_n$, with vertex groups $G_v/G_{v,n}$, edge groups $G_e/G_{e,n}$ and edge morphisms $$G_e/G_{e,n}\to G_{t(e)}/G_{t(e),n}$$ induced by $f_e$; and
\item the graph $L_n(\mathbf G):=\mathbf G_n/\mathbf G_{n+1}$, whose vertex and edge groups are the $n$th layers $L_n(\mathbf G_{v})$ of $\mathbf G_v$ respectively
$L_n(\mathbf G_e)$ of $\mathbf G_e$, and with the morphisms $L_n(f_e)\colon L_n(\mathbf G_e) \to L_n(\mathbf G_{v})$
induced by $f_e$ as the edge morphisms.
\end{enumerate}
\index{$L_n(\mathbf G)$}
We say that another filtration $\mathbf G^*=\{\mathbf G^*_v\}_{v\in V(Y)}$ of $\scrg$ is a {\it compatible stretching}\/ of $\mathbf G$ if the filtrations $\mathbf G_v^*$,  as $v$ ranges over $V(Y)$, form a compatible collection of stretchings of the $\mathbf G_v$ as defined in Section~\ref{subsec:filtrations} above.
Note that in this situation, for each $n\geq 1$, either every vertex group of $L_n(\mathbf G^*)$ is trivial, or $L_n(\mathbf G^*)=L_m(\mathbf G)$ for some $m\geq 1$.

Given a property $\mathfrak P$ of filtrations, we say that $\mathbf G$ has the property $\mathfrak P$ if $\mathbf G_v$ has $\mathfrak P$ for each $v\in V(Y)$. For example,
we say that $\mathbf G$ is separating if for each $v\in V(Y)$, the filtration $\mathbf G_v$ of $G_v$ is separating: $\bigcap_{n\geq 1} G_{v,n}=\{1\}$. We also say that {\it $\mathbf G$ separates the edge groups of $\scrg$}\/ if for each $e\in E(Y)$ and $v=t(e)$, the filtration $\mathbf G_v$ of $G_v$ separates $f_e(G_e)$: $\bigcap_{n\geq 1} G_{v,n} f_e(G_e) = f_e(G_e)$.

\index{filtration!separating}
\index{filtration!separates the edge groups}

\medskip

In the next lemmas, let $\phi\colon\scrg\to\scrg'$ be a morphism of graph of groups.

\begin{lemma} \label{lem:morphisms and filtrations, 1}
Let $\scrh=\{H_{v'}\}_{v'\in V(Y')}$ be a compatible collection of normal subgroups of $\scrg'$. Then $$\phi^{-1}(\scrh):=\big\{\phi_v^{-1}(H_{\phi(v)})\big\}_{v\in V(Y)}$$ is a compatible collection of normal subgroups of $\scrg$.
\end{lemma}

The proof of this lemma is a routine verification, using the diagram \eqref{eq:morphism of graphs of groups} (left to the reader).
In the rest of this section let $\mathbf G'$ be a normal filtration of $\scrg'$.
Then $$\phi^{-1}(\mathbf G'):=\{\phi_v^{-1}(\mathbf G'_{\phi(v)})\}_{v\in V(Y)}$$ is a normal filtration of $\scrg$, by Lemma~\ref{lem:morphisms and filtrations, 1}, and if $\mathbf G'$ is central, a $p$-filtration, or $p$-potent, then clearly $\phi^{-1}(\mathbf G')$ also has the respective property.

\begin{lemma}\label{lem:separating, 1}
Suppose each $\phi_v$ is injective. Then
\begin{enumerate}
\item if $\mathbf G'$ is separating, then $\phi^{-1}(\mathbf G')$ is separating;
\item if each $\phi_e$ is bijective and $\mathbf G'$ separates
the edge groups of $\scrg'$, then $\phi^{-1}(\mathbf G')$ separates the edge groups of $\scrg$;
\item if each $\phi_e$ maps $G_e$ isomorphically onto the first group $G'_{\phi(e),1}$ in the filtration $\mathbf G'_{\phi(e)}$ of $G'_{\phi(e)}$, and $\mathbf G'$ separates the edge groups of $\scrg'$, then $\phi^{-1}(\mathbf G')$ separates the edge groups of $\scrg$.
\end{enumerate}
\end{lemma}
\begin{proof}
If $\mathbf G'$ is separating, then
$$\bigcap_{n\geq 1}\, \phi_v^{-1}(\mathbf G'_{\phi(v),n}) = \phi_v^{-1}\left(\bigcap_{n\geq 1} \mathbf G'_{\phi(v),n}\right)=\phi_v^{-1}(\{1\})=\{1\}$$
for each $v$, hence $\phi^{-1}(\mathbf G')$ is separating.
For each $e$ and $v=t(e)$, denoting by $\phi_v^{-1}\colon \phi_v(G_v)\to G_v$ the inverse of $\phi_v$, by \eqref{eq:morphism of graphs of groups}
we also have
$$f_e(G_e) = \big(\phi_v^{-1}\circ\ad(\delta_e)\circ f'_{\phi(e)}\circ \phi_e\big)(G_e).$$
Assume now that each $\phi_e$ is bijective and $\mathbf G'$ separates the edge groups of $\scrg'$. Then for all $e$ and $v=t(e)$,
we have
\begin{align*}
f_e(G_e) &=
\big(\phi_v^{-1}\circ\ad(\delta_e)\big)(f'_{\phi(e)}(G'_{\phi(e)})) \\
&= \big(\phi_v^{-1}\circ\ad(\delta_e)\big)\left(\bigcap_n f'_{\phi(e)}(G'_{\phi(e)})\cdot \mathbf G'_{\phi(v),n}\right) \\
&=\bigcap_n  \big(\phi_v^{-1}\circ\ad(\delta_e)\big)(f'_{\phi(e)}(G'_{\phi(e)})) \cdot \phi_v^{-1}(\mathbf G'_{\phi(v),n})\\
&= \bigcap_n f_e(G_e)\cdot \phi_v^{-1}(\mathbf G'_{\phi(v),n}).
\end{align*}
Hence $\phi^{-1}(\mathbf G')$ separates the edge groups of $\scrg$. This shows (2). Part (3) is proved in a similar way.
\end{proof}

In the next lemmas we fix $e\in E(Y)$, and write $v=t(e)$, $e'=\phi(e)$, $v'=\phi(v)$.

\begin{lemma}\label{lem:separating, 2}
Suppose that for each $n\geq 1$, the morphism $\phi_e$ restricts to an isomorphism
$G_{e,n}\xrightarrow{\cong} G'_{e',n}$, and
$\mathbf G'_{v'}\cap f_{e'}'(G_{e'}')$ is a uniformly $p$-potent filtration of~$f_{e'}'(G_{e'}')$. Then
$\mathbf G_v\cap f_e(G_e)$ is a uniformly $p$-potent filtration of $f_e(G_e)$.
\end{lemma}
\begin{proof}
Note that the hypothesis that $\mathbf G'_{v'}\cap f_{e'}'(G_{e'}')$ be uniformly $p$-potent is equivalent to the
filtration $\mathbf G'_{e'}=(f_{e'}')^{-1}(\mathbf G'_{v'})$ of $G'_{e'}$ being uniformly $p$-potent, and that it is enough to show that ${\mathbf G}_e=f_e^{-1}(\mathbf G_v)$ is a uniformly $p$-potent filtration of $G_e$.
Now by assumption, $\phi_e$ induces isomorphisms
$G_{e,n}\to G'_{e',n}$, for each $n\geq 1$.
Hence for each $n\geq 1$, $x\mapsto x^p$ yields a surjective group morphism
$$G_{e,n}\to
G_{e,n+1}/G_{e,n+2}$$
whose kernel equals $G_{e,n+1}$, as required.
\end{proof}

The following two lemmas are not used until Section~\ref{sec:mod p homology graph}, and may be skipped upon first reading.

\begin{lemma}\label{lem:separating, 3}
Suppose $\phi_v$ is injective and $\phi_e$ is bijective, and $m\geq 1$ such that
$$G'_{v',n}\cap f_{e'}'(G_{e'}') = \gamma^p_n(f_{e'}'(G_{e'}')) \qquad\text{for $1\leq n\leq m$.}$$
Then
$$G_{v,n}\cap f_e(G_e) = \gamma^p_n(f_e(G_e))\qquad\text{for $1\leq n\leq m$.}$$
\end{lemma}

\begin{proof}
Note that $(\ad(\delta_e^{-1})\circ\phi_v)(\mathbf G_v)\subseteq \mathbf G'_{v'}$ and,  since $\phi_e$ is bijective:
$$(\ad(\delta_e^{-1})\circ\phi_v)(f_e(G_e)) =  f'_{e'}(G'_{e'}), \qquad \gamma^p(f'_{e'}(G'_{e'}))=f'_{e'}(\phi_e(\gamma^p(G_e))).$$
Hence for $1\leq n\leq m$:
$$(\ad(\delta_e^{-1})\circ\phi_v)(G_{v,n}\cap f_e(G_e)) \subseteq G'_{v',n}\cap f'_{e'}(G_{e'}')=f'_{e'}(\phi_e(\gamma^p_n(G_e))). $$
Since $\phi_v$ is injective, this yields $$G_{v,n}\cap f_e(G_e) \subseteq f_e(\gamma^p_n(G_e))=\gamma^p_n(f_e(G_e)).$$
The lemma follows, since $\gamma^p(f_e(G_e))$ is the fastest descending central $p$-filtration of~$f_e(G_e)$.
\end{proof}

\begin{lemma}\label{lem:separating, 4}
Suppose $G_e$ is abelian, $\phi_v$ is injective, $\phi_e$ maps $G_e$ isomorphically onto the first group $G'_{e',1}$ in the filtration $\mathbf G'_{e'}$ of $G'_{e'}$, and $\ell\geq 0$ such that
$$G'_{v',n}\cap f_{e'}'(G_{e'}') = \gamma^p_{n+\ell}(f_{e'}'(G_{e'}'))\qquad\text{for each $n\geq 1$.}$$ Then
$$G_{v,n}\cap f_e(G_e) = \gamma^p_n(f_e(G_e))\qquad\text{for each $n\geq 1$.}$$
\end{lemma}
\begin{proof}
As in the proof of the previous lemma, it suffices to show that
$$(\ad(\delta_e^{-1})\circ\phi_v)(G_{v,n}\cap f_e(G_e)) \subseteq f'_{e'}(\phi_e(\gamma^p_n(G_e)))\qquad\text{for all $n\geq 1$.}$$
Let $n\geq 1$. By hypothesis
$$(\ad(\delta_e^{-1})\circ\phi_v)(G_{v,n}\cap f_e(G_e)) \subseteq
G'_{v',n}\cap f'_{e'}(G'_{e'}) = \gamma^p_{n+\ell}(f'_{e'}(G'_{e'})).$$
Since $G'_{e'}$ is abelian and $\phi_e(G_e)=G'_{e',1}=\gamma^p_{1+\ell}(G'_{e'})$, we have
$$\gamma^p_{n+\ell}(G'_{e'})=\gamma^p_n(\gamma^p_{1+\ell}(G'_{e'}))=
\gamma^p_n(\phi_e(G_e))=\phi_e(\gamma^p_n(G_e))$$
and thus
\begin{multline*}
(\ad(\delta_e^{-1})\circ\phi_v)(G_{v,n}\cap f_e(G_e)) \subseteq \gamma^p_{n+\ell}(f'_{e'}(G'_{e'}))\\ =f'_{e'}(\gamma^p_{n+\ell}(G'_{e'})) =
f'_{e'}(\phi_e(\gamma^p_n(G_e))),
\end{multline*}
as required.
\end{proof}

\subsection{The Mayer-Vietoris sequence associated to a graph of groups}\label{sec:MV}
Let $\scrg$ be a graph of groups with fundamental group $G=\pi_1(\scrg)$. Let $E_+$ be an orientation of the underlying graph $Y$ of $\scrg$. Then for each $G$-module $M$, there is an exact sequence $$\cdots \to H_{n+1}(G;M) \to \bigoplus_{e\in E_+} H_n(G_e;M)\to \bigoplus_{v\in V(Y)} H_n(G_v;M) \to H_n(G;M)\to\cdots$$ where each morphism $H_n(G_v;M)\to H_n(G;M)$ occurring in this sequence is induced by the inclusion $G_v\to G$, and each composition $$H_n(G_e;M)\to \bigoplus_{v\in V(Y)} H_n(G_v;M)\overset{\pi_v}{\longrightarrow} H_n(G_v;M)$$ where $\pi_v$ denotes the natural projection onto $H_n(G_v;M)$, is given by $(f_{e})_*$ (respectively $-(f_{\ol{e}})_*$) if $v=t(e)$ (respectively $v=o(e)$), and is $0$ otherwise. (Cf.~\cite{Ch76} or \cite[Section~II.2.8]{Se80}.) Dually, there is also a Mayer-Vietoris sequence in cohomology, which we won't formulate explicitly here.

\chapter{Embedding Theorems for $p$-Groups}\label{ch:embedding theorems}

\noindent
In this chapter we consider two types of embedding theorems for $p$-groups:
one for amalgamating (filtered) $p$-groups, and one for extending partial automorphisms between subgroups of a given $p$-group to inner automorphisms of a larger $p$-group. These results will be basic for establishing in the next chapter that certain fundamental groups of graphs of groups are (virtually) residually $p$.

\section{An amalgamation theorem for filtered $p$-groups}\label{sec:An Embedding Theorem}

\noindent
An {\it amalgam} of groups consists of a pair $(G,H)$ of groups with a common subgroup~$U$; notation: $G\cup H|U$. An {\it embedding} of an amalgam $G\cup H|U$ into a group~$W$ is a pair $(\alpha,\beta)$ consisting  of injective group morphisms $\alpha\colon G\to W$, $\beta\colon H\to W$ such that $\alpha|U=\beta|U$. (Figure~\ref{fig:embedding}.)
\begin{figure} \centering
\[ \xymatrix{
& W  &\\
G \ar@{.>}[ur]^{\alpha} &  & H \ar@{.>}[ul]_{\beta}\\
& \ar[ul]  U  \ar[ur] &
}\]
\caption{An embedding of an amalgam of groups}\label{fig:embedding}
\end{figure}
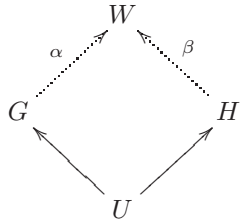
Such an embedding of $G\cup H|U$ into $W$ is said to be {\it strong} if in addition $\alpha(G)\cap\beta(H)=\alpha(U)$.
The following is a variant of a theorem of Higman \cite{Hi64}. The equivalences of (1)--(4) may also be deduced from \cite[Theorem~2]{Le87}; for our application in Section~\ref{sec:A Reduction Theorem}, we are mostly interested in the addendum.

\index{amalgam}
\index{amalgam!embedding}
\index{amalgam!strong embedding}
\index{Higman}

\begin{theorem}\label{thm:higman}
Let $G\cup H|U$ be an amalgam of $p$-groups. The following are equivalent:
\begin{enumerate}
\item $G\cup H|U$ is strongly embeddable into a $p$-group.
\item $G\cup H|U$ is embeddable into a $p$-group.
\item There exist finite-length chief filtrations of $G$ and $H$ inducing the same filtration of $U$.
\item There exist finite-length central $p$-filtrations $\mathbf G$ of $G$ and $\mathbf H$ of $H$ inducing the same filtration of $U$.
\end{enumerate}
Moreover, if we are given $\mathbf G$ and $\mathbf H$ as in \textup{(4)}, then there exists a strong embedding~$(\alpha,\beta)$ of $G\cup H|U$ into a $p$-group $W$ and a central $p$-filtration $\mathbf W$ of $W$ of finite length
which induces $\alpha(\mathbf G)$ on $\alpha(G)$ and  $\beta(\mathbf H)$ on $\beta(H)$.
\end{theorem}

Note that the implication ``(1)~$\Rightarrow$~(2)'' is trivial.
If $W$ is a $p$-group containing both $G$ and $H$ as subgroups, then any chief filtration of $W$ induces a chief filtration on $G$ and a chief filtration on $H$. This  shows ``(2)~$\Rightarrow$~(3)'', and ``(3)~$\Rightarrow$~(4)'' follows from the fact that  each chief filtration of a $p$-group is a central $p$-filtration. Hence it remains to give the proof of ``(4)~$\Rightarrow$~(1)'' in the stronger formulation given in the ``moreover'' part.

\medskip

The class of  \emph{elementary abelian} $p$-groups has the (strong) amalgamation property, i.e.:
every amalgam of elementary abelian $p$-groups $G\cup H|U$ always admits a strong embedding into an elementary abelian $p$-group $W$.
Somewhat more generally,
let $G$ and $H$ be abelian groups (written additively), and let $\varphi\colon U\to G$, $\psi\colon U\to H$ be injective morphisms from a group
$U$ to $G$ respectively $H$. Then
$$K := \big\{ \varphi(u)-\psi(u) : u\in U\big\}$$
is a subgroup of $G\oplus H$, and
$$G\oplus_{U} H := (G\oplus H) / K$$
is called the fiber sum of $G$ and $H$ over $U$ (via $\varphi$, $\psi$). It is easy to check that the natural morphisms $G\to G\oplus_{U} H$ and $H\to G\oplus_{U} H$ are injective. The fiber sum $G\oplus_{U} H$ is the solution to an universal problem: it is the push-out, in the category of abelian groups, of the diagram
$G\overset{\varphi}{\leftarrow} U\overset{\psi}{\rightarrow} H$.
(In connection with this construction it is interesting to note that no varieties of groups other than the variety of all groups and varieties of abelian groups are known to have the amalgamation property; cf.~\cite[pp.~42--43]{Ne67}.)

\medskip

Let now $G\cup H|U$ be an amalgam of $p$-groups and $\mathbf G=\{G_i\}$ and $\mathbf H=\{H_i\}$ be central $p$-filtrations of $G$ respectively $H$ as in \textup{(4)}. For  showing the rest of Theorem~\ref{thm:higman}, after refining $\mathbf G$ and $\mathbf H$ suitably, we may assume that $\mathbf G\cap U=\mathbf H\cap U$. (Lemma~\ref{lem:refine}.) In this case, we actually have a more precise result:

\begin{proposition}\label{prop:higman, more precise}
Let $G\cup H|U$ be an amalgam of $p$-groups, and suppose $\mathbf G$ and~$\mathbf H$ are finite-length central $p$-filtrations of $G$ respectively $H$ which intersect to the same filtration of $U$. Then there exists a $p$-group $W$ containing both $G$ and~$H$ as subgroups, with $G\cap H=U$, a finite-length central $p$-filtration ${\mathbf W}$ of $W$, as well as compatible stretchings ${\mathbf G}^*$ of $\mathbf G$ respectively ${\mathbf H}^*$ of $\mathbf H$, with the following properties:
\begin{itemize}
\item[(A1)] ${\mathbf W}\cap G={\mathbf G}^*$ and ${\mathbf W}\cap H={\mathbf H}^*$;
\item[(A2)] for each $i\geq 1$,  identifying $L_i({\mathbf G}^*)$ and $L_i({\mathbf H}^*)$ with subgroups of $L_i({\mathbf W})$ in the natural way, and similarly  $L_i({\mathbf U})$ with subgroups of $L_i({\mathbf G}^*)$ and $L_i({\mathbf H}^*)$, where
${\mathbf U}:={\mathbf G}^*\cap U={\mathbf H^*}\cap U$,
we have
$$L_i({\mathbf G}^*) \cap L_i({\mathbf H}^*) = L_i({\mathbf U});$$
hence $L_i({\mathbf W})$ embeds $L_i({\mathbf G}^*) \oplus_{L_i({\mathbf U})} L_i({\mathbf H}^*)$ in a natural way.
\end{itemize}
\end{proposition}

The remainder of this section is occupied by the proof of this proposition. The main
idea of the argument is to apply the fiber sum construction inductively to the layers of the given filtrations $\mathbf G$ and $\mathbf H$.
Before we can carry this out, we need some preliminary observations on group algebras, wreath products, and a construction of central filtrations on semidirect products.

\subsection{Group algebras}
Let $R$ be a commutative ring and $G$ be a group, and let $R[G]$ be the group algebra of $G$ over $R$. (As a general reference for group algebras we use \cite{Pa77}.)

\index{group algebra}

\medskip

The {\it augmentation ideal}\/ of $R[G]$ is the kernel $\omega=\omega(R[G])$ of the $R$-algebra morphism $R[G]\to R$ given by $g\mapsto 1$ for every $g\in G$. (Here and below, ``ideal'' will always mean ``two-sided ideal'' unless noted otherwise.) It is easy to see (cf.~\cite[Lemma~3.1.1]{Pa77}) that
$$\omega = \sum_{g\in G} R[G](1-g).$$
Note that if $R$ has characteristic $p$, then $R[G]=R\otimes_{\F_p}\F_p[G]$ and hence $\omega(R[G])^n=R\otimes_{\F_p}\omega(\F_p[G])^n$ for every $n$.

\medskip

Given an ideal $I$ of $R[G]$ we say that a
descending sequence $\{I_n\}_{n\geq 1}$ of ideals of $R[G]$ satisfying $I_1=I$ and $I_m I_n\leq I_{m+n}$ for all $m$, $n$ is a {\it filtration} of $I$.
Every strongly central filtration
$$G=G_1 \geq G_2 \geq \cdots \geq G_n \geq \cdots$$
of $G$ gives rise to a filtration $\{I_n\}$ of the ideal $\omega$ as follows:
For each $g\in G$ define $v(g)$ to be the largest $k\in\N^{\geq 1}\cup\{\infty\}$ for which $g\in G_k$, where we set $G_\infty:=\bigcap_{n\geq 1} G_n$. For each $n\geq 1$ let $I_n$ be the left $R$-submodule of $R[G]$ spanned by all products of the form
$$(1-g_1)\cdot (1-g_2) \cdots (1-g_m)$$
for some $m$ with $v(g_1)+v(g_2)+\cdots+v(g_m)\geq n$. It is easily seen that $\{I_n\}$ is a filtration of $\omega$; cf.~\cite[p.~87]{Pa77}.
Conversely, let
$$\omega = I_1 \geq I_2 \geq \cdots \geq I_n \geq \cdots$$
be a filtration of $\omega$. Then we associate to $\{I_n\}$ a filtration $\{G_n\}$ of $G$ by setting
$$G_n := \big\{ g\in G : 1-g\in I_n \big\}.$$
If $R$ has characteristic $p$, then $\{G_n\}$ is a dimensional $p$-filtration of $G$; cf.~\cite[Lem\-ma~3.3.2]{Pa77}.
In particular, taking $I_n=\omega^n$ for every $n$, we obtain a corresponding dimensional $p$-filtration of $G$. In fact, this filtration turns out to be exactly the lower dimensional $p$-filtration $\{D_n\}$ of $G$:

\begin{theorem} \label{thm:Jennings}
Suppose $R$ is of characteristic $p$ and $G$ is a $p$-group. Then
$$D_n = \{ g\in G: 1-g\in \omega^n \}\qquad\text{for all $n\geq 1$.}$$
Moreover, the filtration of $\omega$ determined by $\{D_n\}$ is precisely $\{\omega^n\}$, the powers of the augmentation ideal. The ideal  $\omega$ is nilpotent of nilpotency class
$$(p-1) \sum_n n \, \operatorname{dim}_{\F_p}(D_n/D_{n+1}).$$
\end{theorem}

This theorem is due to Jennings \cite{Je41}; for a proof see \cite[Theorems~3.3.6 and 3.3.7]{Pa77}.
The {\it nilpotency class} of an ideal $I$ of $R[G]$ is the smallest $d\geq 0$, if it exists, such that $I^{d+1}=0$. By the theorem, the nilpotency class of $\omega$ is an upper bound on the lower $p$-length (and hence the nilpotency class) of $G$.





We denote by
$$\operatorname{ann}(\omega) := \{ r\in R: r\omega=0 \}$$
the left annihilator of $\omega$ (a left ideal of $R[G]$).
Note that if $\omega^{d+1}=0$ then $\operatorname{ann}(\omega)$ contains $\omega^d$.
If $G$ is finite we define
$$\widehat{G} := \sum_{g\in G} g\in R[G].$$
Note that $\widehat{G}g=g\widehat{G}=\widehat{G}$ for all $g\in G$, hence
$\widehat{G}$ is central in $R[G]$.
It is easy to see that if $G$ is finite, then $\operatorname{ann}(\omega)=R\widehat{G}$; cf.~\cite[Lemma~3.1.2]{Pa77}.
The following observation is standard (cf.~\cite[p.~371]{Pa77}):

\begin{lemma}
Suppose $G$ is a $p$-group, and $R$ is a field of characteristic $p$. Then every non-zero ideal $I$ of $R[G]$ contains $\widehat{G}$.
\end{lemma}
\begin{proof}
Let $I\neq 0$ be an ideal of $R[G]$. The ideal $\omega$ is nilpotent, so we can choose an integer $m\geq 0$ maximal with $I\omega^m\neq 0$. Then clearly
$$0\neq I\omega^m\subseteq I\cap \operatorname{ann}(\omega)=I\cap R\widehat{G}$$
and thus $\widehat{G}\in I$ since $R$ is a field.
\end{proof}

Applying this lemma to the prime field $\F_p$ in place of $R$ and using that $R[G]=R\otimes_{\F_p} \F_p[G]$ yields:

\begin{corollary} \label{cor:hatG}
Suppose $G$ is a $p$-group and $\operatorname{char}(R)=p$. Then $\omega^d = R\widehat{G}$ where $d$ denotes the nilpotency class of $\omega$.
\end{corollary}

\subsection{Wreath products}

\noindent
Wreath products play an important role in the proof of the main result of \cite{Hi64}, and so naturally also in the proof of Theorem~\ref{thm:higman}. We recall the definition:
Let $X$, $H$ be groups.
We turn the set $X^H$ of maps $H\to X$ into a group under the coordinate-wise operations.
We have a right action
$$H\times X^H\to X^H\colon (h,f)\mapsto f^h$$
of $H$ on $X^H$, given by
$$f^{h}(k)=f(hk)\qquad\text{ for $f\in X^H$ and $h,k\in H$.}$$
The (unrestricted) {\it wreath product}\/  of $X$ and $H$ is a semidirect product $X\wr H= H\ltimes X^H$; its underlying set is $H\times X^H$, with group operation
$$(h_1,f_1)\cdot (h_2,f_2) = (h_1 h_2, f_1^{h_2} f_2).$$
The group $X^H$ is called the {\it base group} of  $X\wr H$.
Both  $X^H$ and the group $H$   can be naturally identified with subgroups of $X\wr H$. (Note that then for every $h\in H$ and $f\in X^H$ the conjugate $h^{-1}fh$ of  $f$ by $h$ is conveniently given by $f^h$.)
If $K$ is a subgroup (respectively normal subgroup) of $H$, then $KX^H$ is a subgroup (respectively normal subgroup) of $X\wr H$.
Note also that if $X$ is a subgroup of a group $Y$, then $X\wr H$ is a subgroup of $Y\wr H$.

\index{wreath product}

\medskip

The wreath product $X\wr H$ contains a copy of every extension of $H$ by $X$ (theorem of Kalu\u{z}in-Krasner): in fact,
given a group morphism $\theta\colon A\to H$ (not necessarily surjective) with kernel $X$ there always exists a map $a\mapsto f_a\colon A\to X^H$ such that the map
$$a\mapsto (\theta(a),f_a) \colon A\to X\wr H$$
is an injective group morphism (cf.~\cite[p.~303]{Hi64}).
The map $a\mapsto f_a$ is obtained as follows: One first chooses a {\it countermap} to $\theta$, that is, a map $\theta^*\colon H\to A$ (not necessarily a group morphism) such that
$$\theta(\theta^*(\theta(a)\cdot h)) = \theta(a)\cdot \theta(\theta^*(h)) \qquad \text{for all $a\in A$, $h\in H$,}$$
or, in other words:
$$\theta^*(\theta(a)\cdot h) \equiv a\cdot \theta^*(h)\mod X \qquad \text{for all $a\in A$, $h\in H$.}$$
Such countermaps always exist:
Choose a system $S$ of right coset representatives of the subgroup $\theta(A)$ of $H$, and for $h\in H$ denote the unique $s\in S$ such that $hs^{-1}\in\theta(A)$ by $s(h)$. For every $h\in H$ choose $a(h)\in A$ such that $h=\theta(a(h))s(h)$, and for each $s\in S$
choose an arbitrary $a_s\in A$.  Then $\theta^*(h):=a(h)\cdot a_{s(h)}$ defines a countermap to $\theta$.
(So for example, if $\theta$ is surjective then every right inverse of  $\theta$ is a countermap to $\theta$.)

Now set
$$f_a(h) := \big(\theta^*(\theta(a)\cdot h)\big)^{-1}\cdot a\cdot \theta^*(h).$$
Note that $f_a(h)$ indeed lies in $\ker(\theta)=X$, so $f_a$ is an element of $X^H$.
If we want to stress the dependence of $a\mapsto f_a$ on the choice of countermap $\theta^*$, we write $f_a^{\theta^*}$ instead of $f_a$.
We call an embedding $A\to X\wr H$ of the form $a\mapsto (\theta(a),f_a^{\theta^*})$ a {\it standard embedding} of $A$ into $X\wr H$.
If in addition $X$ is given as a subgroup of a group $Y$, we also speak of the composition $A\to X\wr H\to Y\wr H$ of a standard embedding $A\to X\wr H$ with the natural inclusion $X\wr H\to Y\wr H$ as a {\it standard embedding} of $A$ into $Y\wr H$.

\index{wreath product!standard embedding}
\index{countermap}

\medskip

In the following let $W=X\wr H$ be the wreath product of the groups $X$ and $H$.
It is well-known (cf.~\cite{B59}) that if $X\neq 1$ and $H$ is finite, then $W$ is nilpotent if and only if there is a prime $p$ such that $X$ is a nilpotent of finite $p$-power exponent and $H$ is a  $p$-group.
We are mostly interested in the case where the group $H$ is a $p$-group, and $X$ is an elementary abelian $p$-group, and hence may be construed as the additive group of a finite-dimensional $\F_p$-algebra. Somewhat more generally, {\it from now on we assume that $X$ is the additive group of a commutative ring $R$ and~$H$ is finite.}\/ Below we simply denote $X$ by $R$.
We now can and will canonically identify the additive group of the group algebra $R[H]$ of $H$ over $R$ with the base group $R^H$ of the wreath product $W=R\wr H$. After this identification, the action of $H$ on $R^H$ is related to the action of $H$ by left multiplication on the group ring $R[H]$ via
$$f^h = h^{-1} \cdot f\qquad\text{ for all $f\in R[H]$, $h\in H$.}$$
(To avoid confusion between the group operation in $R\wr H$ and the multiplication in the ring $R[H]$, we will always use a dot $\cdot$ to denote the latter. We also reserve~$f$, possibly with decorations, to denote elements of $R[H]$ and $h$ for the elements of the subgroup $H$ of $W$.)
We therefore have the useful identity
\begin{equation}\label{eq:[f,h]}
[f,h] = (h^{-1}-1)\cdot f \qquad (f\in R[H], h\in H),
\end{equation}
where the commutator $[f,h]=f^{-1}h^{-1}fh$ is computed in $W$ and the right-hand side in the group algebra $R[H]$.

\medskip

 We let $\pi$ be the natural map $R\wr H\to R[H]$, which is not a group morphism in general: in fact, for elements $g_1=h_1f_1$ and $g_2=h_2f_2$ of $W$ (where $h_i\in H$, $f_i=\pi(g_i)\in R[H]$), we have
\begin{equation}\label{eq:pi1}
\pi(g_1g_2) = h_2^{-1}\cdot f_1+f_2.
\end{equation}
For $f\in R[H]$, $h\in H$, $g=hf$ we obtain
\begin{equation}\label{eq:pi2}
\pi(g^{-1})=-h\cdot f,
\end{equation}
and for later use we also record that if in addition $f'\in R[H]$ then
\begin{equation}\label{eq:pi}
\pi(f'h^f)=h^{-1}\cdot f'+(1-h^{-1})\cdot f.
\end{equation}
Multiplication $f\mapsto r\cdot f$ in $R[H]$ by a fixed element $r\in R$ extends uniquely to an endomorphism of $W$ which is the identity on $H$.
We also note:

\begin{lemma}\label{lem:fully inv}
Let $G$ be a fully invariant subgroup of $W$. Then $G\cap R[H]$ is a left ideal of $R[H]$, with $G\cap R[H]=\pi(G)$.
 Moreover, if $S$ is a generating set for $G$, then $\pi(S)$ generates  the left ideal $G\cap R[H]$.
\end{lemma}
\begin{proof}
Clearly $G\cap R[H]$ is an additive subgroup of $R[H]$ contained in $\pi(G)$. The subgroup $G$ of $W$ being fully invariant implies that $G=(G\cap H)\, (G\cap R[H])$ and thus $G\cap R[H]=\pi(G)$. Since $G$ is fully invariant, $G\cap R[H]$ is also closed under multiplication by elements from $R$, and
if $f\in G\cap R[H]$ and $h\in H$, then $h\cdot f=f^{h^{-1}}\in G$ since $G$ is normal.
Hence $G\cap R[H]$ is a left ideal of $R[H]$.
Now let $S$ be a generating set for $G$. Then $\pi(G)$ clearly contains the left ideal $I:=R[H]\,\pi(S)$ of $R[H]$ generated by $\pi(S)$, and the identities \eqref{eq:pi1} and \eqref{eq:pi2} yield $\pi(G)=I$.
\end{proof}

The following theorem due to Buckley \cite{Bu70} shows that there is a close relationship between the lower central series of the wreath product $W=R\wr H$ and the powers of the ideal $\omega=\omega(R[H])$ of the group ring $R[H]$. 

\begin{theorem}
For every $n$ we have
$$\gamma_{n+1}(W)\cap R[H] = \omega^n.$$
\end{theorem}

Since the proof of the theorem above given in \cite{Bu70} is somewhat roundabout (it is deduced from properties of polynomial functions on groups), we give a direct proof of a refinement valid in the situation of interest to us, namely, where $\operatorname{char}(R)=p$. {\it From now until the end of this subsection we therefore assume that~$R$ has characteristic $p$.}\/

\begin{proposition}\label{prop:Buckley}
For all $n$:
$$D^p_{n+1}(W)\cap R[H] = \gamma^p_{n+1}(W)\cap R[H] =  \gamma_{n+1}(W)\cap R[H] = \omega^n.$$
\end{proposition}

\begin{proof}
Since
$$D^p_{n+1}(W)\geq \gamma^p_{n+1}(W)\geq \gamma_{n+1}(W),$$
it is enough to show that
\begin{equation}\label{eq:D and omega}
D^p_{n+1}(W)\cap R[H] \leq \omega^n \leq \gamma_{n+1}(W)\cap R[H].
\end{equation}
We proceed by induction on $n$, the case $n=0$ holding vacuously. Suppose we have established \eqref{eq:D and omega} for some $n$.
Then for every $f\in\omega^n$ and $h\in H$, by \eqref{eq:[f,h]} we have $$(h^{-1}-1)\cdot f = [f,h]\in [\gamma_{n+1}(W),W]\leq\gamma_{n+2}(W).$$
Lemma~\ref{lem:fully inv} now yields
$\omega^{n+1} \leq \gamma_{n+2}(W)\cap R[H]$.
For the inclusion on the left in~\eqref{eq:D and omega}, recall first that
the subgroup
$$D^p_{n+2}(W) = \big(D^p_{\lceil (n+2)/p \rceil}(W)\big)^p \prod_{i+j=n+2} \big[D^p_i(W), D^p_j(W)\big]$$
of $W$ is fully invariant. Thus by Lemma~\ref{lem:fully inv} it suffices to show the following claim, where we abbreviate $D_i=D^p_i(W)$: 

\begin{claim}
Suppose $h,h_i,h_j\in H$ and $f,f_i,f_j\in R[H]$ are such that
$$hf\in D_{\lceil (n+2)/p \rceil}\quad\text{and}\quad
h_if_i\in D_i, h_jf_j\in D_j \qquad (i+j=n+2).$$
Then
$$\pi\big((hf)^p\big)\in\omega^{n+1} \quad\text{and}\quad \pi\big( [ h_if_i, h_j f_j ] \big) \in \omega^{n+1}.$$
\end{claim}

For the first part, we note that, writing  $m=\lceil (n+2)/p \rceil$,
we have $f=\pi(hf)\in\pi(D_m)=D_m\cap R[H]$  by Lemma~\ref{lem:fully inv}, and hence $f\in \omega^{m-1}$ by inductive hypothesis. This yields $h=(hf)f^{-1}\in D_m$ and thus $h\in D_m^p(H)$. We obtain
$$\pi\big( (hf)^p \big) = (1+h^{-1}+h^{-p}+\cdots+h^{-(p-1)})\cdot f = (h^{-1}-1)^{p-1}\cdot f\in\omega^{n+1}$$
using Theorem~\ref{thm:Jennings}. For the second part, by standard commutator formulas (cf., e.g., \cite[Section 0.1]{DdSMS}) we have
\begin{multline*}
[h_if_i, h_jf_j] = [h_i, h_jf_j]^{f_i}[f_i,h_if_j]
= [h_i,f_j]^{f_i} [h_i,h_j]^{f_jf_i} [f_i, h_j]^{f_j} \\
= [h_i,f_j] [h_i,h_j]^{f_jf_i} [f_i, h_j],
\end{multline*}
using that $[f_i,f_j]=1$ as well as $[h_i,f_j]^{f_i}=[h_i,f_j]$ and $[f_i,h_j]^{f_j}=[f_i,h_j]$ (since $[h_i,f_j],[f_i,h_j]\in R[H]$).
By \eqref{eq:pi} this yields
$$\pi\big( [ h_if_i, h_j f_j ] \big) =
(1-[h_j,h_i])\cdot (f_i+f_j) + [h_j,h_i]\cdot (1-h_i^{-1})\cdot f_j - (1-h_j^{-1})\cdot f_i.$$
By Theorem~\ref{thm:Jennings} we have
$1-h_i^{-1}\in\omega^i$, $1-h_j^{-1}\in\omega^j$, and $1-[h_j,h_i]\in\omega^{n+2}$; and
by inductive hypothesis $f_i\in\omega^{i-1}$ and $f_j\in\omega^{j-1}$. This implies the second part of our claim.
\end{proof}

Combining this proposition with Lemma~\ref{lem:lower central for semidirect products} and Jennings' theorem (Theorem~\ref{thm:Jennings} above) we obtain:

\begin{corollary}
Suppose  $H$ is a $p$-group. Then $W=R\wr H$ is nilpotent, and the nilpotency class, lower $p$-length, and lower dimensional $p$-length of $W$ all equal
$$1+\textup{(nilpotency class of $\omega$)}=1+ (p-1) \sum_{n\geq 1} n \, \operatorname{dim}_{\F_p}(D_n^p(H)/D_{n+1}^p(H)).$$
\end{corollary}

In the proof of the main result of this section we need:

\begin{corollary}\label{cor:Buckley}
Let $A$ be a group containing $R$ as a central subgroup, let $\theta\colon A\to H$ be a group morphism to a $p$-group $H$ \textup{(}not necessarily surjective\textup{)} with kernel $R$, and let $\alpha\colon A\to W=R\wr H$ be a corresponding standard embedding. Then
$$\alpha(R)=\omega^{d}=\gamma_{d+1}^p(W)\cap R[H]=R\widehat{H}$$
where $d$ is the nilpotency class of $\omega=\omega(R[H])$.
\end{corollary}
\begin{proof}
Let $\theta^*$ be a countermap to $\theta$ such that $\alpha(a)=(\theta(a),f_a^{\theta^*})$ for all $a\in A$. Since $R\leq Z(A)$ we have for all $r\in R$, $h\in H$:
\begin{align*}
f_r(h) &= \theta^*(\theta(r)\cdot h)^{-1}\cdot r\cdot \theta^*(h) \\
       &= \theta^*(h)^{-1}\cdot r\cdot \theta^*(h) = r,
\end{align*}
and hence $f_r = r\widehat{H}\in R[H]$. Now use Corollary~\ref{cor:hatG} and Pro\-po\-sition~\ref{prop:Buckley}.
\end{proof}

\subsection{Central filtrations and semidirect products}
Assume that the group $G=B\rtimes H$ is a semidirect product of its subgroups $B$ and $H$, with $B\trianglelefteq G$. 
The following is easy to verify:

\begin{lemma}\label{lem:filtration of semidirect product}
Let $\{H_n\}$ be a filtration of $H$. Then
$$H_n B  = \{ hb : h\in H_n,\, b\in B\}$$
defines a filtration of $G=B\rtimes H$. If $\{H_n\}$ is normal, then so is $\{H_nB\}$, and for every $m<n$ the natural morphism $H_m/H_{n}\to  H_mB/H_nB$ is an isomorphism. \textup{(}So if $\{H_n\}$ is a central filtration, or a central $p$-filtration,  then $\{H_nB\}$ also has the respective property.\textup{)}
\end{lemma}

Let now $\{G_n\}$ be a central filtration of $G$, and put $B_n:=B\cap G_n$ for each $n$. Note that $[G,B]\leq B$, hence $[G,B_n]\leq B_{n+1}$ for each $n$. 
Assume also that a central finite-length filtration
$$H=H_1\geq H_2 \geq \cdots \geq H_m \geq 1$$
of $H$ is given. We can then combine the central filtration $\{B_n\}$ of $B$ with the central filtration $\{H_nB\}$ of $G$ to a central filtration
$$G=H_1B  \geq \cdots \geq H_mB  \geq B = B_1 \geq B_{2}\geq \cdots \geq B_n\geq \cdots$$
of $G$. (This construction of filtrations on semidirect products is used in the next subsection.)

\subsection{Proof of the embedding theorem}
We need yet another lemma, indicating how an embedding of an amalgam into a wreath product can be constructed inductively. For this, let $G\cup H|U$ be an amalgam of groups. An amalgam $X\cup Y|V$ is called a {\it subamalgam}\/ of  $G\cup H|U$  if $X\leq G$, $Y\leq H$ and $U\cap X=U\cap Y=V$. The subamalgam $X\cup Y|V$ of  $G\cup H|U$  is called {\it normal}\/ if $X$ is normal in $G$ and~$Y$ normal in $H$. In this case, identifying $UX/X\cong U/V$ with $UY/Y$ in the natural way, we can form the {\it factor amalgam}\/ $(G/X)\cup (H/Y)|(U/V)$.
The first part of the following Lemma is \cite[Lemma~2]{Hi64}:

\index{subamalgam}
\index{factor amalgam}

\begin{lemma}
Let $X\cup Y|V$ be a normal subamalgam of the amalgam $G\cup H|U$, and suppose $V$ is central in both $G$ and $H$.
 If $X\cup Y|V$ embeds into a group $T$ and $(G/X)\cup (H/Y)|(U/V)$ embeds into a group $K$, then $G\cup H|U$  embeds into $T\wr K$.
 An analogous statement holds for strong embeddings in place of embeddings.
\end{lemma}
\begin{proof}
Suppose $T$ is a group containing both $X$ and $Y$ as subgroups, and $\theta\colon G\to K$, $\phi\colon H\to K$ are group morphisms with $\ker(\theta)=X$, $\ker(\phi)=Y$ and $\theta|U=\phi|U$. Then by \cite[Lemma~2]{Hi64} there are standard embeddings
$$G\to T\wr K\colon g\mapsto (\theta(g), f_g^{\theta_1^*})$$
and
$$H\to T\wr K\colon h\mapsto (\phi(h), f_h^{\phi_1^*})$$
which form an embedding of $G\cup H|U$ into $T\wr K$. This is already the first part of the lemma.
However, for the statement concerning strong embeddings, we need a more detailed description of how $f_g^{\theta_1^*}$ and $f_h^{\phi_1^*}$ are constructed.
First one chooses a countermap $\theta^*\colon K\to U$ to $\theta|U=\phi|U\colon U\to K$. This yields a standard embedding $$U\to V\wr K\colon u\to (\theta(u),f_u^{\theta^*})=(\phi(u),f_u^{\theta^*}).$$
One can show (cf.~\cite[Lemma~1]{Hi64}) that there exists a countermap
$\theta_1^*\colon K\to G$ to $\theta\colon G\to K$ such that the map $\mu\colon K\to G$ defined by $\theta^*\cdot \mu = \theta_1^*$
 is constant on each right coset of $\theta(U)=\phi(U)$ in $K$.
Similarly, there exists a countermap
$\phi_1^*\colon K\to H$ to $\phi\colon H\to K$ such that the map $\nu\colon K\to H$ defined by $\theta^*\cdot \nu = \phi_1^*$ is constant on each right coset of $\theta(U)=\phi(U)$ in $K$. This yields standard embeddings $G\to X\wr K$ and $H\to Y\wr K$ in the manner described above.

Note that for $u\in U$ and $k\in K$ we have
\begin{align*}
f_u^{\theta_1^*}(k) &= \big(\theta_1^*(\theta(u)\cdot k)\big)^{-1}\cdot u\cdot \theta_1^*(k) \\
&= \big(\mu(\theta(u)\cdot k)\big)^{-1}\cdot \big(\theta^*(\theta(u)\cdot k)\big)^{-1}\cdot u\cdot \theta^*(k)\cdot \mu(k) \\
&= \mu(k)^{-1}\cdot f_u^{\theta^*}(k)\cdot \mu(k)\\
&= f_u^{\theta^*}(k)
\end{align*}
since $V$ is central in $G$. Similarly we have $f_u^{\phi_1^*}=f_u^{\theta^*}$ for all $u\in U$.
Hence the diagram in Figure~\ref{img:embeddings} is commutative.
\begin{figure} \centering
\[\xymatrix{
& T\wr K & \\
X\wr K \ar[ur] & & Y\wr K\ar[ul] \\
  &\ar[ul] V\wr K\ar[ur]  &\\
G \ar[uu] &    & H \ar[uu]\\
& \ar[ul]  U \ar[uu] \ar[ur] &
}\]
\caption{Embeddings into wreath products}
\label{img:embeddings}
\end{figure}
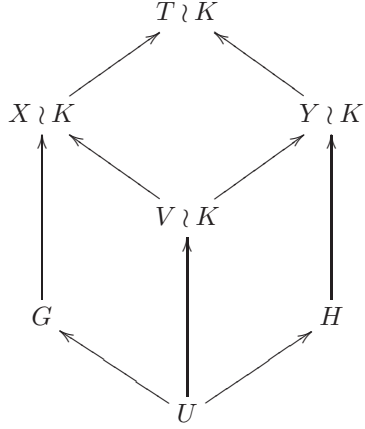

To see the statement about strong embeddings, suppose that $X\cap Y=V$ as subgroups of $T$ and $\theta(G)\cap\phi(H)=\phi(U)$. Let $g\in G$, $h\in H$ such that $(\theta(g),f_g^{\theta_1^*})=(\phi(h),f_h^{\phi_1^*})$. Then $\theta(g)=\phi(h)\in\theta(U)=\phi(U)$, hence
$\mu(\theta(g)\cdot k)=\mu(k)$ and
$\nu(\phi(h)\cdot k)=\nu(k)$ for all $k\in K$, and
$f_g^{\theta_1^*}=f_h^{\phi_1^*}$, hence $f_g^{\theta_1^*}(k)=f_h^{\phi_1^*}(k)\in X\cap Y=V$ for all $k\in K$.
Since $V$ is central in $G$, this yields
\begin{align*}
f_g^{\theta_1^*}(k) &= \big(\theta_1^*(\theta(g)\cdot k)\big)^{-1}\cdot g\cdot \theta_1^*(k) \\
&= \mu(k)^{-1} \cdot \big(\theta^*(\theta(g)\cdot k)\big)^{-1}\cdot g\cdot \theta^*(k)\cdot\mu(k)\\
&= \big(\theta^*(\theta(g)\cdot k)\big)^{-1}\cdot g\cdot \theta^*(k)
\end{align*}
for all $k\in K$, and similarly
$$f_h^{\phi_1^*}(k) = \big(\theta^*(\phi(h)\cdot k)\big)^{-1}\cdot h\cdot \theta^*(k)\qquad\text{for all $k\in K$.}$$
This yields $g=h\in U$.
\end{proof}

With these tools at hand, we can now give the proof of Proposition~\ref{prop:higman, more precise}.
Given a filtration $\mathbf G=\{G_i\}_{i\geq 1}$ of a group, we call the number of distinct non-trivial groups occurring among $G_1,G_2,\dots$ the {\it essential length}\/ of $\mathbf G$. Note that the essential length of $\mathbf G$ is an invariant of the equivalence class of $\mathbf G$. The length of $\mathbf G$ is never smaller than its essential length.

\index{filtration!essential length}
\index{filtration!stretching}
\index{stretching}

Let $G\cup H|U$ be an amalgam of $p$-groups, and let $\mathbf G=\{G_i\}$ and $\mathbf H=\{H_i\}$ be finite-length central $p$-filtrations of $G$ respectively $H$ with $\mathbf G\cap U=\mathbf H\cap U$. 
We proceed by the maximum $\mu$ of the essential lengths of $\mathbf G$ and $\mathbf H$ to show that there exists a strong embedding of $G\cup H|U$ into a $p$-group $W$ and a central $p$-filtration of $W$ as well as compatible stretchings $\mathbf G^*$ of $\mathbf G$ and $\mathbf H^*$  of $\mathbf H$ satisfying conditions (A1) and (A2) in Proposition~\ref{prop:higman, more precise} (after identifying $G$ and $H$ with their images under the embedding of $G\cup H|U$ into $W$).

If $\mu=0$, then there is nothing to show, so assume $\mu>0$. Let $n$ be the maximum of the lengths of $\mathbf G$ and $\mathbf H$.
Then $X:=G_{n}$ is central in $G$ and $Y:=H_{n}$ is central in $H$, and both groups are elementary abelian $p$-groups; their fiber sum~$T$ over $V:=X\cap U=Y\cap U$  is an elementary abelian $p$-group, and $X\cap Y=V$ after identifying $X$ and $Y$ in the natural way with subgroups of $T$. By induction and after replacing $\mathbf G$ and $\mathbf H$ by compatible stretchings if necessary, we may assume that we have group morphisms $\theta\colon G\to K$ and $\phi\colon H\to K$
such that
$$\ker(\theta)=X,\ \ker(\phi)=Y,\ \theta|U=\phi|U, \text{ and } \theta(G)\cap\phi(H)=\theta(U),$$
as well as a central $p$-filtration $\mathbf K=\{K_i\}$
of $K$ such that
\begin{enumerate}
\item $K_i\cap \theta(G) = \theta(G_i)$ and $K_i\cap\phi(H) = \phi(H_i)$ for $1\leq i\leq n$;
\item $K_i\cap \theta(G) = \theta(G_n)$ and $K_i\cap\phi(H) = \phi(H_n)$ for $i\geq n$;
\item $L_i(\mathbf G)\cap L_i(\mathbf H)=L_i(\mathbf U)$  (as subgroups of $L_i(\mathbf K)$) for $1\leq i<n$, where $\mathbf U:=\mathbf G\cap U=\mathbf H\cap U$.
\end{enumerate}
Put $W:=T\wr K$, and 
following the proof of the lemma above take standard embeddings
$$G\to W\colon g\mapsto \alpha(g):=(\theta(g), f_g)$$
and
$$H\to W\colon h\mapsto \beta(h):=(\phi(h), f_h)$$
which form a strong embedding of $G\cup H|U$ into $W$. Now for every $i=1,\dots,m$, where $m=\text{length of $\mathbf K$}$, define the normal subgroup
$$W_i := K_i \cdot T^K= \big\{ (k,f)\in W: k\in K_i, f\in T^K\big\}$$
of $W$, and for $i=1,\dots,l$, where $l=\text{nilpotency class of $W$}$, set
$$W_{m+i} := T^K\cap \gamma^p_{i}(W).$$
Then
$$W=W_1\geq W_2 \geq \cdots \geq W_m \geq W_{m+1} = T^K \geq W_{m+2}\geq\cdots\geq W_{m+l}\geq 1$$
is a central $p$-filtration of $W$. (See the remark following Lemma~\ref{lem:filtration of semidirect product}.)

It remains to show that there are compatible stretchings $\mathbf G^*$ of $\mathbf G$ and $\mathbf H^*$ of~$\mathbf H$ which, together with
$\mathbf W=\{W_i\}$, satisfy conditions (A1) and (A2) in Proposition~\ref{prop:higman, more precise}. Clearly
$$\alpha^{-1}(W_i) = \theta^{-1}(K_i)= \begin{cases}
G_i & \text{for $1\leq i\leq n$} \\
G_n & \text{for $n\leq i\leq m+1$.}
\end{cases}$$
On the other hand, by the construction of $\alpha$ in the proof of the preceding lemma, the image of $\alpha$ is contained in the subgroup $X\wr K$ of $W$, and by Corollary~\ref{cor:Buckley},
$$W_{m+l} \cap X^K = T\widehat{K}\cap X[K] = X\widehat{K} = \alpha(X),$$
hence
$\alpha^{-1}(W_{m+l}) = X = G_n$.
Therefore
$$\alpha^{-1}(W_{m+1})=\alpha^{-1}(W_{m+2})=\cdots=\alpha^{-1}(W_{m+l})=G_n.$$
Thus the stretching $\mathbf G^*=\{G_i^*\}_{i\geq 1}$ of $\mathbf G$ defined by
$$G_i^* = \begin{cases}
G_i & \text{for $1\leq i\leq n$} \\
G_n & \text{for $n\leq i\leq m+l$} \\
1   & \text{for $i>m+l$}
\end{cases}$$
has the property that $\alpha^{-1}(\mathbf W)=\mathbf G^*$.
Similarly one argues that the stretching
$\mathbf H^*=\{H_i^*\}_{i\geq 1}$ of $\mathbf H$ defined by
$$H_i^* = \begin{cases}
H_i & \text{for $1\leq i\leq n$} \\
H_n & \text{for $n\leq i\leq m+l$} \\
1   & \text{for $i>m+l$}
\end{cases}$$
satisfies  $\beta^{-1}(\mathbf W)=\mathbf H^*$; note that $\mathbf G^*$ and $\mathbf H^*$ are compatible. This shows condition (A1).
We also have, with $\omega=\omega(T[K])$:
$$L_i(\mathbf W) = \begin{cases}
L_i(\mathbf K) & \text{for $1\leq i\leq m$} \\
\omega^{i-m-1}/\omega^{i-m}			   & \text{for $m< i< m+l$} \\
\omega^{l-1}		   & \text{for $i=m+l$} \\
1			   & \text{for $i>m+l$.}
\end{cases}$$
Moreover
$$L_i(\mathbf G^*) = \begin{cases}
L_i(\mathbf G) 	& \text{for $1\leq i<n$} \\
1				& \text{for $n\leq i<m+l$} \\
G_n				& \text{for $i=m+l$} \\
1				& \text{for $i>m+l$,}
\end{cases}$$
and similarly for $L_i(\mathbf H^*)$. Property (A2) follows. \qed




\section{Extending partial automorphisms to inner automorphisms}
\label{sec:extending partial automorphisms}

\noindent
Throughout this section we let $G$ be a $p$-group, and we let $\varphi_i\colon A_i\to B_i$ ($i=1,\dots,r$) be partial automorphisms of $G$, i.e., isomorphisms between subgroups of $G$. In this section we want to discuss the following question:
\begin{quote}
{\it Under which conditions on the $\varphi_i$ is there a $p$-group $H$ which contains $G$ as a subgroup  such that each $\varphi_i$ extends to an inner automorphism of $H$?}
\end{quote}
This issue is of interest to us since the existence of $H$ as in the question above is equivalent to the iterated HNN extension
$$\big\langle G, t_1,\dots,t_r : t_ia_it_i^{-1}=\varphi_i(a_i) \text{ for $i=1,\dots,r$, $a_i\in A_i$}\big\rangle$$
of $G$ to be residually $p$.
(See Lemma~\ref{residual p criterion} below.)
It is clear that in order to find such $H$ it suffices to produce a $p$-group $G^*$ which contains $G$ as a subgroup and for each $i$ an extension $\varphi_i^*$ of $\varphi_i$ to an automorphism of $G^*$ such that the subgroup $A$ of $\operatorname{Aut}(G^*)$ generated by $\varphi^*_1,\dots,\varphi^*_r$ has $p$-power order: in this case, the semidirect product $H=G^*\rtimes A$ has the right property.

\index{partial automorphism}

\medskip

Note however that it is not enough to have, for each $i$ individually, an extension of $\varphi_i$ to a $p$-automorphism of some $p$-group containing $G$ as a subgroup:

\begin{example}
Suppose $G=\F_p^2$, and consider the automorphisms $\varphi=\left(\begin{smallmatrix} 1 & 1 \\ 0 & 1\end{smallmatrix}\right)$ and $\psi=\left(\begin{smallmatrix} 1 & 0 \\ 1 & 1\end{smallmatrix}\right)$ of $G$. Then $\varphi$ and $\psi$ both have order $p$; however $\varphi$ and $\psi$ do not commute. Since every $p$-subgroup of $\operatorname{GL}(n,\F_p)$ is conjugate to a subgroup of the group $\operatorname{UT}_1(n,\F_p)$ of upper unitriangular matrices (see, e.g., \cite[0.8]{DdSMS}), and
$\operatorname{UT}_1(2,\F_p)$ is abelian, there are no extensions of $\varphi$, $\psi$ to automorphisms of a $p$-group containing $G$ as a subgroup and which generate a $p$-group.
\end{example}

Given a partial automorphism  $\varphi\colon A\to B$ of $G$, we say that a filtration ${\mathbf G}=\{G_n\}$ of $G$ is {\it $\varphi$-invariant} if $\varphi(A\cap G_n)= B\cap G_n$ for all $n$; in this case, if in addition $\mathbf G$ is normal, then for each $n$, $\varphi$ induces an automorphism between the canonical images $AG_{n+1}\cap G_n/G_{n+1}$ and $BG_{n+1}\cap G_n/G_{n+1}$ of $A$ respectively $B$ in $L_n({\mathbf G})=G_n/G_{n+1}$, which we denote by $L_n(\varphi)$.

\medskip

The following proposition was shown in \cite[Lemma~1.2]{Ch94} (and later rediscovered in~\cite{Mo07}). It may be seen as an analogue for HNN extensions of Higman's theorem~\cite{Hi64} discussed in the previous section; the proof given in  \cite{Ch94} also employs wreath products.

\begin{proposition}
Let $\varphi\colon A\to B$ be a partial automorphism of $G$. Suppose that $G$ admits a $\varphi$-invariant chief filtration $\mathbf G=\{G_n\}$ such that
$$\varphi(a) \equiv a \bmod G_{n+1}\qquad\text{for all $n$ and all $a\in A\cap G_n$.}$$
Then $G$ embeds into a $p$-group $H$ such that $\varphi$ extends to an inner automorphism of $H$ and $\mathbf G$ is induced by a chief filtration of $H$.
\end{proposition}

\index{Chatzidakis}
\index{HNN extension}

This fact easily yields the following corollary, which
generalizes a well-known characterization of unipotent subgroups of the general linear group over a field of characteristic $p$; cf.~\cite[0.8]{DdSMS}:

\begin{corollary}\label{cor:chatzidakis}
The following are equivalent:
\begin{enumerate}
\item There exists a $p$-group $H$ which contains $G$ as a subgroup  such that each $\varphi_i$ extends to an inner automorphism of $H$.
\item There exists a $p$-group $G^*$ which contains $G$ as a subgroup and extensions of $\varphi_i$ to automorphisms of $G^*$ which generate a $p$-subgroup of $\operatorname{Aut}(G^*)$.
\item There exists a chief filtration $\{G_j\}$ of $G$ which is $\varphi_i$-invariant for each $i$ and such that
$$\varphi_i(a) \equiv a \bmod G_{j+1}\qquad\text{for all $i$, $j$ and all $a\in A_i\cap G_j$.}$$
\end{enumerate}
\end{corollary}
\begin{proof}
The equivalence of (1) and (2) was already noted above. The implication (3)~$\Rightarrow$~(1) follows from the proposition above by induction on $r$.
 For the converse note that every chief filtration of a $p$-group containing $G$ as a subgroup induces a chief filtration of $G$, and that every central filtration is invariant under inner automorphisms.
\end{proof}

Criterion (3) may actually be relaxed slightly:

\begin{lemma}\label{lem:hnn}
Suppose there exists a finite-length central filtration $\mathbf G=\{G_j\}$ of $G$ which is $\varphi_i$-invariant for each $i$ and such that for each $j$ there is a $p$-group~$H_j$ containing  $L_j(\mathbf G)$ as a subgroup and for each $i$ an extension of the partial automorphism $L_j(\varphi_{i})$ of $L_j(\mathbf G)$ induced by $\varphi_i$ to an inner automorphism of $H_j$. Then there exists a $p$-group $H$ which contains $G$ as a subgroup  such that each $\varphi_i$ extends to an inner automorphism of $H$.
\end{lemma}

\begin{proof}
For each $j$ there is a chief filtration $\mathbf G_j=\{G_{jk}\}$ of the $j$th layer $L_j(\mathbf G)$ of $\mathbf G$ which is $L_j(\varphi_i)$-invariant for each $i$ and such that
$$L_j(\varphi_i)(a)\equiv a\bmod G_{j,k+1}\qquad\text{ for all $i$, $j$, $k$ and $a\in L_j(A_i\cap\mathbf G)\cap G_{jk}$.}$$
For each $j$ let $\pi_j\colon G_j\to G_j/G_{j+1}=L_j(\mathbf G)$ be the natural epimorphism. Then $\mathbf G_j$ lifts to a complete filtration $\mathbf G^*_j=\pi_j^{-1}(\mathbf G_j)$ of $G_j$.
The $\mathbf G^*_j$ combine to a complete filtration $\mathbf G^*$ of $G$:
$$G=G_{11}^*\geq G_{12}^*\geq\cdots \geq G^*_{j,k}\geq G^*_{j,k+1}\geq\cdots\geq G^*_{j,l_j+1}=G_{j+1}\geq \cdots$$
This filtration refines $\mathbf G$ and hence is central. Moreover,
for $k\leq l_j=\text{length of $\mathbf G_j$}$, the $k$th layer of $\mathbf G_j^*$ is $L_k(\mathbf G^*_j)\cong L_k(\mathbf G_j)=\Z/p\Z$, whereas for $k>l_j$ the $k$th term of $\mathbf G_j^*$ is $G_{j+1}$, and
$$\varphi_i(a)\equiv a \bmod G^*_{j,k+1}\qquad\text{ for all $i$, $j$, $k$ and $a\in A_i\cap  G^*_{jk}$.}$$
Hence $\mathbf G^*$ is a chief filtration satisfying condition (3) in Corollary~\ref{cor:chatzidakis}.
\end{proof}

As an application of this lemma we obtain:

\begin{corollary}\label{cor:generate HNN}
Suppose $G$ is abelian, and let $H$ be a subgroup of $G$ containing all $A_i$ and $B_i$ \textup{(}$i=1,\dots,r$\textup{)}.
Then the following are equivalent:
\begin{enumerate}
\item There are a $p$-group $H^*\geq H$ and extensions of each $\varphi_i$ to an inner automorphism of $H^*$.
\item There are a $p$-group $G^*\geq G$ and extensions of each $\varphi_i$ to an inner automorphism of $G^*$.
\end{enumerate}
\end{corollary}
\begin{proof}
Clear (2) implies (1). Suppose conversely that (1) holds. Then $G\geq H\geq 1$ is a central filtration of $G$ satisfying the conditions of the previous lemma, which yields (2).
\end{proof}

This corollary immediately implies:

\begin{corollary} \label{cor:generate HNN, 2}
Suppose $G$ is abelian, and  let $\varphi_i'\colon A_i'\to B_i'$ \textup{(}$i=1,\dots,r$\textup{)} be partial automorphisms of an abelian $p$-group $G'$.
Let $\Phi\colon G\to G'$ be a morphism which, for $i=1,\dots,r$, restricts to morphisms $A_i\to A_i'$ and  $B_i\to B_i'$ making
$$
\xymatrix{A_i \ar[r]^{\varphi_i}\ar[d]^{\Phi} & B_i \ar[d]^{\Phi} \\
A_i' \ar[r]^{\varphi_i'} & B_i'
}
$$
commutative.
If $\Phi$ is injective and
$$\big\langle G, t_1,\dots,t_r : t_ia_it_i^{-1}=\varphi_i(a_i) \text{ for $i=1,\dots,r$, $a_i\in A_i$}\big\rangle$$
is residually $p$, then
$$\big\langle G', t_1,\dots,t_r : t_ia_i't_i^{-1}=\varphi_i'(a_i') \text{ for $i=1,\dots,r$, $a_i'\in A_i'$}\big\rangle$$
is residually $p$.
\end{corollary}

This fact will be used in the next chapter, however only in the trivial case where $G$ and $G'$ are elementary abelian $p$-groups. We would like to take the opportunity to stress that the analogue of the last corollary without the requirement that $\Phi$ be injective is false, as a simple example based on permutation matrices shows:

\begin{example}
Let $p=3$, and
consider $G=\F_p^3$ and its subgroups $A=\F_p\oplus\F_p\oplus 0$, $B=\F_p\oplus 0\oplus\F_p$. Let $\widetilde{\varphi}$ be the automorphism of $G$ given by
$\widetilde{\varphi}(x,y,z) = (y,z,x)$.
Then $\widetilde{\varphi}$ has order $3$, and $\widetilde{\varphi}$ restricts to an isomorphism $\varphi\colon A\to B$. Hence  the HNN extension
$$H:=\langle G,t : \text{$tat^{-1}=\varphi(a)$ for all $a\in A$}\rangle$$
of $G$ with associated subgroups $A$, $B$ identified via $\varphi$ is residually $3$.
Now consider the morphism $\Phi\colon G\to G'=\F_p^2$ given by
$\Phi(x,y,z) = (x,y+z)$.
Then $\Phi$ restricts to isomorphisms $A\to G'$ and $B\to G'$, and if we let $\varphi'$ be the automorphism of $G'$ given by
$\varphi'(u,v) = (v,u)$,
then $\varphi'\circ(\Phi|A)=\Phi\circ\varphi$. Hence $\Phi$ gives rise to a surjective morphism
$$H \to
  H' := \langle G',t : \text{$ta't^{-1}=\varphi'(a')$ for all $a'\in A'$}\rangle.$$
However, $H'$ is not residually $3$, since $\varphi'$ has order $2$.
\end{example}

This phenomenon forced us to include the condition that the morphisms $\Phi_n$ arising in the definition of a $p$-potent filtration of a group (cf.~Section~\ref{sec:p-potent}) to be \emph{injective.}

\chapter{Residual Properties of Graphs of Groups}\label{ch:residual properties of graphs of groups}

\noindent
We say that a filtration $\mathbf G=\{\mathbf G_v\}_{v\in V(Y)}$ of a graph of groups $\scrg$ with underlying graph $Y$ is \emph{$p$-excellent} if
\begin{enumerate}
\item $\mathbf G$ is normal with $[G_v:G_{vn}]<\infty$ for each $v\in V(Y)$ and $n\geq 1$;
\item $\mathbf G$, construed as a complete filtration of $\mathbf G_1$, is $p$-potent;
\item $\mathbf G$ is separating;
\item $\mathbf G$ separates the edge groups of $\scrg$; and
\item for each $e\in E(Y)$, the filtration $\mathbf G_{t(e)}$ intersects to a uniformly $p$-potent filtration on the subgroup $f_e(G_e)$ of $G_{t(e)}$.
\end{enumerate}
The main result of this chapter is:

\begin{theorem}\label{thm:reduction theorem, 2}
Let $\scrg$ be a graph of finitely generated groups admitting a $p$-excellent filtration. Then $\pi_1(\scrg)$ is virtually residually $p$.
\end{theorem}

We prove this result in the last section of this chapter. After some general facts concerning residual properties of fundamental graphs of groups (Section~\ref{sec:residual properties of graphs of groups}) we first establish a useful criterion  for fundamental groups to be residually $p$  (Section~\ref{sec:A Reduction Theorem}) and introduce an unfolding procedure on graphs of groups (Section~\ref{sec:unfolding a graph of groups}).

\index{filtration!$p$-excellent|textbf}

\section{Root properties and fundamental groups of graphs of groups}
\label{sec:residual properties of graphs of groups}

\noindent
Following \cite{Gru57}, we say that a property $\mathfrak P$ of groups is a {\it root property}\/ if
\begin{enumerate}
\item every subgroup of a group with property $\mathfrak P$ has property $\mathfrak P$;
\item every direct product of two groups with property $\mathfrak P$ has property $\mathfrak P$;
\item if $G$ is a group and $K\trianglelefteq H\trianglelefteq G$ are such that $G/H$, $H/K$ have property $\mathfrak P$, then $K$ contains a subgroup $L$, {\it normal in $G$,} such that $G/L$ is $\mathfrak P$.
\end{enumerate}
Important root properties are: solvability, finiteness, and having order a power of a given prime $p$ (see \cite{Gru57}).
(Nilpotence does not satisfy (3) and hence is not a root property.)
This notion is tailor-made for the proof of the following lemma to go through:

\begin{lemma}{\cite[Lemma~1.5]{Gru57}}\label{lem:gruenberg}
Let $\G$ be a group and $N\trianglelefteq \G$, and let $\mathfrak P$ be a root property.
If the group $N$ is residually $\mathfrak P$ and $\G/N$ has property $\mathfrak P$, then $\G$ is residually $\mathfrak P$.
\end{lemma}

\index{group!residually $\mathfrak P$}

In the rest of this section we let $\mathfrak P$ be a root property,  we let
$\scrg$ be a graph of groups based on a graph $Y$, and write $G=\pi_1(\scrg)$. We are interested in conditions which force $G$ to be residually $\mathfrak P$.
{\it In the rest of this section we assume that $\mathfrak P$ is residually satisfied by every finitely generated free group.}
Of course, the case where $\mathfrak P$ is the property of having $p$-power order is the one of most interest for us, but the extra generality comes at little additional cost, and might be useful for further applications.
Probably the most interesting example for a root property $\mathfrak P$ of the type considered here which does not include the condition of finiteness  is the property of being a solvable group. In connection with this it is worth noting that there is a distinction between ``residually solvable'' and ``residually finite solvable'': there do exist finitely presentable solvable groups which are not residually finite, cf.~\cite{Ba73}.

\medskip

\index{root property}
\index{$p$-group}

If each vertex group of $\scrg$ is a $p$-group, then $G$ always has a free (but not necessarily normal) subgroup whose index is the least common multiple of the orders of the vertex groups (hence a power of $p$), cf.~\cite[II.2.6, Lemmas~8 and 10]{Se80}; to have a free {\it normal}\/ subgroup of $p$-power index is equivalent to $G$ being residually $p$.
More generally:

\begin{lemma} \label{residual p criterion}
Suppose that each vertex group of $\scrg$ is a finite $\mathfrak P$-group. Then the following are equivalent:
\begin{enumerate}
\item $G$ has a free normal subgroup $N$ such that $G/N$ has
$\mathfrak P$.
\item $G$ is residually $\mathfrak P$.
\item There is a morphism $\psi\colon G\to P$ to a $\mathfrak P$-group $P$ such that $\psi|G_v$ is injective for every $v\in V(Y)$.
\end{enumerate}
\end{lemma}

\begin{proof}
The implication ``(1)~$\Rightarrow$~(2)'' follows from Lemma~\ref{lem:gruenberg} and the assumption that free groups are residually $\mathfrak P$. Suppose $G$ is residually $\mathfrak P$. Then for every $v\in V(Y)$ and $g\in G_v\setminus\{1\}$ there is a normal subgroup $N_g$ of $G$  with $G/N_g$ satisfying~$\mathfrak P$ and $g\notin N_g$. Let $N$ be the intersection of all $N_g$ as $g$ ranges over $\bigcup_{v\in V(Y)} G_v\setminus\{1\}$. Then $N$ is a normal subgroup of $G$ with $G/N$ satisfying~$\mathfrak P$, and the natural morphism $\psi\colon G\to G/N$ is injective on each $G_v$. This shows ``(2)~$\Rightarrow$~(3).''
Finally, given $\psi\colon G\to P$ as in (3), the kernel $N$ of $\psi$ is a normal subgroup of $G$ with $G/N$ satisfying $\mathfrak P$, and $N$ is free by \cite[II.2.6, Lemma~8]{Se80}. This shows ``(3)~$\Rightarrow$~(1).''
\end{proof}

\begin{remark}
The proof of the implications ``(3)~$\Rightarrow$~(1)~$\Rightarrow$~(2)'' did not use the finiteness of the $G_v$, and the proof of ``(2)~$\Rightarrow$~(3)~$\Rightarrow$~(1)'' is valid without assuming that free groups are residually $\mathfrak P$.
\end{remark}

\begin{corollary}\label{cor:residually P quotient}
Let $H$ be a normal subgroup of $G$ such that $G/H$ satisfies $\mathfrak P$. Then $\pi_1(\scrg/\scrh)$ is residually $\mathfrak P$, where $\scrh=\{H\cap G_v\}_{v\in V(Y)}$.
\end{corollary}
\begin{proof}
The  unique morphism $\pi_1(\scrg/\scrh)\to G/H$ making \eqref{eq:factorization} commutative is injective on each vertex group $G_v/H_v$ of $\scrg/\scrh$, so $\pi_1(\scrg/\scrh)$ is residually $\mathfrak P$ by
the implication ``(3)~$\Rightarrow$~(2)'' in Lemma~\ref{residual p criterion} (the proof of which didn't need finiteness of the $G_v/H_v$).
\end{proof}

\subsection*{Digression: trees of elementary abelian $p$-groups}\label{sec:homogeneous}
Suppose the underlying graph $Y$ of $\scrg$ is a tree, and fix a vertex $v_0$ of $Y$.
Assume that for each edge $e$ of $Y$, the morphism $f_{e}\circ f_{\ol{e}}^{-1}$ extends to an isomorphism $\varphi_e\colon G_{o(e)}\to G_{t(e)}$, such that $\varphi_{e}^{-1}=\varphi_{\ol{e}}$ for all $e$.
For each vertex $v$ of $Y$ we then have an isomorphism $\varphi_v\colon G_v\to G_{v_0}$ given as the composition
$$\varphi_v := \varphi_{e_n}\circ \varphi_{e_{n-1}} \circ \cdots \circ \varphi_{e_1}$$
where $(e_1,\dots,e_n)$ is the geodesic in $Y$ from $v$ to $v_0$. In particular, $\varphi_{v_0}=\id_{G_{v_0}}$, and
$\varphi_{t(e)}\circ f_e = \varphi_{t(\ol{e})}\circ f_{\ol{e}}$ for every edge $e$ of $Y$. This yields a simple but useful fact about the iterated amalgamated product $G=\pi_1(\scrg)$:

\begin{lemma} \label{lem:special amalgam}
There exists a unique morphism $\varphi\colon G\to G_{v_0}$ such that for each vertex $v$ of $Y$, the diagram
$$\xymatrix{
G \ar[r]^{\varphi} & G_{v_0} \\
G_v\ar[u] \ar[ur]_{\varphi_v} &
}$$
commutes, where $G_v\to G$ is the natural morphism. 
\end{lemma}

The following application  (not used later) of Lemmas~\ref{residual p criterion} and \ref{lem:special amalgam} is implicit in \cite{CM05}. A finite group $\G$ is homogeneous (in the sense of model theory) if every isomorphism between subgroups of $\G$ is induced by an automorphism of $\G$. Finite homogeneous groups have been completely classified \cite{CF00}; for example,
given a prime $p$, the
group of the form $\Z/p^k\Z \oplus \cdots \oplus \Z/p^k\Z$ (for some $k$) are homogeneous, and if $p$ is odd, then these are the only homogeneous $p$-groups. (There are two non-abelian homogeneous $2$-groups.)

\begin{corollary}
Suppose $Y$ is a tree, every vertex group $G_v$ is a finite homogeneous $\mathfrak P$-group, and $G_v\cong G_w$ for all $v,w\in V(Y)$.
Then $G$ has a free normal subgroup $N$ with $G/N\cong G_v$, hence $G$ is residually $\mathfrak P$.
\end{corollary}

\index{group!homogeneous|textbf}

\subsection{Hempel's criterion}
The following is a generalization of the criterion for residual finiteness of fundamental groups of graphs of groups found in \cite[Theorem~3.1]{He87} and, with a different proof, in \cite{Sh87} (see also \cite[Section~II.2.6]{Se80}).
It unifies and generalizes several criteria for groups to be residually $p$ in the literature. (For example, it can be used to easily deduce \cite[Corollary~3.5, Theorems~4.2 and 4.3]{KM93}.)

\begin{proposition} \label{Hempel criterion residually p}
Suppose that  the following hold:
\begin{itemize}
\item[(A)] For each $v_0\in V(Y)$ and each $g\in G_{v_0}\setminus \{1\}$ there is a compatible collection $\scrh=\{H_v\}_{v\in V(Y)}$ of normal subgroups such that $g\notin H_{v_0}$ and $\pi_1(\scrg/\scrh)$  is residually $\mathfrak P$.
\item[(B)] For all $e_1,\dots,e_n\in E(Y)$ and $g_i\in G_{t(e_i)}\setminus f_{e_i}(G_{e_i})$, $i=1,\dots,n$ there exists a compatible collection $\scrh=\{H_v\}_{v\in V(Y)}$ of normal subgroups such that $g_i\notin H_{t(e_i)}\cdot f_{e_i}(G_{e_i})$ for $i=1,\dots,n$ and $\pi_1(\scrg/\scrh)$ is residually $\mathfrak P$. 
\end{itemize}
Then $G$ is residually $\mathfrak P$.
\end{proposition}

\index{Hempel}
\index{group!residually $\mathfrak P$}


\begin{proof}
In this proof $\scrh$ ranges over all compatible  collections  of normal subgroups such that $\pi_1(\scrg/\scrh)$ is residually $\mathfrak P$. The natural morphisms $\pi_{\scrh}\colon \pi_1(\scrg)\to\pi_1(\scrg/\scrh)$ combine to a morphism
$$\pi\colon G\to G^*:=\prod_{\scrh} \pi_1(\scrg/\scrh).$$
Since $G^*$ is residually $\mathfrak P$, it is enough to show that $\pi$ is injective. Let $g\neq 1$ be an element of $G$, and represent $g$ by a path
$$\gamma=(g_0, e_1, g_1, e_2, \dots, e_n, g_n)$$
in $\scrg$ as in Section~\ref{sec:pi1}. Then $\pi_{\scrh}(g)$ is represented by the path $$\pi_{\scrh}(\gamma):=(\pi_{\scrh}(g_0), e_1, \pi_{\scrh}(g_1), e_2, \dots, e_n, \pi_{\scrh}(g_n))$$ in $\scrg/\scrh$.  Moreover, if $\gamma$ is reduced, then assumptions (A) and (B) 
yield that for some~$\scrh$, the path $\pi_{\scrh}(\gamma)$ is reduced as well, in particular $\pi_{\scrh}(g)\neq 1$. Therefore we have $\pi(g)\neq 1$.
\end{proof}

We do not know whether in general we can simplify condition (B) by requiring this condition to hold only for a single edge at a time (i.e., for $n=1$).
Next, we indicate a few situations where this is possible. The first one concerns the case where we can choose all subgroups in $\mathcal H$ to have finite index:

\index{topology!pro-$\mathfrak P$}

\begin{corollary}\label{cor:vertex and edge groups closed}
Suppose that  the following hold:
\begin{itemize}
\item[(A$'$)] For each $v_0\in V(Y)$ and each $g\in G_{v_0}\setminus \{1\}$ there is a compatible collection $\scrh=\{H_v\}_{v\in V(Y)}$ of finite index normal subgroups such that $g\notin H_{v_0}$ and $\pi_1(\scrg/\scrh)$  is residually $\mathfrak P$.
\item[(B$'$)] For all $e_0\in E(Y)$ and $g\in G_{t(e_0)}\setminus f_{e_0}(G_{e_0})$ there exists a compatible collection $\scrh=\{H_v\}_{v\in V(Y)}$ of finite index normal subgroups such that $g\notin H_{t(e_0)}\cdot f_{e_0}(G_{e_0})$  and $\pi_1(\scrg/\scrh)$ is residually $\mathfrak P$.
\end{itemize}
Then $G$ is residually $\mathfrak P$. Moreover, every vertex and edge group of $\scrg$ \textup{(}identified with a subgroup of $G$ in the natural way\textup{)} is closed in the pro-$\mathfrak P$ topology on $G$.
\end{corollary}

Recall here that the pro-$\mathfrak P$ topology  on a group $\G$ is the topology with fundamental system of neighborhoods of $1$ given by the normal subgroups $N\trianglelefteq \G$ with $\G/N$ a $\mathfrak P$-group.

\medskip

The first statement in this corollary is a consequence of Proposition~\ref{Hempel criterion residually p} and the following lemma, for which we let $\scrh_i=\{H_{iv}\}_{v\in V(Y)}$ \textup{(}$i=1,2$\textup{)} be compatible collections  of normal subgroups, and set
$$\scrh_1\cap\scrh_2:=\{H_{1v}\cap H_{2v}\}_{v\in V(Y)}.$$
Then $\scrh_1\cap\scrh_2$ is a compatible collection of normal subgroups, and we have a natural group morphism
\begin{equation}\label{eq:H1H2}
\pi_1(\scrg/(\scrh_1\cap\scrh_2))\to \pi_1(\scrg/\scrh_1)\times \pi_1(\scrg/\scrh_2)
\end{equation}
which is injective on vertex groups.

\begin{lemma}\label{lem:intersect H}
Suppose $G_v/H_{iv}$ is finite for $i=1,2$ and every vertex $v$.
If $\pi_1(\scrg/\scrh_1)$  and $\pi_1(\scrg/\scrh_2)$ are residually $\mathfrak P$, then so is $\pi_1(\scrg/(\scrh_1\cap\scrh_2))$.
\end{lemma}
\begin{proof}
Suppose $\pi_1(\scrg/\scrh_1)$  and $\pi_1(\scrg/\scrh_2)$ are residually $\mathfrak P$.
Since we have $[G_v:H_{iv}]<\infty$ for $i=1,2$ and all $v$, by Lemma~\ref{residual p criterion} there exist group morphisms $\psi_i\colon\pi_1(\scrg/\scrh_i)\to P_i$ to  $\mathfrak P$-groups $P_i$ such that $\psi_i|(G_v/H_{iv})$ is injective, for $i=1,2$ and all $v$. Composing with the morphism \eqref{eq:H1H2},
we obtain a group morphism $\pi_1(\scrg/(\scrh_1\cap\scrh_2))\to P_1\times P_2$ which is injective on vertex groups, and this shows (again by Lemma~\ref{residual p criterion})  that $\pi_1(\scrg/(\scrh_1\cap\scrh_2))$ is residually $\mathfrak P$.
\end{proof}

For showing the second statement in Corollary~\ref{cor:vertex and edge groups closed} we use the following simple observation (the proof of which we skip):

\begin{lemma}\label{lem:separate finite subgroups}
Suppose $\G$ is a residually $\mathfrak P$ group, and let $\Delta$ be a finite subgroup of $\G$. Then for every $\gamma\in\G$ with $\gamma\notin\Delta$ there exists a morphism $\phi\colon\G\to P$ to a $\mathfrak P$-group such that $\phi(\gamma)\notin \phi(\Delta)$.
\end{lemma}

\begin{proof}[Proof of Corollary~\ref{cor:vertex and edge groups closed}]
As remarked, the first statement in this corollary follows from Proposition~\ref{Hempel criterion residually p} and Lemma~\ref{lem:intersect H}.
Let $g\in G\setminus G_{v_0}$, $v_0\in V(Y)$. By the proof of Proposition~\ref{Hempel criterion residually p} there exists a compatible collection $\scrh$ of finite index normal subgroups such that $\pi_1(\scrg/\scrh)$ is residually $\mathfrak P$ and
$\pi_\scrh(g)\notin \pi_\scrh(G_{v_0})$. Since $G_{v_0}/H_{v_0}$ is finite, there exists a morphism $\phi\colon\pi_1(\scrg/\scrh)\to P$ to a $\mathfrak P$-group
such that $\phi(\pi_\scrh(g))\notin\phi(G_{v_0}/H_{v_0})=\phi(\pi_\scrh(G_{v_0}))$, by Lemma~\ref{lem:separate finite subgroups}. Hence $G_{v_0}$ is closed in $G$.
Similarly ones sees that each edge group of $\scrg$ is closed in $G$.
\end{proof}

If all the edge groups $G_e$ of $\scrg$ are trivial, then every collection  $\{H_v\}_{v\in V(Y)}$ of subgroups $H_v\leq G_v$ of the vertex groups is compatible, and  if
$\scrh_i=\{H_{iv}\}_{v\in V(Y)}$ ($i=1,2$) are collections  of normal subgroups, then the morphism \eqref{eq:H1H2} is injective; hence if in addition
the fundamental group of $\scrg/\scrh_i$ is residually $\mathfrak P$ for $i=1,2$, then so is the fundamental group of  $\scrg/(\scrh_1\cap\scrh_2)$. So in this special case, Proposition~\ref{Hempel criterion residually p} also simplifies:

\begin{corollary}
Suppose that $G_e=1$ for every $e\in E(Y)$. If $\scrg$ satisfies condition \textup{(A)} in the previous proposition, then $G$ is residually $\mathfrak P$.
\end{corollary}

We record two (presumably well-known) consequences: 

\begin{corollary}\label{cor:trivial edge groups}
The fundamental group of every graph of groups whose vertex groups are residually $\mathfrak P$ and whose edge groups are trivial is residually $\mathfrak P$.
\end{corollary}
\begin{proof}
By the previous corollary, it is enough to show that if $v_0\in V(Y)$ is such that
$G_{v_0}$ is a $\mathfrak P$-group and $G_v=1$ for all $v\in V(Y)$ with $v\neq v_0$, and $G_e=1$ for all $e\in E(Y)$, then $G$ is residually $\mathfrak P$.  However, this is immediate by the equivalence of (3) and (2) in Lemma~\ref{residual p criterion}.
\end{proof}

In particular, the previous corollary shows that the free product of finitely many residually $p$ groups is residually $p$, and (combined with Theorem~\ref{thm:bass-serre}) that the free product of finitely many virtually residually $p$ groups is virtually residually $p$.

\begin{corollary}\label{cor:crucial, preparation}
Assume that all vertex groups of $\scrg$ are residually $\mathfrak P$, and suppose there exists a morphism $\psi\colon G=\pi_1(\scrg)\to P$   to a finite $\mathfrak P$-group $P$ which is injective when restricted to each subgroup $f_e(G_e)$ of $G$. Then $G$ is residually $\mathfrak P$.
\end{corollary}
\begin{proof}
By Theorem~\ref{thm:bass-serre}, the kernel $N$ of $\psi$ is the fundamental group of a graph of groups whose vertex groups are of the form $N\cap G_v$, and whose edge groups are of the form $N\cap f_e(G_e)$. However, since $\psi|f_e(G_e)$ is injective, we have $N\cap f_e(G_e)=\{1\}$ for each $e$. Thus $N$ is the fundamental group of a graph of groups with  residually $\mathfrak P$ vertex groups  and trivial edge groups. Hence $N$ is residually $\mathfrak P$ by Corollary~\ref{cor:trivial edge groups}, and so $G$ is residually $\mathfrak P$ by Lemma~\ref{lem:gruenberg}.
\end{proof}

Finally, we note that
in later sections, we will apply Proposition~\ref{Hempel criterion residually p} with the~$\scrh$ coming from a filtration of $\scrg$; in this case, this proposition reads as follows:

\begin{corollary}\label{cor:Hempel, cor}
Let $\mathbf G$ be a normal separating filtration of $\scrg$ which separates the edge groups of $\scrg$, such that
$\pi_1(\scrg/\mathbf G_n)$ is residually $\mathfrak P$ for every $n\geq 1$. Then $G$ is residually $\mathfrak P$.
\end{corollary}



\subsection*{Digression: necessity of conditions (A) and (B)}
(The material in the rest of this section is not used later.)
It is interesting to ask whether the conditions (A) and (B) in the previous proposition are necessary for $G$ to be residually $\mathfrak P$. From Corollary~\ref{cor:residually P quotient} we immediately obtain:

\begin{lemma}\label{lem:A}
If $G$ is residually $\mathfrak P$, then $\scrg$ satisfies \textup{(A)}.
\end{lemma}

The situation for condition (B) is slightly less clear. Here we can only show:

\begin{lemma}\label{lem:B}
Suppose for each edge $e\in E(Y)$, either
\begin{enumerate}
\item the subgroup $f_e(G_e)$ of $G_{t(e)}$  is maximal with respect to being verbal with respect to a collection $W_e$ of words; or
\item $G_{t(\overline{e})}$ is abelian and $f_e(G_e)$, $f_{\overline{e}}(G_e)$ are proper subgroups of $G_{t(e)}$, $G_{t(\overline{e})}$, respectively.
\end{enumerate}
If $G$ is residually $\mathfrak P$, then $\scrg$ satisfies \textup{(B)}.
\end{lemma}

In this context, recall the following definition:

\begin{defn}
Let $W$ be a collection of elements of a free group $F$ (``words''). A group $\G$ is {\it verbal} with respect to $W$ if every morphism $F\to \G$ maps each $w\in W$ to $1$. For example, a group $\G$ is verbal with respect to $W=\{xyx^{-1}y^{-1}\}$ if and only if $\G$ is abelian.
\end{defn}

\index{group!verbal}

We have the following generalization of a lemma from \cite{LN91}:

\begin{lemma}\label{lem:verbal}
Suppose $\G$ is a  group and $H$ is a subgroup of $\G$. Assume that $\G$ is residually $\mathfrak P$ and  $H$ is  verbal with respect to a collection of words $W$. Then
$$H^* := \bigcap \big\{ H\cdot K : \text{$K\trianglelefteq \G$ and  $\G/K$ has property $\mathfrak P$}\big\}$$
is also verbal with respect to $W$.
\end{lemma}
\begin{proof}
For a normal subgroup $K$ of $\G$, every substitution of elements of $H\cdot K$ into a word from $W$ yields an element of $K$ (since such an element maps to $1$ under $\G\to \G/K$).
Since $\G$ is residually $\mathfrak P$, the intersection of all $K\trianglelefteq \G$ such that  $\G/K$ has property $\mathfrak P$ is trivial. Hence $H^*$ is verbal with respect to $W$.
\end{proof}

We now prove Lemma~\ref{lem:B}. So suppose $G=\pi_1(\scrg)$ is residually $\mathfrak P$. Let $e$ be an edge of $Y$, and $g\in G_{v}\setminus f_e(G_e)$, where $v=t(e)$. In order to establish (B) it suffices to show: there is $K\trianglelefteq G$ such that $G/K$ is a $\mathfrak P$-group and $g\notin f_e(G_e)\cdot (K\cap G_v)$. To see this note that if $K_i\trianglelefteq G$ are such that $G/K_i$ is a $\mathfrak P$-group, $i=1,2$, then $G/(K_1\cap K_2)$ is also a $\mathfrak P$-group, and use Corollary~\ref{cor:residually P quotient}

Assume first that $f_e(G_e)$ is a maximal subgroup of $G_{v}$   with respect to being verbal with respect to some $W_e$.  Considering $H:=f_e(G_e)$ as a subgroup of $G$, we see that $H^*$ as defined in Lemma~\ref{lem:verbal} is verbal with respect to $W$, hence $H^*\cap G_{v}=H$. Thus there is some $K\trianglelefteq G$ such that $G/K$ has property $\mathfrak P$ and $g\notin H\cdot (K\cap G_v)$.

Suppose next that $G_w$ ($w=o(e)=t(\overline{e})$) is abelian and $f_e(G_e)$, $f_{\overline{e}}(G_e)$ are proper subgroups of $G_{v}$, $G_{w}$, respectively. Choose $h\in G_w\setminus f_{\overline{e}}(G_e)$. We may assume that our base point $v_0$ for $G=\pi_1(\scrg,v_0)$ agrees with $w$. Then $u:=[ege^{-1},h]=eg^{-1}\overline{e} h^{-1} e g \overline{e} h$ is an element of $G$ which is not the identity. Hence there is some $K\trianglelefteq G$ such that $G/K$ has property $\mathfrak P$ and $u\notin K$.  This yields $g\notin f_e(G_e)\cdot (K\cap G_v)$:
otherwise $g \equiv f_e(g_e) \bmod K$ for some $g_e\in G_e$, so $ege^{-1} \equiv f_{\ol{e}}(g_e) \bmod K$ and hence $u \equiv [f_{\ol{e}}(g_e) , h] \equiv 1 \bmod K$ since $G_w$ is abelian, a contradiction.

In both cases we found $K\trianglelefteq G$ such that $G/K$ is a $\mathfrak P$-group and $g\notin f_e(G_e)\cdot (K\cap G_v)$ as required. \qed

\medskip

Given an orientable $3$-manifold $N$ with incompressible toroidal boundary,  the fundamental groups of the boundary components of $N$ are maximal abelian subgroups  of $\pi_1(N)$; cf.~Lemma~\ref{lem:maximal abelian} below. This fact in combination with Lemmas~\ref{lem:A} and \ref{lem:B},~(1) shows that in proving that the fundamental group of a $3$-manifold is residually $\mathfrak P$, we essentially have to use Proposition~\ref{Hempel criterion residually p}:

\begin{corollary}
If $N$ is $3$-manifold  which is closed, orientable, and prime,  with non-trivial JSJ decomposition, then $\pi_1(N)$ is residually $\mathfrak P$ if and only if the graph of groups $\scrg$ associated to the JSJ decomposition of $N$ satisfies conditions \textup{(A)} and \textup{(B)}.
\end{corollary}

(See Section~\ref{sec:JSJ} for the definition of the graph of groups $\scrg$.)

\section{A criterion for being residually $p$}\label{sec:A Reduction Theorem}

\noindent
We now would like to apply the results of the previous chapter to establish a criterion for fundamental groups of certain graphs of groups to be residually $p$. This is Theorem~\ref{thm:reduction theorem} at the end of this section. Before we can formulate this theorem, we need to introduce a generalization of the fiber sum construction, which is then used to define the ``partial abelianization'' of a fundamental group of a graph of abelian groups.

\subsection{The fiber sum of a graph of abelian groups}\label{sec:fiber sum}
Let $\mathcal G$ be a graph of abelian groups, and let $E_+$ be an orientation of its underlying graph $Y$. The injective morphisms
$$G_e \to G_{t(e)}\oplus G_{t(\ol{e})}\colon g\mapsto f_e(g)-f_{\ol{e}}(g)$$
combine to a natural morphism
$$\bigoplus_{e\in E_+} G_e \to \bigoplus_{v\in V(Y)} G_v$$
whose cokernel we call the {\it colimit}\/ $\Sigma(\scrg)$ of $\scrg$.
(Indeed, $\Sigma(\scrg)$ may be seen as the colimit, in the sense of category theory, of a certain diagram  of abelian groups.)
The colimit of $\scrg$ is determined, up to isomorphism, by the following universal property: for every abelian group $H$ and every collection $\{\psi_v\colon G_v\to H\}_{v\in V(Y)}$ of morphisms with the property that for each edge $e\in E_+$, the diagram in Figure~\ref{fig:colimit}
commutes, there exists a unique morphism $\psi\colon\Sigma(\scrg)\to H$ such that $\psi_{v} = \psi\circ\iota_v$ for all $v\in V(Y)$, where $\iota_v\colon G_v\to \Sigma(\scrg)$ is the natural morphism. Also note that $\Sigma(\scrg)$ does not depend on the choice of orientation $E_+$ of $Y$.

\begin{figure}
\[\xymatrix{
&     H                & \\
& \Sigma(\scrg) \ar@{.>}[u]^\psi & \\
G_{t(e)} \ar[ur]_{\iota_{t(e)}}
\ar@/^/[uur]^{\psi_{t(e)}}
 & & G_{t(\ol{e})} \ar[ul]^{\iota_{t(\ol{e})}} \ar@/_/[uul]_{\psi_{t(\ol{e})}} \\
& \ar[ul]^{f_e} G_e\ar[ur]_{f_{\ol{e}}} &
}\]
\caption{Universal property of $\Sigma(\scrg)$.}
\label{fig:colimit}
\end{figure}

If $Y$ is a tree, we also speak of the {\it fiber sum}\/ $\Sigma(\scrg)$ of $\scrg$. If $Y$ has exactly two vertices $v_0$, $v_1$ and one topological edge $\{e,\ol{e}\}$, then $\Sigma(\scrg)$ is indeed the fiber sum of $G_{v_0}$ and $G_{v_1}$ over $G_e$ as introduced earlier. More generally, if $v_0$ is a terminal vertex of $Y$ (i.e., there is only one edge $e$ with $t(e)=v_0$), then in a natural way
$$\Sigma(\mathcal G)= G_{v_0} \oplus_{G_e} \Sigma(\scrg|Y'),$$ where $e$ is the edge of $Y$ with $t(e)=v_0$, and $Y'=Y\setminus\{v_0\}$. By induction on the number of vertices of $Y$, this yields in particular:

\index{graph of groups!fiber sum of}

\begin{lemma}\label{lem:injection into fiber sum}
Suppose $Y$ is a tree. Then for each $v\in V(Y)$, the natural morphism $G_v\to\Sigma(\scrg)$ is injective.
\end{lemma}

If $Y$ is a tree, then the fiber sum of $\scrg$ may also be identified with the abelianization  of the fundamental group $G$ of $\scrg$; more generally:

\begin{lemma}\label{lem:homology and fiber sums}
Let $T$ be a maximal subtree of $Y$, and $E:=E(Y)\setminus E(T)$. Suppose $\iota_{t(e)}\circ f_e=\iota_{o(e)}\circ f_{\ol{e}}$ for all $e\in E$, where $\iota_v\colon G_v\to\Sigma(\scrg|T)$, $v\in V(Y)$, is the natural morphism. Then there is a natural isomorphism
$$G_{\operatorname{ab}}=H_1(G;\Z) \xrightarrow{\cong} \Sigma(\scrg|T)\oplus\Z^{e} \qquad (e=|E|/2).$$
If for some prime $p$, all $G_v$ are elementary abelian $p$-groups, then we have $$H_1(G;\F_p) \cong \Sigma(\scrg|T)\oplus\F_p^e.$$
\end{lemma}

This follows easily from the Mayer-Vietoris sequence associated to a graph of groups (cf.~Section~\ref{sec:MV}). 


\begin{remark}
Note that under the assumption of the previous lemma, for each vertex $v$, the natural morphism $G_v\to G$ induces an injective morphism $G_v=H_1(G_v;\Z)\to H_1(G;\Z)$. This conclusion can also be achieved under a weaker hypothesis on $\scrg$: it is enough to require that  there exists a morphism $G\to A$ to an abelian group which is injective on the vertex groups of $\scrg$ (since such a morphism factors through $G\to H_1(G;\Z)$). Similarly, if there exists a morphism $G\to A$ to an elementary abelian $p$-group which is injective on the vertex groups of $\scrg$, then the induced morphisms $H_1(G_v;\F_p)\to H_1(G;\F_p)$ are injective.
\end{remark}

\subsection{Partial abelianization}\label{sec:partial ab}
Let $\scrg$ be a graph of groups based on the graph $Y$.
Let $T$ be a maximal subtree of $Y$, let $E=E(Y)\setminus E(T)$, and let $E_+$ be an orientation of $E$.
From Section~\ref{sec:pi1} recall that the fundamental group $G$ of $\scrg$ may be identified with the iterated HNN extension
$$\pi_1(\scrg,T) = \big\langle \pi_1(\scrg|T),\ e\in E_+: \text{$e a e^{-1} = \varphi_e(a)$ for all $e\in E_+$, $a\in A_e$}\big\rangle$$
of $\pi_1(\scrg|T)$, where
$$\varphi_e=f_{\ol{e}}\circ f_e^{-1}\colon A_e:=f_e(G_e) \to B_e:=f_{\ol{e}}(G_{\ol{e}}) \qquad (e\in E_+).$$
Suppose now that all vertex groups of $\scrg$ are abelian.
We let
$$\pi_1^*(\scrg,T) = \big\langle \Sigma(\scrg|T),\ e\in E_+: \text{$e a e^{-1} = \varphi_e(a)$ for all $e\in E$, $a\in A_e$}\big\rangle.$$
The natural morphism $$\pi_1(\scrg|T)\to H_1(\pi_1(\scrg|T))=\Sigma(\scrg|T)$$ extends to a morphism $G=\pi_1(\scrg,T)\to \pi_1^*(\scrg,T)$, which is injective on the vertex groups of $\scrg$. The natural morphism
$G\to G_{\operatorname{ab}}$ factors through $G\to \pi_1^*(\scrg,T)$. For this reason we call $\pi_1^*(\scrg,T)$ the \emph{partial abelianization} of $G$ along $T$. Note that although $\pi_1(\scrg,T)$ is independent (up to isomorphism) of the choice of $T$, simple examples show that the group $\pi_1^*(\scrg,T)$ in general depends on $T$. 

\index{partial abelianization|textbf}
\index{root property}

In the following let $\mathfrak P$ be a root property such that free groups are residually~$\mathfrak P$.
By Lemma~\ref{residual p criterion}, we have:

\begin{lemma}\label{lem:residually p for partial abelianization}
Suppose all vertex groups of $\scrg$ are finite abelian groups. If $\pi_1^*(\scrg,T)$ is residually $\mathfrak P$, then
$G=\pi_1(\scrg,T)$ is also residually $\mathfrak P$.
\end{lemma}


We do not know whether the converse of the implication stated in the previous lemma holds in general; however, we do have:

\begin{lemma}\label{lem:crucial}
Suppose there exists a morphism $G\to P$ to a finite abelian $\mathfrak P$-group $P$, which is injective when restricted to each vertex group $G_v$ of $\scrg$. Then $\pi_1^*(\scrg,T)$ is residually $\mathfrak P$.
\end{lemma}

This is a consequence of Corollary~\ref{cor:crucial, preparation} and the following observation:

\begin{lemma}
Let $\psi\colon G\to P$ be a morphism to an abelian group $P$.
Then there is a morphism $\pi_1^*(\scrg,T)\to P$ which agrees with $\psi$ on each vertex group of $\scrg$ and such that $e\mapsto 1$ for $e\in E_+$.
\end{lemma}
\begin{proof}
By the universal property of the fiber sum, the restriction of $\psi$ to a morphism $\pi_1(\scrg|T)\to P$ factors through a morphism $\Sigma(\scrg|T)\to P$. Since $P$ is abelian, this morphism extends to a morphism $\pi_1^*(\scrg,T)\to P$ by $e\mapsto 1$ for $e\in E_+$ which has the required properties.
\end{proof}

\subsection{The reduction theorem}\label{subsec:reduction theorem}
Let $\scrg$ be a graph of finitely generated groups with underlying graph $Y$, and let $\mathbf G$ be a central $p$-filtration of $\scrg$. Recall that the $n$th layer $L_n(\mathbf G)$ of $\mathbf G$ is the graph of groups whose underlying graph is also~$Y$ and whose vertex and edge groups are the elementary abelian $p$-groups $L_n(\mathbf G_v)$ and~$L_n(\mathbf G_e)$, respectively, with edge morphisms $L_n(f_e)$.  Fix a maximal subtree $T$ of $Y$. For every $n\geq 1$ we can form the partial abelianization
$\pi_1^*(L_n(\mathbf G),T)$ of the fundamental group of  $L_n(\mathbf G)$ along $T$.

The following important theorem is the culmination of the results obtained so far in this paper. Essentially, it reduces the question whether $\pi_1(\scrg)$ is residually $p$ to finding a filtration $\mathbf G$ such that the partially abelianized fundamental group of the first layer $L_1(\mathbf G)$ of $\mathbf G$ is residually $p$.

\begin{theorem}\label{thm:reduction theorem}
Suppose the following hold:
\begin{enumerate}
\item $\mathbf G$ is $p$-potent;
\item for each $e$, $\mathbf G_{t(e)}$
intersects to a uniformly $p$-potent filtration on $f_e(G_e)$;
\item $\pi_1^*(L_1(\mathbf G),T)$ is residually $p$.
\end{enumerate}
Then  $\pi_1(\scrg/\mathbf G_n)$ is residually $p$ for each $n\geq 1$.
If in addition $\mathbf G$ is separating and separates the edge groups of $\scrg$,
then $G=\pi_1(\scrg)$ is residually $p$, and the vertex and edge groups of $\scrg$ are closed in the pro-$p$ topology on $\pi_1(\scrg)$.
\end{theorem}

For the proof of this theorem we first show that its hypotheses imply that the partial abelianization of each layer of $\mathbf G$ is residually $p$ (not only the first, as stipulated by condition (3)). Below, we continue to use the notation introduced in Section~\ref{sec:partial ab} above.
For every edge $e\in E=E(Y)\sm E(T)$ we equip
$A_e=f_e(G_e)$ and $B_e=f_{\ol{e}}(G_{e})$ with filtrations
$\mathbf A_e:=A_e\cap\mathbf G_{t(e)}$ respectively $\mathbf B_e:=B_e\cap\mathbf G_{t(\ol{e})}$, and let
$$L_n(\varphi_{e})\colon L_n(\mathbf A_e) \to L_n(\mathbf B_e)$$
be the isomorphism induced by $\varphi_e=f_{\ol{e}}\circ f_{e}^{-1}$.

\begin{lemma}\label{lem:reduction theorem}
Suppose $\mathbf G$ satisfies conditions \textup{(1)--(3)} in Theorem~\ref{thm:reduction theorem}.
Then for each $n\geq 1$, the group $\pi_1^*(L_n(\mathbf G),T)$ is residually $p$.
\end{lemma}
\begin{proof}
By (1) and (2) we can take injective $\F_p$-linear maps
$$\Phi_{v,n}\colon L_1(\mathbf G_v) \to L_{n+1}(\mathbf G_v)\qquad (v\in V(Y))$$
such that  for each edge $e$ with $t(e)=v$, $\Phi_{v,n}$ restricts to isomorphisms $L_1(\mathbf A_e)\to L_{n+1}(\mathbf A_e)$ and $L_1(\mathbf B_e)\to L_{n+1}(\mathbf B_e)$, and the diagram
\[\xymatrix{
L_{n+1}(\mathbf A_e) \ar[rr]^{L_{n+1}(\varphi_{e})}& &  L_{n+1}(\mathbf B_e)\\
L_1(\mathbf A_e) \ar[u]^\cong\ar[rr]^{L_1(\varphi_{e})}& & L_1(\mathbf B_e) \ar[u]^\cong
}\]
commutes. The $\Phi_{v,n+1}$ give rise to an {\it injective}\/ $\F_p$-linear map
$$\Sigma(L_1(\mathbf G)|T)\to\Sigma(L_{n+1}(\mathbf G)|T)$$
which restricts to isomorphisms $L_1(\mathbf A_e)\to L_{n+1}(\mathbf A_e)$
and $L_1(\mathbf B_e)\to L_{n+1}(\mathbf B_e)$ (after identifying $L_i(\mathbf A_e)$ and $L_i(\mathbf B_e)$ with subgroups of $L_i(\Sigma(\mathbf G|T))$ as usual).
This morphism extends to a morphism $$\pi_1^*(L_1(\mathbf G),T)\to \pi_1^*(L_{n+1}(\mathbf G),T)$$
by the identity on $E_+$.
Therefore, since the group $\pi_1^*(L_1(\mathbf G),T)$ is residually $p$ (by assumption (3)),
so is $\pi_1^*(L_n(\mathbf G),T)$ for each $n > 1$, by Corollary~\ref{cor:generate HNN, 2}.
\end{proof}

We will deduce Theorem~\ref{thm:reduction theorem} from the following proposition:

\begin{proposition} \mbox{} \label{prop:reduction}
Assume that each vertex group  of $\scrg$ is a $p$-group and $\mathbf G$ is of finite length. Suppose moreover that $\pi_1^*(L_n(\mathbf G),T)$ is  residually $p$ for every~$n\geq 1$. Then $G$ is residually $p$.
\end{proposition}

Let us see how this proposition yields Theorem~\ref{thm:reduction theorem}.
So suppose conditions (1)--(3) hold.
By Lemma~\ref{lem:reduction theorem},  the groups $\pi_1^*(L_n(\mathbf G),T)$ are residually $p$, for each~$n\geq 1$. By the proposition applied to $\scrg/\mathbf G_n$ in place of $\scrg$ and the filtration of $\scrg/\mathbf G_n$ of length $<n$ given on vertex groups by
$$G_{v1}/G_{vn}\geq G_{v2}/G_{vn} \geq \cdots\geq G_{v,n-1}/G_{vn}\geq 1,$$  we obtain that $\pi_1(\scrg/\mathbf G_n)$ is residually $p$, for each $n\geq 1$. This is the first conclusion of Theorem~\ref{thm:reduction theorem}.
The rest of the theorem follows from this and
Corollaries~\ref{cor:Hempel, cor} and \ref{cor:vertex and edge groups closed} of Proposition~\ref{Hempel criterion residually p}.

\medskip

It remains to give a proof of  Proposition~\ref{prop:reduction}.
For this, we need a consequence of the main results of Section~\ref{sec:An Embedding Theorem}:

\begin{lemma}\label{lem:higman, 2}
Suppose $Y$ is a tree, all vertex groups of $\scrg$ are $p$-groups, and  $\mathbf G$ has finite length. Then there exists a $p$-group $P$, a finite-length central $p$-filtration~$\mathbf P$
of $P$, and for each $v\in V(Y)$ a group embedding $\psi_v\colon G_v\to P$ and a compatible stretching $\mathbf G^*=\{\mathbf G_{v}^*\}_v$ of $\mathbf G$  such that
\begin{enumerate}
\item $\psi_{t(e)}\circ f_e = \psi_{t(\overline{e})}\circ f_{\overline{e}}$ for every $e\in E(Y)$;
\item $\psi_v$ is a morphism $(G_v,\mathbf G_v^*)\to (P,\mathbf P)$ of filtered groups, for every $v\in V(Y)$;
\item for each $n\geq 1$, after naturally identifying each vertex group $L_n(\mathbf G^*_v)$ of the tree $L_n(\mathbf G^*)$ of elementary abelian $p$-groups
with its image in $L_n(\mathbf P)$, the latter group contains the \textup{(}internal\textup{)} fiber sum $\Sigma(L_n(\mathbf G^*))$ of $L_n(\mathbf G^*)$.
\end{enumerate}
In particular, $G=\pi_1(\scrg)$ is residually $p$.
\end{lemma}
\begin{proof}
The proof of the first statements  proceeds by an obvious induction on the number of vertices of $\scrg$, using Proposition~\ref{prop:higman, more precise} in the inductive step. The existence of the family of group embeddings $\psi_v\colon G_v\to P$ satisfying (1) together with Lemma~\ref{residual p criterion} yields that $\pi_1(\scrg)$ is residually $p$.
\end{proof}

With this lemma in hand, we now turn to:

\begin{proof}[Proof of Proposition~\ref{prop:reduction}]
Note that the hypotheses of the proposition are preserved under passage from $\mathbf G$ to a compatible stretching.  (Since free groups are residually $p$.) Thus,
by Lemma~\ref{lem:higman, 2} applied to the tree of groups $\scrg|T$ in place of $\scrg$, we may assume that there exists a morphism $\psi\colon \pi_1(\scrg|T)\to P$ to a $p$-group~$P$ and a finite-length central $p$-filtration $\mathbf P$ of $P$ such that
\begin{enumerate}
\item $\psi\circ f_e = \psi\circ f_{\ol{e}}$ for every $e\in E(T)$;
\item $\psi$ restricts to an injective morphism $(G_v,\mathbf G_v)\to (P,\mathbf P)$ of filtered groups, for every $v\in V(Y)$; and
\item $L_n(\mathbf P)\geq\Sigma(L_n(\mathbf G)|T)$ for each $n\geq 1$ (after naturally identifying each vertex group $L_n(\mathbf G_v)$ of $L_n(\mathbf G)|T$ with its image in $L_n(\mathbf P)$ under $L_n(\psi|G_v)$).
\end{enumerate}
For $e\in E$ define $$\psi_e:=\psi\circ\varphi_e\circ\psi^{-1}\colon\psi(A_e)\to\psi(B_e).$$
The filtration $\mathbf P$ of $P$ is $\psi_e$-invariant, for each $e\in E$. For each $n\geq 1$, the iterated HNN extension
\begin{multline*}
\pi_1^*(L_n(\mathbf G),T)= \\ \big\langle \Sigma(L_n(\mathbf G)|T),\ e\in E_+: \text{$eae^{-1}=L_n(\varphi_{e})(a) $ for all $e\in E_+$, $a\in L_n(\mathbf A_{e})$}\big\rangle
\end{multline*}
of $\Sigma(L_n(\mathbf G)|T)$ is residually $p$, by assumption, hence by Corollary~\ref{cor:generate HNN} (or simply because $\Sigma(L_n(\mathbf G)|T)$ is a direct summand of the $\F_p$-linear space $P_n$),  there is a $p$-group containing $P_n$ and an extension of each $L_n(\psi_e)$, $e\in E_+$, to an inner automorphism.
Hence
 by Lemma~\ref{lem:hnn}
 there is a $p$-group $Q\geq P$ and for each $e\in E_+$ an element $h_e\in Q$ such that
$h_e a h_e^{-1}=\psi_e(a)$ for all $e\in E_+$ and $a\in\psi(A_e)$.
Hence $\psi$ extends to a morphism $G\to Q$ by $e\mapsto h_e$ for all $e\in E_+$ 
and so by Lemma~\ref{residual p criterion}, $G$ is residually $p$.
\end{proof}



In Section~\ref{sec:virtual res p} we use the theorem above to prove a criterion for the fundamental group of a graph of groups to be \emph{virtually} residually $p$.
A last missing ingredient is a procedure on graphs of groups introduced in the next section.

\section{Unfolding a graph of groups}
\label{sec:unfolding a graph of groups}

\noindent
Let $\scrg$ be a graph of groups with underlying graph $Y$ and fundamental group $G=\pi_1(\scrg)$.
Suppose $\psi\colon G\to A$ is a surjective group morphism to a finite group $A$ which
is trivial on each vertex group $G_v$ of $\scrg$.
By Theorem~\ref{thm:bass-serre}, the (finite index, normal) subgroup $K=\ker(\psi)$ of $G$ can also be realized as the fundamental group of a certain graph of groups whose vertex and edge groups are given by the vertex respectively edge groups of $\scrg$. We need an explicit description of such a graph of groups $\widetilde{\scrg}$, which may be thought of as a kind of ``unfolding'' of $\scrg$ along $\psi$.
We begin with a description of the underlying graph $\widetilde Y$ of $\widetilde{\scrg}$.

\index{unfolding}

\subsection{Unfolding a graph}\label{subsec:unfolding a graph}
Let $Y$ be a graph;  we write $V=V(Y)$ and $E=E(Y)$.  We denote by $\pi_1(Y)$ the fundamental group of $Y$, i.e., $\pi_1(Y)=\pi_1(\scry)$ where $\scry$ is the graph of groups based on $Y$ with trivial vertex groups. We fix a maximal subtree $T$ of $Y$, and identify $\pi_1(Y)$ with $\pi_1(Y,T):=\pi_1(\scry, T)$ as indicated in \eqref{eq:maximal subtree description of pi1}.

Let $A$ be a finite group and $\psi\colon \pi_1(Y)\to A$ a surjective group morphism.
Note that then $\psi$ extends uniquely to a group morphism $\pi(\scry)\to A$. (This depends on the choice of $T$.) Here $\pi(\scry)$ is the path group of $\scry$, i.e., the group generated by $E$ subject to the relations $e^{-1}=\ol{e}$ for $e\in E$.  We have $\psi(e)=1$ for all $e\in E(T)$.

Let now $\widetilde{Y}$ be the finite graph with vertex set $V(\widetilde{Y})=A\times V$, edge set  $E(\widetilde{Y})=A\times E$, and
$$\ol{(\alpha,e)}=(\psi(e)\alpha,\ol{e}),\quad o(\alpha,e) = (\alpha,o(e)), \quad t(\alpha,e) = (\psi(e)\alpha, t(e))$$
for all $(\alpha,e)\in E(\widetilde{Y})$.
It is easy to see that $\widetilde{Y}$ is connected since $\psi(E)$ generates the group $A$. We have a natural morphism $\phi\colon \widetilde{Y}\to Y$ of graphs given by
$$V(\widetilde{Y})=A\times V\to V\colon (\alpha,v)\mapsto v, \qquad
  E(\widetilde{Y})=A\times E\to E\colon (\alpha,e)\mapsto e.$$
We call $\phi$ the \emph{unfolding of the graph $Y$ along $\psi$.}
Ostensibly, $\phi$ seems to depend on the choice of $T$. However,
it is not hard to see (e.g., using covering spaces and Proposition~\ref{prop:phistar} below) that in fact, given another choice $T_1$ of maximal subtree of~$Y$ and $\phi_1\colon \widetilde{Y}_1\to Y$ constructed as above, there exists an isomorphism $\Phi\colon \widetilde{Y}\to\widetilde{Y}_1$ of graphs such that $\phi_1\circ\Phi=\phi$.
If $Y$ is a tree (so $Y=T)$, then the group $\pi_1(Y)$ and hence $A$ is trivial, and $\widetilde{Y}$ is naturally isomorphic to the original graph $Y$.

\subsection*{Digression: unfolding of graphs and Cayley graphs}
There is an obvious connection between unfolding of graphs and Cayley graphs of finite groups, as explained in the following example (not used later):

\begin{example}
Recall that
given a generating set $S$ of a group $\Gamma$, the Cayley graph of $\Gamma$ with respect to $S$ is the {\it oriented}\/ (possibly infinite) graph $C=C(\Gamma,S)$ with vertex set $V(C)=\Gamma$, orientation $E(C)_+=\Gamma\times S$, and
$$o(\gamma,s)=\gamma, \quad t(\gamma,s)=\gamma s\qquad \text{for each $(\gamma,s)\in \Gamma\times S$.}$$
Let $E_+$ be an orientation of $Y$; then $\psi(E_+)$ generates $A$, and $E(\widetilde{Y})_+=A\times E_+$ is an orientation of $\widetilde{Y}$.
Thus the natural surjections $$V(\widetilde{Y})=A\times V\to A\colon (\alpha,v)\mapsto\alpha$$ and $$E(\widetilde{Y})_+=A\times E_+\to E(C)_+=A\times\psi(E_+)\colon (\alpha,e)\mapsto (\alpha,\psi(e))$$
define a morphism $\widetilde{Y}\to C=C(A,\psi(E_+))$ of graphs.
In particular, if $Y$ has only a single vertex, and $\psi|E$ is injective, then $\widetilde{Y}$ is nothing but the Cayley graph of $A$ with respect to its generating set $\psi(E_+)$.
\end{example}

\subsection{Unfolding a graph of groups}\label{subsec:unfolding}
We continue to use the notation introduced in the last subsection, but assume in addition that we are given a graph of groups $\scrg$ based on $Y$.
We define $\widetilde{\scrg}$ as the graph of groups with underlying graph~$\widetilde{Y}$, vertex respectively edge groups
$$\widetilde{G}_{(\alpha,v)} = G_v, \quad \widetilde{G}_{(\alpha,e)} = G_e \qquad \text{for all $\alpha\in A$, $v\in V$, $e\in E$,}$$
and edge morphisms
$$\widetilde{f}_{(\alpha,e)} = f_e\qquad \text{for all $(\alpha,e)\in E(\widetilde{Y})$.}$$
There is a natural morphism $\widetilde{\scrg}\to\scrg$ of graphs of groups with underlying morphism of graphs $\phi\colon\widetilde{Y}\to Y$, which is the identity on each vertex group $\widetilde{G}_{(\alpha,v)}=G_v$ and on each edge group $\widetilde{G}_{(\alpha,e)} = G_e$. This morphism, also denoted by $\phi$, yields a surjective morphism $\pi(\widetilde{\scrg})\to\pi(\scrg)$, which restricts to a morphism
$\phi_*\colon\pi_1(\widetilde{\scrg})\to \pi_1(\scrg)=G$ between fundamental groups. We call $\phi$ the {\it unfolding of $\scrg$ along $\psi$.}\/ In the proposition below we continue to denote the extension of $\psi\colon\pi_1(Y,T)\to A$ to a morphism $\pi_1(\scrg,T)\to A$ with $g\mapsto 1$ for $g\in G_v$, $v\in V$, by the same symbol.

\medskip

The following proposition may also be proved easily using the topological interpretation of graphs of groups; we give a purely algebraic proof.

\begin{proposition}\label{prop:phistar}
The morphism $\phi_*$  is injective, with image $K=\ker(\psi)$. 
\end{proposition}
\begin{proof}
Let $v_0\in V$ be a base point of $G$, that is, $G=\pi_1(\scrg,v_0)$.
We choose $(1,v_0)$ as our base point for $\pi_1(\widetilde{\scrg})=\pi_1(\widetilde{\scrg},(1,v_0))$.
Let $\widetilde{g}$ be an element of $\pi_1(\widetilde{\scrg})$:
$$\widetilde{g} = g_0\, (\alpha_1,e_1)\, g_1\, (\alpha_2,e_2) \cdots (\alpha_n,e_n)\, g_n;$$
where $g_i\in \widetilde{G}_{t(\alpha_i,e_i)}$, and
\begin{equation}\label{eq:unwind 1}
(\psi(e_i)\alpha_i,t(e_i)) = t(\alpha_i,e_i) = o(\alpha_{i+1},e_{i+1}) = (\alpha_{i+1},o(e_{i+1})) \quad (1\leq i<n)
\end{equation}
and
\begin{equation}\label{eq:unwind 2}
(\psi(e_n)\alpha_n,t(e_n)) = t(\alpha_n,e_n) = o(\alpha_1,e_1)=(\alpha_1,o(e_1))=(1,v_0).
\end{equation}
Note that then
\begin{equation}\label{eq:unwind 1'}
\alpha_{i+1}=\psi(e_i\cdots e_1)\qquad (1\leq i<n)
\end{equation}
by \eqref{eq:unwind 1} and
\begin{equation}\label{eq:unwind 2'}
\alpha_n = \psi(e_n)^{-1}
\end{equation}
by \eqref{eq:unwind 2},
so $$\psi(\phi_*(\widetilde{g}))=\psi(e_n\cdots e_1)=1.$$
Conversely, given an element $g=g_0\,e_1\,g_1\,e_2\cdots e_n\, g_n$ of $K$, defining $\alpha_i$ as in \eqref{eq:unwind 1'} and \eqref{eq:unwind 2'} yields an element $\widetilde{g}$ of  $\pi_1(\widetilde{\scrg})$ as above with $\phi_*(\widetilde{g})=1$. This shows that $\phi_*(\pi_1(\widetilde{\scrg}))=K$.
Now suppose $\widetilde{g}\neq 1$, and the representation of $\widetilde{g}$ above is reduced, that is, either
\begin{enumerate}
\item $n=0$ and $g_0\neq 1$; or
\item $n>0$ and $g_i\notin \widetilde{f}_{(\alpha_i,e_i)}(\widetilde{G}_{(\alpha_i,e_i)})$ for each $i$ such that
$(\alpha_{i+1},e_{i+1})=\ol{(\alpha_i,e_i)}$.
\end{enumerate}
Since
$\ol{(\alpha_i,e_i)} = (\psi(e_i)\alpha_i,\ol{e_i})$,
the image  $\phi_*(\widetilde{g})=g_0 e_1 g_1 e_2 \cdots e_n g_n$ of $\widetilde{g}$ under $\phi_*$ is also represented by a reduced path. Thus $\phi_*$ is injective.
\end{proof}

The following simple observation is used in Section~\ref{sec:virtual res p}.

\begin{lemma}\label{lem:unfolding and quotients}
Let $\psi\colon \pi_1(Y)\to A$ be a morphism onto a finite group, and let $\phi\colon\widetilde{\scrg}\to\scrg$ be the unfolding of $\scrg$ along $\psi$.
Let $\scrh=\{H_v\}$ be a compatible collection of normal subgroups of $\scrg$,  and let
$\widetilde{\scrg/\scrh}\to\scrg/\scrh$ be the unfolding of $\scrg/\scrh$ along $\psi$. Then $\widetilde{\scrg}/\widetilde{\scrh} \cong\widetilde{\scrg/\scrh}$ as graphs of groups, where $\widetilde{\scrh}=\phi^{-1}(\scrh)$.
\end{lemma}

We leave the proof to the reader.

\index{HNN extension}

\subsection*{Digression: unfolding an HNN extension}
It might be instructive for the reader to see  the unfolding procedure at work in a concrete simple example. (The material in this subsection is not used later.)
Suppose the underlying graph $Y$ of~$\scrg$ has only one vertex $v_0$; we write $G_{v_0}=:\Sigma$. The fundamental group of $\scrg$ is the iterated HNN extension
$$G = \big\langle \Sigma,\ e\in E_+: \text{$e f_e(g) e^{-1} = f_{\ol{e}}(g)$ for all $e\in E_+$, $g\in G_e$}\big\rangle$$
of $\Sigma$, where $E_+$ is an orientation of $Y$.
In this case, the definition of the unfolding $\widetilde{\scrg}$ of $\scrg$ along $\psi$ given above simplifies somewhat:
we may simply describe its underlying graph $\widetilde{Y}$ by $V(\widetilde{Y})=A$, $E(\widetilde{Y})=A\times E$, and
$$\ol{(\alpha,e)}=(\psi(e)\alpha,\ol{e}), \quad
o(\alpha,e)=\alpha, \quad t(\alpha,e)=\psi(e)\alpha,$$
and we have
$$\widetilde{G}_{\alpha} = \Sigma, \quad \widetilde{G}_{(\alpha,e)} = G_e, \quad \widetilde{f}_{(\alpha,e)} = f_e,$$
for $(\alpha,e)\in E(\widetilde{Y})$.
Figure~\ref{fig:loop} illustrates this for the case where $Y$ has only one topological edge $\{e,\ol{e}\}$, and $A=\Z/s\Z$ is cyclic, with $\varphi:=f_e\circ f_{\ol{e}}^{-1}$.

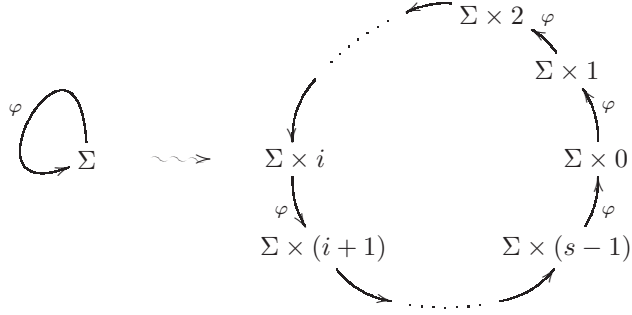
\begin{figure}
$$
\xygraph{
!{<0cm,0cm>;<1cm,0cm>:<0cm,1cm>::}
!{(0,0);a(0)**{}?(2.0)}*+{\Sigma}="G0"
"G0" :@`{p+(0,2.5),p+(-2,-1)}_{\varphi} "G0"
}
\xymatrix{\quad\ar@{~>}[r]&\quad}
\xygraph{
!{<0cm,0cm>;<1cm,0cm>:<0cm,1cm>::}
!{(0,0);a(0)**{}?(2.0)}*+{\Sigma\times 0}="G0"
!{(0,0);a(36)**{}?(2.0)}*+{\Sigma\times 1}="G1"
!{(0,0);a(72)**{}?(2.0)}*+{\Sigma\times 2}="G2"
!{(0,0);a(108)**{}?(2.0)}*+{}="G3"
!{(0,0);a(144)**{}?(2.0)}*+{}="Gi-2"
!{(0,0);a(180)**{}?(2.0)}*+{\Sigma\times i}="Gi-1"
!{(0,0);a(216)**{}?(2.0)}*+{\Sigma\times (i+1)}="Gi"
!{(0,0);a(252)**{}?(2.0)}*+{}="Gi+1"
!{(0,0);a(288)**{}?(2.0)}*+{}="Gs-2"
!{(0,0);a(324)**{}?(2.0)}*+{\Sigma\times (s-1)}="Gs-1"
"G0":@/_0.5em/"G1"_{\varphi}
"G1":@/_0.5em/"G2"_{\varphi}
"G2":@/_0.5em/"G3"
"G3":@{.}@/_0.2em/"Gi-2"
"Gi-2":@/_0.5em/"Gi-1"
"Gi-1":@/_0.5em/"Gi"_{\varphi}
"Gi":@/_0.5em/"Gi+1"
"Gi+1":@{.}@/_0.2em/"Gs-2"
"Gs-2":@/_0.5em/"Gs-1"
"Gs-1":@/_0.5em/"G0"_{\varphi}
}
$$
\caption{Unfolding an HNN extension along a morphism to $\Z/s\Z$}
\label{fig:loop}
\end{figure}

\medskip

Suppose now that $A$ is a subgroup of $\Aut(\Sigma)$, and for each edge $e$ of $Y$, the automorphism $\psi(e)\in A$ extends the partial automorphism $f_{e}\circ f_{\ol{e}}^{-1}$ of $\Sigma$. 
Let $\widetilde{T}$ be a maximal subtree of $\widetilde{Y}$, and  identify each vertex group $\widetilde{G}_\alpha$ with a subgroup of $\widetilde{G}:=\pi_1(\widetilde{\scrg},\widetilde{T})$ as usual.
In this situation unfolding has the following desirable property:

\begin{lemma}\label{lem:unfolding HNN}
There exists a morphism $\widetilde{G}\to\Sigma$ which is bijective on each vertex group $\widetilde{G}_\alpha$ of $\widetilde{\scrg}$.
\end{lemma}
\begin{proof}
For  all $(\alpha,e)\in E(\widetilde{Y})$,  the isomorphism $\psi(e)$ extends
$\widetilde{f}_{(\alpha,e)}\circ\widetilde{f}_{\ol{(\alpha,e)}}^{-1}=
f_{e}\circ f_{\ol{e}}^{-1}$, and the composition of the $\psi(e)$ along the geodesic $((\alpha_1,e_1),\dots,(\alpha_n,e_n))$ from the vertex $\alpha$  to $1$ in $\widetilde{T}$ is the isomorphism $\widetilde{G}_\alpha \to \widetilde{G}_1$ given by
$\psi(e_n\cdots e_1) = \alpha^{-1}$.
Hence by Lemma~\ref{lem:special amalgam} there is a morphism
$$\widetilde{\varphi}\colon\pi_1(\widetilde{\scrg}|\widetilde{T})\to \widetilde{G}_1=\Sigma$$ which, for each $\alpha\in A$, extends  $\alpha^{-1}$ (and hence is bijective on $\widetilde{G}_\alpha$).
Now
$\alpha^{-1} = (\psi(e)\alpha)^{-1}\circ\psi(e)$,
and therefore
$$\widetilde{\varphi}\circ  \widetilde{f}_{\ol{(\alpha,e)}} = \alpha^{-1} \circ \widetilde{f}_{\ol{(\alpha,e)}} =
(\psi(e)\alpha)^{-1}  \circ \widetilde{f}_{(\alpha,e)} = \widetilde{\varphi} \circ \widetilde{f}_{(\alpha,e)}.
$$
(See Figure~\ref{fig:unfolding HNN}.)
This allows us to extend  $\widetilde{\varphi}$  to
a morphism $\widetilde{G}\to \widetilde{G}_1=\Sigma$ by  $(\alpha,e)\mapsto 1$ for all edges $(\alpha,e)\in E(\widetilde{Y})\setminus E(\widetilde{T})$.
\end{proof}

An analogue of this fact will be established for general graphs of elementary abelian $p$-groups in Proposition~\ref{prop:unfolded}.

\begin{figure}
$$
\xygraph{
!{<0cm,0cm>;<1cm,0cm>:<0cm,1cm>::}
!{(0,0)}*+{\widetilde{G}_\alpha=\Sigma}="Galpha"
!{(8,0)}*+{\widetilde{G}_{\psi(e)\alpha}=\Sigma}="Gpsiealpha"
!{(3,-3)}*+{\Sigma}="G1"
!{(5,-3)}*+{\Sigma}="G2"
!{(4,-2)}*+{\Sigma}="G3"
!{(2,-2)}*+{\Sigma}="G4"
!{(6,-2)}*+{\Sigma}="G5"
!{(1,-1)}*+{\Sigma}="G6"
!{(7,-1)}*+{\Sigma}="G7"
!{(6,0)}*+{\Sigma}="G8"
!{(-1,1)}*+{\Sigma}="G9"
!{(1,1)}*+{\Sigma}="G10"
!{(1.5,1.5)}="G11"
!{(-1.5,1.5)}="G12"
!{(0.5,1.5)}="G13"
!{(-0.5,1.5)}="G14"
!{(3,-1)}="G15"
!{(9.5,1.5)}="G16"
!{(6.5,1.5)}="G17"
!{(4,-4)}*+{\widetilde{G}_1=\Sigma}="G0"
"G1"-"G4"
"Galpha":@{-->}@/^2.5em/"Gpsiealpha"_{\psi(e)}
"G2"-"G5"
"G1"-"G3"
"G0"-"G1"
"G0"-"G2"
"G10":@{.}"G11"
"G9":@{.}"G12"
"G10":@{.}"G13"
"G9":@{.}"G14"
"Galpha"-"G9"
"Galpha"-"G10"
"G4":@{.}"G6"
"G6"-"Galpha"
"G5":@{.}"G7"
"G4":@{.}"G15"
"Gpsiealpha":@{.}"G16"
"Gpsiealpha":@{.}"G17"
"G7"-"G8"
"G7"-"Gpsiealpha"
"Galpha":@{-->}@/_3.5em/"G0"_{\alpha^{-1}}
"Gpsiealpha":@{-->}@/^4.5em/"G0"^{(\psi(e)\alpha)^{-1}}
}
$$
\caption{The tree of groups $\widetilde{\scrg}|\widetilde{T}$ in the proof of Lemma~\ref{lem:unfolding HNN}}
\label{fig:unfolding HNN}
\end{figure}

\section{A criterion for being virtually residually $p$}\label{sec:virtual res p}

\noindent
We now continue where Section~\ref{sec:A Reduction Theorem} left off, and put ourselves back in the setting of Section~\ref{subsec:reduction theorem}. That is, we  let $\scrg$ be a graph of finitely generated groups with underlying graph $Y$.
We recall the statement of the main theorem of this chapter:

\begin{theorem}\label{thm:reduction theorem, 3}
If $\scrg$ admits a $p$-excellent filtration, then the group $G=\pi_1(\scrg)$ is virtually residually $p$.
\end{theorem}

\index{filtration!$p$-excellent}

Throughout the rest of this section, we fix a maximal subtree $T$ of the underlying graph $Y$ of $\scrg$. We identify $G$ with $\pi_1(\scrg,T)$ and in this way each vertex group~$G_v$ of $\scrg$ with a subgroup of $G$.

\medskip

Theorem~\ref{thm:reduction theorem, 3} will be obtained by applying the reduction theorem (Theorem~\ref{thm:reduction theorem}) to a suitable finite index subgroup of $\pi_1(\scrg)$;
the heart of the construction of such a subgroup is contained in the
following proposition:

\begin{proposition}\label{prop:unfolded}
Suppose all vertex groups $G_v$ of $\scrg$ are elementary abelian $p$-groups. Then
there exists a morphism $\psi\colon \pi_1(Y)\to A$ onto a finite group $A$, with corresponding unfolding $\widetilde{\scrg}\to\scrg$ of $\scrg$, and
a morphism $\pi_1(\widetilde{\scrg})\to\Sigma(\scrg|T)$ which is injective when restricted to each vertex group of $\widetilde{\scrg}$.
\end{proposition}


\index{group!homogeneous}

\begin{proof}
We exploit the homogeneity of the elementary abelian $p$-group $\Sigma=\Sigma(\scrg|T)$:
Choose an orientation $E_+$ of $E=E(Y)\setminus E(T)$, i.e., $E$ is the disjoint union $E=E_+\cup\ol{E_+}$; and for every $e\in E_+$ choose an extension of the partial automorphism $\varphi_e:=f_{\ol{e}}\circ f_{e}^{-1}$ of  $\Sigma$ to an automorphism $\sigma_e$ of $\Sigma$.

We denote the free group on the generators $E_+$ by $F$. Note that $F=\pi_1(Y,T)$. We let $\psi\colon F\to\Aut(\Sigma)$ be the morphism with $e\mapsto\sigma_e$ ($e\in E_+$), and
we consider the semidirect product $F\ltimes\Sigma$ of $F$ with $\Sigma$ via $\psi$:
$$F\ltimes\Sigma=\big\langle \Sigma,\ e\in E_+ : \text{$e a e^{-1} =  \sigma_e(a)$ for all $e\in E_+$, $a\in \Sigma$}\big\rangle.$$
We have a surjective morphism of the group
$$\pi_1^*(\scrg,T)=
\big\langle \Sigma,\ e\in E_+ : \text{$e a e^{-1} = \varphi_e(a)$ for all $e\in E_+$, $a\in f_e(G_e)$}\big\rangle$$
onto $F\ltimes\Sigma$ which is the identity on $\Sigma$ and on $E_+$.
Composing with the natural morphism $\pi_1(\scrg,T)\to\pi_1^*(\scrg,T)$ we thus obtain a morphism $\Psi\colon\pi_1(\scrg,T)\to F\ltimes\Sigma$ which is injective on the vertex groups of $\scrg$.
Let $A$ be the subgroup of $\Aut(\Sigma)$ generated by the $\sigma_e$ ($e\in E_+$).
Writing $\Gamma:=\ker(\psi)$, we have a short exact sequence
$$1 \to \Gamma\ltimes\Sigma \to F\ltimes\Sigma\to A\to 1.$$
Note that $\Gamma$ acts trivially on $\Sigma$, in particular, $\Gamma\ltimes\Sigma=\Gamma\times\Sigma$. As in Section~\ref{sec:unfolding a graph of groups} we continue to denote the extension of $\psi\colon\pi_1(Y,T)\to A$ to a morphism $\pi_1(\scrg,T)\to A$ which is the trivial morphism on each vertex group $G_v$ by the same symbol $\psi$.
Then the following diagram commutes:
\[\xymatrix{
\pi_1(\scrg,T) \ar[r]^\Psi \ar[dr]^\psi & F\ltimes\Sigma \ar[d] \\
      & A.
}\]
We let  $\phi\colon\widetilde{\scrg}\to\scrg$ be the unfolding of $\scrg$ along $\psi\colon \pi_1(Y,T)\to A$ as constructed in Section~\ref{subsec:unfolding} above.
Then by Proposition~\ref{prop:phistar}
$$\phi_*(\pi_1(\widetilde{\scrg}))=\ker\left(\pi_1(\scrg,T)\xrightarrow{\psi} A\right)$$
and thus
$$\Psi(\phi_*(\pi_1(\widetilde{\scrg})))=\ker(F\ltimes\Sigma\to A)=\Gamma\ltimes\Sigma.$$
So there is a morphism $\widetilde{\Psi}\colon\pi_1(\widetilde{\scrg})\to\Gamma\ltimes\Sigma$ making the diagram
\[\xymatrix{
\pi_1(\widetilde{\scrg}) \ar[d]_{\phi_*} \ar[r]^{\widetilde{\Psi}} & \Gamma\ltimes\Sigma \ar[d] \\
\pi_1(\scrg) \ar[r]^\Psi  & F\ltimes\Sigma
}\]
commute. Clearly $\widetilde{\Psi}$ is injective on the vertex groups of $\widetilde{\scrg}$. Let $\pi$ be the natural projection $\Gamma\ltimes\Sigma=\Gamma\times\Sigma\to\Sigma$. Then $\pi\circ\widetilde{\Psi}\colon\pi_1(\widetilde{\scrg})\to\Sigma$ is injective on the vertex groups of $\widetilde{\scrg}$.
\end{proof}

For the following corollary of Proposition~\ref{prop:unfolded}, let $\scrh=\{H_v\}$ be a compatible collection of normal subgroups of $\scrg$ with $\gamma^p_2(G_v)\leq H_v$ for all $v\in V(Y)$. (Later this will be applied to $\scrh=\mathbf G_2$ where $\mathbf G$ is a complete $p$-excellent filtration of $\scrg$.) Then $\scrg/\scrh$ is a graph of elementary abelian $p$-groups based on $Y$. 
Given a morphism $\phi\colon\widetilde{\scrg}\to\scrg$ we denote by $\widetilde{\scrh}:=\phi^{-1}(\scrh)$ the compatible collection of normal subgroups of $\widetilde\scrg$ defined in Lemma~\ref{lem:morphisms and filtrations, 1}.
Note that then $\gamma^p_2(\widetilde{G}_{\widetilde{v}})\leq \widetilde{H}_{\widetilde{v}}$ for all vertices $\widetilde{v}$ of the underlying graph $\widetilde{Y}$ of $\widetilde{\scrg}$.

\begin{corollary}\label{cor:unfolded, 1}
There exists a   morphism $\phi\colon \widetilde{\scrg}\to\scrg$ of graphs of groups with the following properties:
\begin{enumerate}
\item Each group morphism $\phi_{\widetilde{v}}\colon \widetilde{G}_{\widetilde{v}}\to G_{\phi(\widetilde{v})}$ and
$\phi_{\widetilde{e}}\colon \widetilde{G}_{\widetilde{e}}\to G_{\phi(\widetilde{e})}$ is bijective;
\item the induced group morphism $\phi_*\colon \pi_1(\widetilde{\scrg})\to G=\pi_1(\scrg)$ is injective;
\item its image $\phi_*(\pi_1(\widetilde{\scrg}))$ is a finite index normal subgroup  of $G$; and
\item for every maximal subtree $\widetilde{T}$ of the graph underlying the graph of groups~$\widetilde{\scrg}$, the group $\pi_1^*(\widetilde{\scrg}/\widetilde{\scrh},\widetilde{T})$ is residually $p$.
\end{enumerate}
\end{corollary}
\begin{proof}
Applying
Proposition~\ref{prop:unfolded} with $\scrg/\scrh$ in place of $\scrg$, we obtain a morphism $\psi\colon\pi_1(Y)\to A$ to a finite group, with corresponding unfolding $\widetilde{\scrg/\scrh}\to\scrg/\scrh$ of $\scrg/\scrh$, and a morphism $\pi_1(\widetilde{\scrg/\scrh})\to\Sigma$ to an elementary abelian $p$-group which is injective when restricted to each vertex group of $\widetilde{\scrg/\scrh}$. Let $\phi\colon\widetilde{\scrg}\to\scrg$ be the unfolding of $\scrg$ along $\psi$.
Then (1)--(3) hold by Proposition~\ref{prop:phistar}.
By Lemma~\ref{lem:unfolding and quotients} we have $\widetilde{\scrg/\scrh}\cong\widetilde{\scrg}/\widetilde{\scrh}$, hence (4) follows from Lemma~\ref{lem:crucial}.
\end{proof}

We can now prove Theorem~\ref{thm:reduction theorem, 3}. Suppose first that $\mathbf G$ is a {\it complete}\/ $p$-excellent filtration of $\scrg$. Apply
Corollary~\ref{cor:unfolded, 1} with $\scrh:=\mathbf G_2$, and let $\phi\colon\widetilde{\scrg}\to\scrg$ have properties (1)--(4) in that corollary. By (1) and Lemmas~\ref{lem:separating, 1} and \ref{lem:separating, 2}, the filtration $\widetilde{\mathbf G}=\phi^{-1}(\mathbf G)$ of $\widetilde{\scrg}$ remains $p$-excellent.
This together with (4) shows that $\widetilde{\scrg}$ satisfies the hypotheses of Theorem~\ref{thm:reduction theorem}. Hence $\pi_1(\widetilde{\scrg})$ is residually $p$, and thus by (2) and~(3), the group $\pi_1(\scrg)$ is virtually residually $p$.

If $\mathbf G$ is an arbitrary $p$-excellent filtration of $\scrg$, we consider a
morphism of graphs of groups satisfying (1)--(3) in Proposition~\ref{prop:commoncover} applied to the compatible collection of normal subgroups $\{H_v\}=\{G_{v,1}\}$ of $\scrg$.
The pullback of  $\mathbf G$ under such a morphism is complete and, by Lemmas~\ref{lem:separating, 1} and \ref{lem:separating, 2}, remains $p$-excellent.
 In this way we reduce to the case of complete $\mathbf G$ treated above. \qed

\chapter{Proof of the Main Results}\label{ch:proof of the main results}

\noindent
In this chapter we first define (in Section~\ref{sec:p-compatible}) what we mean by a filtration of a group to be $p$-compatible with a given collection of subgroups, and then show (in Section~\ref{sec:p-compatible linear}) a general theorem which allows to construct filtrations of finitely generated linear groups which are $p$-compatible with certain abelian subgroups. In Section~\ref{sec:proof of the main theorem} we then apply this theorem to give a proof of our main theorem from the introduction. Section~\ref{sec:approx} fills in the proof of a commutative algebraic fact used earlier (in Section~\ref{sec:p-compatible linear}), and Section~\ref{sec:fibered} discusses the main theorem in the special case of fibered $3$-manifolds.

\section{$p$-compatible filtrations}\label{sec:p-compatible}

\noindent
We make the following somewhat ad-hoc definition (which is related to the ``$p$-congruence systems'' of \cite[Interlude~B]{DdSMS}):

\begin{definition}
Let $G$ be a group, let $\mathcal T$ be a collection of subgroups of $G$, and let $\ell\geq 0$ be an integer.
We say that a filtration $\mathbf G=\{G_n\}_{n\geq 1}$ of $G$ is  \emph{$p$-compatible of level $\ell$ with $\mathcal T$} if the following hold:
\begin{enumerate}
\item $\mathbf G$ is a normal filtration of $G$ with $[G:G_n]<\infty$ for each $n\geq 1$;\item $\mathbf G$, construed as a complete filtration of $G_1$, is $p$-potent;
\item $\mathbf G$ is separating (i.e., $\bigcap_{n=1}^\infty G_n=\{1\}$);
\item $\mathbf G$ separates $\mathcal T$ (i.e., $\bigcap_{n=1}^\infty G_n\cdot T=T$ for every  $T\in\mathcal T$); and
\item for every $T\in\mathcal T$ the filtration $\mathbf G$ intersects to the lower central $p$-filtration $\gamma^p(T)$  of $T$, shifted by $\ell$ terms to the left, i.e.,
$\mathbf G\cap T = \{\gamma^p_{n+\ell}(T)\}_{n\geq 1}$.
\end{enumerate}
A filtration of $G$ which is $p$-compatible  of level $0$ with $\mathcal T$ will simply said to be $p$-compatible with $\mathcal T$. (We will only encounter $p$-compatible filtrations of level $0$ or~$1$ below.)
\end{definition}

\index{filtration!$p$-compatible}

If the group $G$ admits a filtration which is $p$-compatible with some collection of subgroups, then $G$ is virtually residually $p$.
We introduced the notion of $p$-compatible filtration because it furnishes a ``local'' version of the concept of $p$-excellent filtration of a graph of groups considered in the previous chapter:
Let~$\ell\geq 0$ be an integer; then,
given a graph of groups $\scrg$ and for each vertex $v$ a filtration~$\mathbf G_v$ of $G_v$ which is $p$-compatible of level $\ell$ with the subgroups $f_e(G_e)$ of $G_v$, where $e$ ranges over all edges with $t(e)=v$, we obtain a filtration $\mathbf G=\{\mathbf G_v\}_v$ of $\scrg$. We say that a filtration $\mathbf G$ of $\scrg$ arising in this way is \emph{$p$-compatible of level $\ell$.}
Such a $p$-compatible filtration $\mathbf G$ of level $\ell$ of $\scrg$ automatically satisfies conditions (1)--(4) in the definition of $p$-excellency; if in addition the lower central $p$-series of the edge groups of $\scrg$, shifted by $\ell$ terms to the left, is uniformly $p$-potent (e.g., if the edge groups of $\scrg$ are abelian $p$-torsion free, see Section~\ref{sec:p-potent}), then every $p$-compatible filtration of level $\ell$ is $p$-excellent.
From the main result (Theorem~\ref{thm:reduction theorem, 2}) of the preceding chapter we therefore immediately obtain:

\begin{theorem}\label{thm:virt res p for graphs of linear groups, 1}
Let $\scrg$ be a graph of finitely generated groups, and suppose
all edge groups $G_e$ are abelian $p$-torsion free. Suppose for some $\ell\geq 0$,
$\scrg$ admits a $p$-compatible filtration of level $\ell$.
Then $\pi_1(\scrg)$ is virtually residually $p$.
\end{theorem}

In the next section we construct filtrations of linear groups which are $p$-compa\-tible with certain abelian subgroups. In Section~\ref{sec:proof of the main theorem} we'll apply this construction together with the criterion in Theorem~\ref{thm:virt res p for graphs of linear groups, 1} above in
the case of graphs of groups arising from $3$-manifolds.

\medskip

In the rest of this section we collect a few more observations on $p$-compatible filtrations, and record an application of Theorem~\ref{thm:virt res p for graphs of linear groups, 1}.
First we note that under natural conditions, if there is a complete $p$-compatible filtration at all, then there is a canonical one, by Lemma~\ref{lem:canonical filtration}:

\begin{lemma}\label{lem:canonical filtration, 2}
Let $G$ be a group  and let $\mathcal T$ be a collection of subgroups of $G$.
If~$\mathbf G$ is a central $p$-filtration of $G$ such that $\mathbf G\cap T = \gamma^p(T)$
for every $T\in\mathcal T$, then $\gamma^p(G)\cap T=\gamma^p(T)$ for every $T\in\mathcal T$.
\textup{(}Therefore, assuming that $\gamma^p(G)$ is $p$-potent and
$G_{\operatorname{ab}}$ is finitely generated, if there exists a complete filtration of $G$ which is $p$-compatible with $\mathcal T$, then $\gamma^p(G)$ is  $p$-compatible with $\mathcal T$.\textup{)}
\end{lemma}

Examples of $p$-compatible filtrations of free groups may be obtained via the following:

\begin{lemma}\label{lem:p-compatible for free groups}
Let $G$ be a finitely generated residually $p$ group with $p$-torsion free
 quotients $G/\gamma_{n+1}(G)$,  for $n=1,2$,
and let $\mathcal T$ be a collection of retracts of $G$. If $p$ is odd, then the complete filtration $\gamma^p(G)=\{\gamma^p_n(G)\}_{n\geq 1}$ of $G$ is $p$-compatible with~$\mathcal T$. For arbitrary $p$,
the filtration $\{\gamma^p_{n+1}(G)\}_{n\geq 1}$ is $p$-compatible of level $1$ with $\mathcal T$.
\end{lemma}
\begin{proof}
By Lemma~\ref{lem:retract},  $\gamma^p(G)$ separates $T$, and
by Lemma~\ref{lem:lower central for semidirect products} we have $\gamma^p(G)\cap T = \gamma^p(T)$.
This shows that $\gamma^p(G)$ satisfies conditions (1) and (3)--(5) in the definition of $p$-compatibility. If $p$ is odd, then $\gamma^p(G)$ is $p$-potent by Corollary~\ref{cor:phi_n}, hence $\gamma^p(G)$ is $p$-compatible with $\mathcal T$. For general $p$,
shifting $\gamma^p(G)$ by $1$ to the left yields a $p$-potent filtration of $\gamma^p_2(G)$, by Lemma~\ref{lem:2-potent}.
\end{proof}

\label{retract}

By virtue of the previous lemma,  Theorem~\ref{thm:virt res p for graphs of linear groups, 1} applies to graphs of finitely generated residually $p$ groups with abelian $p$-torsion free edge groups, such that  for each edge $e$, the image of the edge group $G_e$ under the edge morphism $f_e$ is a retract of $G_{t(e)}$.
This leads to a refinement of a result of Wise:

\begin{corollary}
Let $\scrg$ be a graph of free groups of finite rank with cyclic edge groups, and $G=\pi_1(\scrg)$. Suppose that $G$ is \emph{balanced,} i.e., there are no $g\in G$, $g\neq 1$, and non-zero integers $m$, $n$ with $m\neq\pm n$ such that $g^m$ and $g^n$ are conjugate. Then for every $p$, $G$ is virtually residually $p$.
\end{corollary}

Wise \cite{Wi00} showed that the balancedness condition formulated in this corollary implies that there exists a morphism $\scrg^*\to\scrg$ of graphs of groups having finite degree and such that $\scrg^*$ is \emph{clean}, that is, each of the (cyclic) edge groups of $\scrg^*$ maps to a free factor of the target (free) vertex group under the edge morphism. He then proved that fundamental groups of clean graphs of free groups are residually finite; in fact, much better, they are subgroup separable. Our Theorem~\ref{thm:virt res p for graphs of linear groups, 1} shows that they are, for each $p$, virtually residually $p$. Note that if $G$ is a balanced group, then $G$ does not contain any Baumslag-Solitar group.
Hsu and Wise \cite{HW10} recently showed
that if $\scrg$ is a graph of free groups with cyclic edge groups and $G=\pi_1(\scrg)$  contains no Baumslag-Solitar group, then $G$ virtually embeds into a right-angled Artin group, and hence has a finite-index subgroup which is residually $p$ for every~$p$ (and
is linear over $\Z$); this also implies the corollary above in a stronger form.

\index{Wise}


\section{$p$-compatible filtrations of linear groups}\label{sec:p-compatible linear}

\noindent
Given a commutative ring $R$, we denote by
$\operatorname{UT}_1(n,R)$ the group of  upper unitriangular $n\times n$-matrices with entries in $R$ (a subgroup of $\SL(n,R)$).  A matrix in $\GL(n,R)$ is called unipotent if it is conjugate to an element of $\operatorname{UT}_1(n,R)$, and a subgroup of $\GL(n,R)$ is called unipotent if it is conjugate to a subgroup of $\operatorname{UT}_1(n,R)$. (If $K$ is a field, then every subgroup of $\GL(n,K)$ consisting entirely of unipotent matrices is unipotent, cf.~\cite[Corollary~1.21]{We73}.) A subgroup $T$ of a group $G\leq\GL(n,R)$ is said to be maximal unipotent if $T$ is unipotent and not contained in a strictly larger unipotent subgroup of $G$.
The main result of this section is:

\index{$\operatorname{UT}_1(n,R)$}
\index{group!unipotent}
\index{group!linear}

\begin{theorem}\label{thm:virt res p for graphs of linear groups, 2}
Let $G$ be a finitely generated linear group, and let $\mathcal T$ be a finite collection of finitely generated abelian unipotent subgroups of $G$, such that for each~$T\in\mathcal T$,
\begin{enumerate}
\item $T$ is maximal abelian or maximal unipotent; and
\item the image of $T$ in $H_1(G;\Z)$ has rank $\geq \rank(T)-1$.
\end{enumerate}
Then for all but finitely many $p$, there exists a filtration of $G$ which is $p$-compatible  of level $1$ with $\mathcal T$.
\end{theorem}

The proof of this theorem is given in Section~\ref{sec:proof of thm:virt res p for graphs of linear groups, 2} below.
Theorem~\ref{thm:virt res p for graphs of linear groups, 2} together with Theorem~\ref{thm:virt res p for graphs of linear groups, 1} immediately yields a criterion for the fundamental group of certain graphs of linear groups to be virtually residually $p$ (which refines \cite[Theorem~5.1]{He87}):

\begin{corollary}\label{cor:virt res p for graphs of linear groups}
Let $\scrg$ be a graph of finitely generated linear groups. Suppose that each edge group $G_e$ maps to a finitely generated abelian unipotent subgroup of~$G_v$ \textup{(}where $v=t(e)$\textup{)} which is maximal abelian or maximal unipotent in $G_v$ and whose image in $H_1(G_v;\Z)$ has rank $\geq\rank(G_e)-1$. Then for all but finitely many~$p$, the group $\pi_1(\scrg)$ is virtually residually $p$.
\end{corollary}

 The hypothesis~(2) on $T$ in the theorem above  is satisfied if $T$ is
 free abelian of rank $2$ and maps to an infinite subgroup of $H_1(G;\Z)$; this is the situation of interest in our later applications. Hypothesis~(2), however, also holds if $T$ is cyclic; hence from the previous corollary we obtain:

\begin{corollary}
Let $\scrg$ be a graph of finitely generated linear groups with cyclic unipotent edge groups. Suppose that for each edge $e$, the image of $G_e$  in $G_{v}$, where $v=t(e)$, is maximal abelian.
Then for all but finitely many $p$, the group $\pi_1(\scrg)$ is virtually residually $p$.
\end{corollary}

Before we turn to the proof of Theorem~\ref{thm:virt res p for graphs of linear groups, 2} we present an instructive example:

\begin{example}[Wehrfritz {\cite[p.~410]{We73-2}}]
Consider the subgroup
$$G=\left\langle
\left(\begin{matrix} 1 & 1 \\ 0 & 1\end{matrix} \right), \left(\begin{matrix} 2 & 0 \\ 0 & 1\end{matrix} \right)
\right\rangle$$
of $\GL(2,\Q)$, which has a presentation $G=\langle a,t \,|\, t^{-1}at=a^2\rangle$ (so $G$ is the Baumslag-Solitar group $\operatorname{BS}(1,2)$). The group $G$ is
metabelian, and its (unique) maximal unipotent subgroup is the normal subgroup
$\left(\begin{smallmatrix} 1 & \Z[\frac{1}{2}] \\ 0 & 1\end{smallmatrix} \right)$
of $G$ (isomorphic to the additive group of $\Z[\frac{1}{2}]$, in particular, not finitely generated).
This shows that in general,
the assumption in our theorem that $G$ be finitely generated does not force its abelian maximal unipotent subgroups to be finitely generated. (However, if $G$ is discrete, then each abelian maximal unipotent subgroup of $G$ is finitely generated, so
 the assumption in our theorem that each member of $\mathcal T$ be finitely generated may be dropped.)

This example can also be used to show that the hypothesis in Corollary~\ref{cor:virt res p for graphs of linear groups} that each edge group maps onto a maximal abelian or maximal unipotent subgroup of the target vertex group cannot be omitted:
Let $T=\left\langle
\left(\begin{smallmatrix} 1 & 1 \\ 0 & 1\end{smallmatrix} \right)\right\rangle$, a unipotent abelian subgroup of $G$ (which, however, is clearly neither maximal abelian nor maximal unipotent). It was shown by Higman \cite{Hig51} that
the amalgamated product $G\ast_{T=T} G$ of $G$ with itself via the identity $T\to T$ is non-Hopfian and hence not residually finite. Of course, examples for this phenomenon (in the form of HNN extensions of $\Z$) are also provided by Baumslag-Solitar groups: as is well-known, for integers $m$, $n$ with $|m|,|n|>1$ and $|m|\neq |n|$, the group $$\operatorname{BS}(m,n)=\langle a,t \,|\, t^{-1}a^mt=a^n \rangle$$ is not residually finite.
\end{example}

\subsection{A localization theorem}
A key algebraic ingredient for the proof of Theorem~\ref{thm:virt res p for graphs of linear groups, 2} in full generality is the following fact.
Let $R$ be a finitely generated subring of $\mathbb C$.
It is well-known that  for all but finitely many $p$ there exists a maximal ideal $\mm$ of $R$ such that $\chr(R/\mm)=p$. (See, e.g., Proposition~\ref{Non-trivial and reduced} below.) Evidently the characteristic of $R/\mm^i$ then divides $p^i$, for every $i\geq 1$.
But in general we do not have $\chr(R/\mm^i)=p^i$.
(Consider, e.g.,  $R=\Z[\sqrt{2}]$, $\mm=(\sqrt{2})$ and $p=2$.) 
The following theorem implies that by judicious choice of $p$ we can exclude this phenomenon:

\begin{theorem}\label{thm:proprmi}
For all but finitely many primes $p$ there exists a maximal ideal~$\mathfrak m$ of $R$ such that the localization $R_{\mathfrak m}$ of $R$ at $\mathfrak m$ is regular unramified of residue characteristic $p$.
\end{theorem}

\index{ring!regular}
\index{ring!unramified local}

We refer to Section~\ref{sec:approx} below for an explanation of the concepts from commutative algebra mentioned in the conclusion of this theorem. For its application in the present section, the reader only needs to be aware of one aspect of this conclusion, namely:

\begin{lemma}\label{lem:dividing by p}
Let $(S,\nn)$ be a regular unramified local ring of residue characteristic~$p$.
Then for all $i\geq j\geq 0$ and $x\in S$:
$$p^jx \equiv 0 \mod \nn^i  \qquad \Longleftrightarrow \qquad x\in\nn^{i-j}.
$$
\textup{(}Here $\nn^0=S$.\textup{)}
\end{lemma}

This follows immediately from Lemma~\ref{Criterion} below.
In particular,  given a prime~$p$ and a  maximal ideal $\mathfrak m$ as in Theorem~\ref{thm:proprmi}, we do have $\operatorname{char}(R/\mathfrak m^i)=p^i$ for every~$i\geq 1$. It may be worth noting that
our theorem thus implies a more precise version of \cite[Proposition~2]{Ha01}, which states that except for a finite number of primes $p$, for every $i\geq 1$ there exists a morphism $R\to R_i$ to some finite ring $R_i$ such that $\chr(R_i)=p^i$.

\medskip

We postpone the proof of Theorem~\ref{thm:proprmi} to Section~\ref{sec:approx}. We do note here, however, that in the case where $R$ is a ring of integers in a number field, this fact is quite easy to show:

\begin{example}
Suppose $R$ is the ring of integers in a number field $K$, and let $\delta$ be the discriminant of $K$ (a non-zero integer). For every $p$ there exists a prime (hence maximal) ideal $\mm$ of $R$, and $R_\mm$ is a DVR (hence regular) of residue characteristic $p$. Moreover, if $p$ does not divide $\delta$, then for every such $\mm$, $R_\mm$ is unramified. (See any standard text in algebraic number theory, e.g., \cite[Chapter~III, Section~2]{Neukirch}.)
\end{example}

This allows us to give a proof of Theorem~\ref{thm:virt res p for graphs of linear groups, 2} in the case relevant to our applications to $3$-manifolds (namely, where $G$ is a lift of a finite covolume torsion free Kleinian group) which bypasses Theorem~\ref{thm:proprmi}; this point will be addressed in Section~\ref{sec:Kleinians}.
Before we now turn the proof of Theorem~\ref{thm:virt res p for graphs of linear groups, 2}, we make some observations on congruence subgroups.

\index{group!Kleinian}

\subsection{Congruence subgroups}\label{sec:csg}
Let $(R,\mathfrak m)$ be a complete Noetherian local ring  whose residue field $R/\mathfrak m$ is finite of characteristic $p$.
For every $i\geq 1$ the natural surjective ring morphism $R\to R/\mathfrak m^i$ induces a surjective group morphism
$$\SL(n,R)\to\SL(n,R/\mathfrak m^i),$$
whose kernel is called the \emph{$i$-th congruence subgroup} of $\SL(n,R)$, and denoted by $\SL^i(n,R)$.
Note that $R/\mm^i$  and hence $\SL(n,R/\mathfrak m^i)$ are finite, for any $i\geq 1$.
In the following we  abbreviate $G_i:=\SL^i(n,R)$, and set $G:=G_1$. If $R$ is an integral domain, then $\bigcap_{i\geq 1} \mm^i=\{0\}$ by the Krull Intersection Theorem, hence the normal filtration $\{G_i\}_{i\geq 1}$ of $\SL(n,R)$ is separating (and thus $\SL(n,R)$ is residually finite).

\index{congruence subgroup}

Proofs of the following properties of congruence subgroups can be found in \cite{LS94} (see also \cite[Chapter~13]{DdSMS}). Given an ideal $\mathfrak a$ of $R$, we write $\mathfrak a^{(i)}=\mathfrak a\times\cdots\times\mathfrak a$ for the $i$-th cartesian power of the set $\mathfrak a$ (not to be confused with the $i$-th power $\mathfrak a^i$ of the ideal $\mathfrak a$).

\begin{proposition}\label{prop:LS94}
For all integers $i,j\geq 1$:
\begin{enumerate}
\item $G_i/G_{i+1}\cong (\mathfrak m^i)^{(n^2-1)}/(\mathfrak m^{i+1})^{(n^2-1)}$ \textup{(}an elementary abelian $p$-group\textup{)};
\item $[G_i, G_j] \subseteq G_{i+j}$;
\item $(G_i)^p \subseteq G_{ip}$ if $R$ has equal characteristic $p$.
\end{enumerate}
\end{proposition}

In particular, by (1) and (2), $\{G_i\}_{i\geq 1}$ is a (strongly) central $p$-filtration of $G=G_1$, hence $G_i\geq\gamma^p_i(G)$ for every $i\geq 1$.
If $R$ has equal characteristic $p$, then $\{G_i\}_{i\geq 1}$ is also dimensional, by (3), so $G_i\geq D^p_i(G)$ for every $i\geq 1$.
Under additional assumptions, these statements can be strengthened:

\begin{proposition}{\cite[Proposition~3.2]{LS94}}
If $p\neq 2$ or $n\neq 2$, then
$$\gamma_i(G) = \gamma^p_i(G) = G_i\qquad\text{for all $i\geq 1$,}$$
and if in addition $R$ is of equal characteristic $p$, then
$$D^p_i(G)=G_i\qquad\text{for all $i\geq 1$.}$$
\end{proposition}

The following lemma combined with the previous proposition shows that for odd $p$ and well-behaved $R$, the filtration $\{G_i\}_{i\geq 1}$ of $G$ is strongly $p$-potent:

\begin{lemma}
Suppose that $p$ is odd and $(R,\mathfrak m)$ is regular unramified. Then for every $i\geq 1$, the morphism
$\phi_i\colon G_i/G_{i+1} \to G_{i+1}/G_{i+2}$ induced by $M\mapsto M^p$ is injective.
\end{lemma}
\begin{proof}
Let $M\in G_i=\operatorname{SL}^i(n,R)$, $i\geq 1$, such that $M^p\in G_{i+2}$; we need to show $M\in G_{i+1}$. Write $M=\id+N$ where $N\in (\mathfrak m^i)^{n\times n}$. Then
$$pN + {p\choose 2}N^2+\cdots+pN^{p-1}+N^p=M^p-\id\in (\mathfrak m^{i+2})^{n\times n}.$$
We have $N^p\in (\mathfrak m^{i+2})^{n\times n}$ since $p>2$, and $pN^2\in (\mathfrak m^{i+2})^{n\times n}$.
Moreover, for $j=1,\dots,p-1$, every binomial coefficient ${p\choose j}$ is divisible by $p$. Therefore $pN\in (\mathfrak m^{i+2})^{n\times n}$. Since $(R,\mathfrak m)$ is regular unramified of residue characteristic $p$, this yields $N\in(\mathfrak m^{i+1})^{n\times n}$ as required (Lemma~\ref{lem:dividing by p}).
\end{proof}

\subsection{Proof of Theorem~\ref{thm:virt res p for graphs of linear groups, 2}}
\label{sec:proof of thm:virt res p for graphs of linear groups, 2}
Let $R$ be a commutative ring. Given an $n\times n$-matrix $N=(N_{ij})$ with entries in $R$ and $k\in\{0,\dots,n-1\}$, we call the $(n-k)$-tuple $(N_{j+k,j})_{j=1,\dots,n-k}$ the $k$th codiagonal of $N$. So the $0$th codiagonal is the main diagonal of $N$; an upper triangular matrix whose main diagonal is zero is called strictly upper triangular.
If $N$ is strictly upper triangular and the $i$th codiagonal of  $N$ vanishes for $i=0,\dots,k$, then $N^{n-k}=0$.
The following lemma is obvious if $n=2$ (the case of interest in the situation arising from $3$-dimensional topology considered in Section~\ref{sec:proof of the main theorem}).

\begin{lemma}\label{lem:order}
Suppose $R$ has characteristic $p^d$ where $d\geq 1$ and $p\geq n$. Then $\operatorname{UT}_1(n,R)$ has exponent $p^d$. If $N$ is a non-zero strictly upper triangular $n\times n$-matrix over $R$ and  some  entry in the first non-zero codiagonal of $N$ is  a unit in $R$, then
$M=\id+N\in\operatorname{UT}_1(n,R)$ has order $p^d$.
\end{lemma}
\begin{proof}
Note that every element $M$ of $\operatorname{UT}_1(R,n)$ has the form $M=\id+N$ where $N$ is strictly upper triangular (in particular, $N^n=0$); hence
$$M^{p^d}=(\id+N)^{p^d} = \sum_{i=0}^{n-1} {p^d\choose i} N^i = \id$$
since for $i=1,\dots,p-1$, the binomial coefficient ${p^d\choose i}$ is divisible by $p^d$. Now suppose in addition that $N\neq 0$ and some entry in the first non-zero codiagonal of $N$ is a unit in $R$. Say this entry occurs in the $k$th codiagonal of $N$, where $k\in\{1,\dots,n-1\}$. Then all entries in the $k$th codiagonal of $N^2,N^3,\dots$ are zero. Since for every $i\geq 1$,
$$(\id+N)^i=\id+iN+\text{linear combination of $N^2, N^3,\dots$}$$ and a unit of $R$ has order $p^d$ as element of the additive group of $R$, we see that~$\id+N$ has order $p^d$.
\end{proof}

Let now $G$ be a finitely generated linear group, and suppose $\mathcal T=\{T_1,\dots,T_r\}$ is a collection of finitely generated abelian unipotent subgroups of $G$, such that each $T_i$ is maximal abelian or maximal unipotent, and
$$d_i := \dim(\ker(H_1(T_i;\Q)\to H_1(G;\Q))) \leq 1\quad\text{for $i=1,\dots,r$.}$$
We write $T_i$ additively; also,
the natural morphism $T_i\to H_1(T_i;\Z)$ is an isomorphism, and in the following we identify these two groups using this map.
We pick a family $\{\lambda_{ij}\}_{1\leq j\leq e_i}$ of elements of $T_i$ which form a basis for $T_i\otimes_\Z\Q=H_1(T_i;\Q)$. In the case that $d_i=1$ we choose $\lambda_{i1}$ to lie in $\ker(H_1(T_i;\Z)\to H_1(G;\Z))$.

In order to prove the theorem we may and shall assume that $G\leq\SL(n,\C)$ for some $n\geq 2$.
Take $Q_i\in\SL(n,\C)$ such that $Q_i T_i Q_i^{-1}\leq \operatorname{UT}_1(n,\C)$.  Since  $G$ is finitely generated there is a finitely generated subring $R$ of $\C$ which contains all the entries of each element of $G$, as well as all entries of $Q_i^{\pm 1}$ for $i=1,\dots,r$ and the multiplicative inverses of the non-zero entries of
$Q_i \lambda_{i1} Q_i^{-1}$, for $i=1,\dots,r$.
By Theorem~\ref{thm:proprmi} there exists a finite set $P$ of primes such that for any prime $p\not\in P$ there exists a maximal ideal $\mm$ of $R$ such that $R_{\mm}$ is regular unramified of residue characteristic $p$.

We  write $H=H_1(G;\Z)/\operatorname{tor}$. As a consequence of the Elementary Divisors Theorem for subgroups of finitely generated free abelian groups \cite[Theorem~III.7.8]{La02}, after enlarging $P$ suitably if necessary, we may assume that for any~$p\not\in P$,   any $i\in \{1,\dots,r\}$ and any $k\geq 1$ we have
\[ \im\left(H_1(T_{i};\Z)\to H_1(G;\Z)\to H\to H/p^k H\right)\cong (\Z/p^k\Z)^{e_i-d_i}. \]
In the case that $d_i=1$, the element $\l_{i1}$ maps to zero under the map $H_1(T_{i};\Z)\to H/p^k H$. Again, after enlarging $P$ if necessary, we may also assume that $P$ contains all primes $\leq n$.

Let now $p$ be a prime with $p\notin P$. Let  $\mm$ be a maximal ideal of $R$ such that the local ring $R_{\mm}$, whose completion we denote by $\Lambda:=\widehat{R_{\mm}}$, is regular unramified of residue characteristic $p$. The natural ring morphisms $R\to R_{\mm} \to \Lambda = \widehat{R_{\mm}}$ are injective, hence give rise to a group embedding $\SL(n,R)\to\SL(n,\Lambda)$. We denote the map $G\to \SL(n,R)\to\SL(n,\Lambda)$ by $\rho$, and
the map $G\to H_1(N;\Z)/\operatorname{tor}=H$ by~$\theta$.
Given $k\geq 1$ we denote the map
$$G\xrightarrow{\rho} \SL(n,\Lambda)\to
\SL(n,\Lambda/\mm^k \Lambda) = \SL(n,R/\mm^k)$$
by~$\ol{\rho}$, and we denote the map
$$G\to H \to H/p^kH$$
by~$\ol{\theta}$.
For $k\geq 1$ we now let
\begin{align*}
 G_k &= \ker\left( G\xrightarrow{\ol{\rho}\times \ol{\theta}} \SL(n,R/\mm^k)\times H/p^k H\right) \\
 &= (\rho\times\theta)^{-1}\big(\SL^k(n,\Lambda)\times p^k H\big).
\end{align*}
We claim that the filtration  $\mathbf G=\{G_{k}\}_{k\geq 1}$ of $G$ is $p$-compatible with $\mathcal T$ of level~$1$. By the results in Section~\ref{sec:csg},
$\mathbf G$ is a separating normal $p$-filtration of $G$ with each~$G_k$ of finite index in $G$, and considered as a complete filtration of $G_1$, $\mathbf G$ is central and (strongly) $p$-potent.

To see that $\mathbf G$ separates $\mathcal T$, let $i\in\{1,\dots,r\}$ and $g\in G\setminus T_i$. If $T_i\leq G$ is maximal unipotent, we have $Q_i g Q_i^{-1}\notin\operatorname{UT}_1(n,R)$; since $\bigcap_{k\geq 1}\mm^k=\{0\}$ there exists $k\geq 1$ such that $\ol{\rho}(Q_i g Q_i^{-1})\notin\operatorname{UT}_1(n,R/\mm^k)$.
If $T_i$ is maximal abelian, then there is some $t\in T$ with $[g,t]\neq 1$; as before, there is some $k$ such that $1\neq \ol{\rho}([g,t]) = [\ol{\rho}(g),\ol{\rho}(t)]$ in $\SL(n,R/\mm^k)$.
In both cases we obtain $\ol{\rho}(g)\notin\ol{\rho}(T_i)$, i.e., $g\notin G_k\cdot T_i$.

It remains to show that
$$T_i\cap G_k = p^k T_i = \gamma_{k+1}^p(T_i)\qquad\text{for all $k\geq 1$.}$$
Since $T_i$ (being unipotent) is torsion free, we have
$T_i/p^kT_i\cong (\Z/p^k\Z)^{e_i}$;
we need to show that
$$I:=\im\left(T_i\xrightarrow{\ol{\rho}\times \ol{\theta}}  \SL(n,R/\mm^k)\times H/p^kH\right)\cong
(\Z/p^{k}\Z)^{e_i}.$$
Recall that
$$\im\left( T_i \xrightarrow{\ol{\theta}} H/p^kH\right) \cong (\Z/p^k\Z)^{e_i-d_i}.$$
Note that by choice of $R$ and $\mathfrak m$, Lemma~\ref{lem:order} implies that the matrix $\ol{\rho}(\lambda_{i1})$ has order  $p^{k}$ in $\SL(n,R/\mm^k)$.
So in the case that $d_i=0$ we clearly have $I\cong (\Z/p^{k}\Z)^{e_i}$ as desired. If $d_i=1$ then $\ol{\theta}(\lambda_{i1})=0\in H/p^kH$, whereas $\ol{\rho}(\lambda_{i1})$ has order $p^k$; thus also in this case $I\cong (\Z/p^{k}\Z)^{e_i}$. \qed

\begin{remark}
Suppose in the proof of the theorem, in addition to $p\notin P$ we also require that $p$ is such that $\gamma_2(G)/\gamma_3(G)$ is $p$-torsion free. Then the filtration $\{\gamma^p_{k+1}(G)\}_{k\geq 1}$ of $\gamma^p_2(G)$ is $p$-potent, by Lemma~\ref{lem:2-potent}. Moreover, by Proposition~\ref{prop:lower p},~(2) we have
$$\gamma^p_k(G)\leq \ker(G\to H\to H/p^{k-1}H)\quad\text{for all $k\geq 1$,}$$
with equality for $k=1$.
Hence (identifying $G$ with its image in $\SL(n,\Lambda)$ under $\rho$), the filtration $\mathbf G^*=\{G^*_k\}_{k\geq 1}$ of $G$ defined by
$$G^*_k = \SL^k(\Lambda,n) \cap \gamma^p_{k+1}(G)$$
is also $p$-compatible with $\mathcal T$ of level $1$, cf.~Lemma~\ref{lem:canonical filtration}.
\end{remark}

\section{Proof of the main theorem}
\label{sec:proof of the main theorem}

\noindent
In this section we give a proof  of the main theorem as stated in the introduction, by applying the results of the previous section. We begin by first discussing the special cases of hyperbolic $3$-manifolds and Seifert fibered manifolds.

\subsection{Hyperbolic $3$-manifolds}\label{sec:hyperbolic}
It is a well-known fact that the fundamental group of every hyperbolic $3$-manifold is linear and hence virtually residually $p$, for all but finitely many $p$.
We need a more precise formulation of this fact:

\index{$3$-manifold!hyperbolic}

\begin{definition}
Let $N$ be a $3$-manifold which is either closed or has  
incompressible boundary. Let $\mathcal T$ be the collection of fundamental groups of the boundary components of $N$, considered as subgroups of $G=\pi_1(N)$ in the natural way.
Given a prime $p$ and an integer $\ell\geq 0$  we say that a filtration $\mathbf G$ of $G$ is a \emph{boundary compatible  $p$-filtration of level $\ell$} for $N$ if $\mathbf G$ is $p$-compatible of level $\ell$ with the collection of subgroups $\pi_1(T)$, $T\in\mathcal T$.
\end{definition}

\index{$p$-filtration!boundary compatible}

\begin{proposition}\label{prop:virtphyp}
Let $N$ be an orientable $3$-manifold with empty or toroidal and incompressible boundary such that the interior of $N$ has a complete hyperbolic structure.  Then for all but finitely many $p$, there exists a boundary compatible  $p$-filtration of level $1$ for $N$.
\end{proposition}

We need a topological lemma sketched in \cite[p.~390]{He87}: 

\begin{lemma}\label{lem:maximal abelian}
Let $N$ be an orientable $3$-manifold with toroidal incompressible boundary, and let $T$ be a component of $\partial N$. Then $\pi_1(T)$ is a maximal abelian subgroup of $\pi_1(N)$, i.e., no element of $\pi_1(N)\setminus\pi_1(T)$ commutes with $\pi_1(T)$.
\end{lemma}
\begin{proof}
Suppose $g\in \pi_1(N)\setminus\pi_1(T)$ and $\pi_1(T)$  generate an  abelian subgroup~$A$ of $\pi_1(N)$ which properly contains $\pi_1(T)$.
Since $N$ is not a $3$-torus, $A\cong \Z\oplus\Z$ by \cite[Theorem~9.13]{He76}.
Denote by $\widetilde{N}\to N$ the covering of $N$ corresponding to $A$.  There exists a compact submanifold $M$ of $\widetilde{N}$ such that $\pi_1(M)\to \pi_1(\widetilde{N})$ is an isomorphism  \cite[Theorem~8.6]{He76}.
We can assume that $M$ contains a homeomorphic lifting $S$ of $T$ in its boundary.
Since $\pi_1(T)$ is of finite index in~$A$ and since $N$ is orientable, we can apply  \cite[Theorem~10.5]{He76}   to conclude that $M$ is a product
$S\times [0,1]$ with possibly some $3$-balls removed. In particular $\pi_1(S)\to \pi_1(M)$  is an isomorphism, which implies that $\pi_1(T)\to A=\pi_1(\widetilde{N})$ is an isomorphism, contradicting our assumption that  $g\notin \pi_1(T)$.
\end{proof}

Now let $N$ be an orientable $3$-manifold with empty or toroidal and incompressible boundary. Suppose moreover that $N$ is hyperbolic, so $G=\pi_1(N)$ has a faithful representation to a discrete torsion free  subgroup of $\operatorname{PSL}(2,\C)$. Thurston showed (cf., e.g., \cite[Proposition~3.1.1]{CS83} or \cite{Kr85}) that this representation lifts to a discrete and faithful representation  $G\to \SL(2,\C)$; we identify $G$ with its image under this embedding. Let now $T$ be one of the (toroidal) components of $\partial N$.
Standard Poincar\'e duality arguments yield that the natural image of $\pi_1(T)\cong\Z\oplus\Z$ in~$H_1(G;\Z)$ has rank $\geq 1$.
Every discrete abelian subgroup of $\SL(2,\C)$ of rank $2$ is unipotent (cf.~\cite[p.~382]{He87}). Hence $\pi_1(T)$ is unipotent; in fact, by the previous lemma, $\pi_1(T)$ is a maximal unipotent subgroup of $G$. Proposition~\ref{prop:virtphyp} now is a consequence of Theorem~\ref{thm:virt res p for graphs of linear groups, 2}. \qed

\subsection{$p$-compatible filtrations of lifts of Kleinian groups}\label{sec:Kleinians}

In the proof of Theorem~\ref{thm:virt res p for graphs of linear groups, 2} in the previous section, given a finitely generated subgroup $G$ of $\SL(n,\C)$, we chose a certain finitely generated subring $R$ of $\C$ so that $G\leq\SL(n,R)$, and we appealed to Theorem~\ref{thm:proprmi} to obtain, for all but finitely many $p$, a maximal ideal $\mathfrak m$ of $R$ such that $\chr(R/\mm^k)=p^k$ for every $k\geq 1$. Above, this result has been applied in the context
where $G$ is a lift from $\operatorname{PSL}(2,\C)$ of a finite covolume torsion free Kleinian group. Thurston observed that then
the Mostow Rigidity Theorem implies that $G$ is conjugate (in $\SL(2,\C)$) to a subgroup of $\SL(2,K)$ where $K$ is a number field. (In fact,
 $K$ can be taken to be a quadratic extension of the trace field $\Q(\tr g:g\in G)$ of $G$, see \cite[Corollary~3.2.4]{MR03}.) Hence in this case, the ring $R$ introduced in the proof of Theorem~\ref{thm:virt res p for graphs of linear groups, 2} may be assumed to be a subring of a ring $\mathcal O$ of $S$-integers of a number field $K$. The ring $\mathcal O$ is not finitely generated in general. However, all but finitely many $p$ are unramified in $K$ and do not become units in~$\mathcal O$, and for those $p$, if we choose the prime ideal $\mathfrak p$ of $\mathcal O$ which contains $p$, then $\mathfrak p$ is maximal and
$\mathcal O_{\mathfrak p}$ is regular unramified with finite residue field of characteristic $p$.
We may now carry out the rest of the argument in the proof above with $R$ and $\mm$ replaced by $\mathcal O$ respectively $\mathfrak p$. Thus, in this application of Theorem~\ref{thm:virt res p for graphs of linear groups, 2} one may avoid the use of  Theorem~\ref{thm:proprmi} at the cost of appealing to a non-trivial fact from hyperbolic geometry. 
In general, however, Theorem~\ref{thm:virt res p for graphs of linear groups, 2} cannot simply be reduced to the number field case: there do exist finitely presented linear groups which admit no faithful representation as linear groups over global fields whatsoever \cite{GS79}.

\index{group!Kleinian}
\index{Thurston}

\subsection{Seifert fibered spaces}\label{sec:seifert}
In the case of Seifert fibered spaces we have:

\index{$3$-manifold!Seifert fibered}

\begin{proposition}\label{prop:Seifert res p}
Let $N$ be a Seifert fibered space. Then
$\pi_1(N)$ has a finite index subgroup which is residually $p$ for every $p$.
\end{proposition}

Rhemtulla \cite{Rh73} showed that every group which is residually $p$ for infinitely many $p$ is bi-orderable. Hence the previous proposition also yields another proof of the result (obtained in \cite[Corollary~1.6]{BRW05}) that fundamental groups of Seifert fibered spaces are virtually bi-orderable.
By similar arguments as the ones used below in the proof of this proposition one may also show the folklore result that the fundamental group of a Seifert fibered space is linear over $\Z$; we will not pursue this here. 



%

\begin{lemma}\label{lem:S1-bundles res p}
Suppose $N\to F$ is an $S^1$-bundle where $F$ is an orientable surface with $\chi(F)\leq 0$. Then $\pi_1(N)$ has a finite index subgroup which is residually $p$ for each $p$.
\end{lemma}
\begin{proof}
If $N$  has non-empty boundary, then $F$ also has non-empty boundary and we obtain $H^2(F;\Z)=0$, hence the Euler class of the $S^1$-bundle $N\to F$ is trivial. We therefore conclude that $N$ is a trivial $S^1$-bundle, i.e., $N=S^1\times F$. But then $\pi_1(N)=\Z\times \pi_1(F)$ is the direct product of $\Z$ with a free group,  hence $\pi_1(N)$ is residually $p$ for each $p$.

Now assume that $N$ and hence $F$ is closed. The subgroup $\ll t\rr$ of $\pi_1(N)$ generated by a regular fiber $t$ is normal and infinite cyclic, and
we have a short exact sequence
\[ 1\to \ll t\rr \to \pi_1(N)\to \pi_1(F)\to 1.\]
Let $e\in H^2(F;\Z)\cong \Z$ be the Euler class of $F$. A presentation for $G:=\pi_1(N)$ is given by
\[ G = \left\langle a_1,b_1,\dots,a_r,b_r,t : \prod_{i=1}^r [a_i,b_i]=t^e,\ \text{$t$ central}\right\rangle.\]
Let
\[ H = \left\langle a,b,u : [a,b]=u,\ \text{$u$ central}\right\rangle.\]
One verifies easily that the finitely presented group $H$ is torsion free, $Z(H)=\langle u\rangle$, and $H/Z(H)\cong \Z\times \Z$ (hence $H$ is nilpotent). Thus $H$ is residually $p$ for every $p$ \cite[Theorem~2.1~(i)]{Gru57}. Let $G_e$ be the subgroup of $G$ generated by the $a_i$, $b_i$ and $t^e$.
The assignment
\[ \rho(a_1)=a,\quad \rho(b_1)=b,\quad \rho(a_i)=\rho(b_i)=1\ \text{for $i\geq 2$,}\quad\rho(t^e)=u\]
yields a morphism $\rho\colon G_e\to H$. Let also
$\sigma$ be the composition $$G_e\to G=\pi_1(N)\to \pi_1(F).$$
Then $$\rho\times\sigma\colon G_e\to H\times\pi_1(F)$$
is an embedding. It is well-known that orientable surface groups, such as $\pi_1(F)$, are residually free (cf.~\cite{Ba62}) and hence residually $p$ for each $p$.
This shows that the finite index subgroup $G_e$ of $G$ is residually $p$ for each $p$.
\end{proof}

\index{surface group}

\begin{proof}[Proof of Proposition~\ref{prop:Seifert res p}]
Clearly we only have to consider the case that the fundamental group of $N$ is infinite.
In that case it is well-known that $N$ is finitely covered by a $3$-manifold  which is an $S^1$-bundle over an orientable surface $F$ with $\chi(F)\leq 0$ (cf.~\cite{He76}). Now use Lemma~\ref{lem:S1-bundles res p}.
\end{proof}

Analogous to Proposition~\ref{prop:virtphyp} we also have,  for certain trivial $S^1$-bundles:


\begin{proposition}
\label{prop:seifert simple lower p-central}
Suppose $N=S^1\times F$ where $F$ is an orientable surface with at least two boundary components, and $G=\pi_1(N)$.
Then  for any $p$ there exists a boundary compatible $p$-filtration of level $1$ for $N$.
\end{proposition}

\begin{proof}
This follows immediately from Lemma~\ref{lem:p-compatible for free groups}, observing that since $F$ has at least two boundary components, its fundamental group $\Gamma=\pi_1(F)$ is free, and for each $c\in\Gamma$ representing a boundary curve of $F$, the subgroup $C=\langle c\rangle$ of $\Gamma$ generated by $c$ is a free factor of $\Gamma$.
\end{proof}

\begin{remark}
Let $F$ be an orientable surface. It is a classical fact (Fricke)  that $\G=\pi_1(F)$ has a faithful representation as a Fuchsian group 
such that each boundary curve $c$ of $F$ is represented by a parabolic transformation \cite[pp.~98--99]{Ma73}.
That is, there is a faithful representation
$\varrho \colon G\to\operatorname{PSL}(2,\R)$ where each $c$ is represented by a unipotent matrix in $\SL(2,\R)$.
Every such $\varrho$ lifts to a representation $\G\to\SL(2,\R)$. In fact, one may choose $\varrho$ to have image in $\SL(2,\Q)$; this follows from Teichm\"uller theory \cite[Theorem~1]{WM85} (see also \cite[Section~4.9]{MR03}).
One may therefore also try to apply Theorem~\ref{thm:virt res p for graphs of linear groups, 2} to prove the proposition above (even if $b_0(\partial F)=1$).
However, this only yields the existence of a boundary compatible $p$-filtration of level~$1$ for $N=S^1\times F$ for cofinitely many $p$.
\end{remark}

\index{group!Fuchsian}

\subsection{Proof of the main theorem}\label{subsec:proof of the main theorem}
Let $N$ be a $3$-manifold.
We have to show that for all but finitely many $p$ the group $G=\pi_1(N)$ is virtually residually $p$.

\medskip

\subsubsection{Reduction to the case of orientable and closed prime manifolds}

By gluing in $3$-balls if necessary we can always assume, without changing $G$, that $\partial N$ has no spherical components. In this situation we have:

\begin{lemma} \label{lem:incomp}
Suppose $N$ has no spherical boundary components. Then there exist $3$-manifolds $N_1,\dots,N_k$ whose boundaries are incompressible and without spherical components, and a free group $F$
such that $\pi_1(N)\cong \pi_1(N_1)*\dots *\pi_1(N_k)*F$.
\end{lemma}

\begin{proof}
Let $\Sigma\subseteq \partial N$ be a compressible component.
 By Dehn's Lemma there exists a properly  embedded disk
$D\subseteq N$ such that  $\partial D \subseteq \Sigma$ is homologically essential.

Let $N'$ be the result of capping off the spherical boundary components of $N\sm \nu D$ by $3$-balls.
If $N'$ is connected, then $\pi_1(N)=\pi_1(N')\ast \Z$, otherwise $\pi_1(N)=\pi_1(N_1)*\pi_1(N_2)$
where $N_1$, $N_2$ are the two components of $N'$.
 The lemma now follows by induction on
the lexicographically ordered pair $(-\chi(\partial N),b_0(\partial N))$ since we have either $-\chi(\partial N')<-\chi(\partial N)$ (in the case that $\Sigma$ is not a torus), or $\chi(\partial N')=\chi(\partial N)$ and $b_0(\partial N')<b_0(\partial N)$ (in the case that $\Sigma$ is a torus).
\end{proof}

Combining Lemma~\ref{lem:incomp} with the remark following Corollary~\ref{cor:trivial edge groups} we see that it suffices to consider $3$-manifolds with incompressible boundary.
Suppose $N$ has  non-empty incompressible  boundary. Then we consider the double $N\cup_{\partial N}N$, which is a closed $3$-manifold.
Then the natural morphism $\pi_1(N)\to \pi_1(N\cup_{\partial N} N)$ is injective, and since the property of being virtually residually $p$ restricts to subgroups we can from now on assume that $N$ is in fact closed.
By going to a double cover, if necessary, we can and will also assume that $N$ is orientable.
Finally, by the remarks after Corollary~\ref{cor:trivial edge groups} again it suffices to show that prime $3$-manifold
 groups
are virtually residually $p$ for all but finitely many $p$.

\medskip

\subsubsection{The setup for the case of a closed orientable prime $3$-manifold}\label{sec:JSJ}
Now suppose $N$ is a closed orientable prime $3$-manifold.
If $N$ is hyperbolic or Seifert fibered,
it follows immediately from Propositions~\ref{prop:virtphyp} respectively
\ref{prop:Seifert res p} that $G$ is virtually residually $p$ for all but finitely many $p$.

Assume that $N$ has a non-trivial JSJ decomposition; that is, there are
pairwise distinct, incompressible embedded tori $T_1,\dots,T_r\subseteq N$ ($r>0$) such that the components  $N_{1},\dots,N_{s}$ of $N$ cut along $T_1,\dots,T_r$ are either Seifert fibered or algebraically atoroidal (and thus, by geometrization, hyperbolic).
We associate a graph~$Y$ to this decomposition of $N$ as follows: take a set $V=\{v_1,\dots,v_s\}$ of $s$ distinct vertices and an orientation $E_+=\{e_1,\dots,e_r\}$ of its edge set (with $|E_+|=r$), and for every~$i$ choose distinct vertices $v_j$ and $v_k$ such that $T_i\subseteq \partial N_{j}$ and $T_i\subseteq \partial N_{k}$ and set $o(e_i)=v_j$, $t(e_i)=v_k$.
We may describe $N$ as the topological realization of a graph $\scrn$ of based CW-complexes with underlying graph $Y$, whose vertex spaces are the $N_j$ and whose edge spaces are the tori $T_i$. We let $\scrg$ be the associated graph of groups with underlying graph $Y$ (cf.~Section~\ref{sec:graphs of cw-complexes}); so $G=\pi_1(N)=\pi_1(\scrg)$.

\index{JSJ decomposition}
\index{$3$-manifold!Seifert simple}

\medskip

\subsubsection{Reduction to Seifert simple manifolds} \label{sec:reduction to seifert simple}
In the following we say that a closed orientable prime $3$-manifold is \emph{Seifert simple} if
 any Seifert fibered component in its JSJ decomposi\-tion is of the form $S^1\times F$ for some orientable surface $F$, where~$F$ is either closed or has at least two boundary components.
We now show that there exists a finite cover $\widetilde{N}\to N$ such that $\widetilde{N}$ is Seifert simple.

\medskip

We first prove the following lemma which is a slight variation on well-known lemmas  (cf., e.g., \cite[p.~391]{He87}, \cite[Lemma~6]{Ha01}, or \cite[Proposition~2.5]{PS99}).

\begin{lemma}\label{lem:sfsq}
Let $M$ be a Seifert fibered manifold with non-empty incompressible boundary. Let $q$ be a prime with $q>5$.
Then there exists a finite index subgroup~$\widetilde{\Gamma}$ of $\Gamma:=\pi_1(M)$ with the following properties:
\bn
\item
 for any torus $T\subseteq \partial M$ the group $\pi_1(T)\cap \widetilde{\Gamma}$ is the unique characteristic subgroup of $\pi_1(T)$ of index $q^2$;
 and
\item the cover $\widetilde{M}$ corresponding to $\widetilde{\Gamma}$
is of the form $S^1\times F$ for some orientable surface $F$ with $b_0(\partial F)\geq 2$.
\en
\end{lemma}

\begin{proof}
Let $q$ be a prime with $q>5$.
Denote by $\ol{F}$ the base orbifold of $M$ and denote the boundary curves by $c_1,\dots,c_k$ ($k\geq 1$).
We denote by $O$ the $2$-orbifold with underlying space $D^2$ and cone angle $2\pi/q$ and finally we denote by $O_1,\dots,O_k$ copies of $O$.
 We let $\ol{F}_q$ be the closed orbifold given by gluing to  each boundary component $c_i$ of $\ol{F}$ the $2$-orbifold $O_i$.
 Note that
 \[ \pi_1(\ol{F}_q)=\pi_1(\ol{F})/\ll c_1^q,\dots,c_k^q\rr.\]
In particular the images of $c_1,\dots,c_k$ are now elements of order $q$ in the orbifold fundamental group $\pi_1(\ol{F}_q)$.
Since $q>6$ we see that $\ol{F}_q$ is a good orbifold (see \cite[Theorem~13.3.6]{Th80}), i.e., there exists an epimorphism $\a\colon\pi_1(\ol{F}_q)\to G$ to a finite group $G$
such that the corresponding  finite cover $g\colon F_q\to \ol{F}_q$ has the property that  $F_q$ is a surface.
Note that we can and will assume that   $|G|>q$.
The preimages of $O_i$ are disks, and this  implies that the image of $\a(\pi_1(O_i))\in G$ has order $q$. Since $\pi_1(O_i)=\ll c_i  : c_i^q=1\rr$
it follows that $\a(c_i)\in G$ has order $q$.

Also note that  the preimage of $\ol{F}\subseteq \ol{F}_q$ under $g$ is a surface $F$. Since $\pi_1(\ol{F})\to \pi_1(\ol{F}_q)\to G$ is surjective we see that ${F}$ is connected.
Note that
\[ b_0(\partial F)=\sum_{i=1}^k b_0(\pi^{-1}(c_i))=k\frac{|G|}{q}>1. \]
We now return to the study of $M$.
First note that since $M$ has non-empty incompressible boundary we know that $\G=\pi_1(M)$ is infinite
and hence   the subgroup $\ll t\rr \leq \Gamma$ generated by a regular fiber is normal and infinite cyclic (cf.~\cite[p.~122]{Br93}).
Note that there exists a canonical isomorphism  $\pi_1(M)/\ll t\rr \cong \pi_1(\ol{F})$.
We denote by~$\b$ the map $\pi_1(M)\to G$ defined as the composition
\[ \pi_1(M)\to \pi_1(M)/\ll t\rr =\pi_1(\ol{F})\to \pi_1(\ol{F}_q)\xrightarrow{\a} G.\]
Now let $\pi\colon\widehat{M}\to M$ be the cover corresponding to $\b$.
Then $\widehat{M}$ is a Seifert fibered space with base orbifold the surface $F$, i.e., $\widehat{M}$ is a $S^1$-bundle over the surface $F$ (see also  \cite[Theorem~11.10]{He76}).
As in the proof of Lemma \ref{lem:S1-bundles res p} we conclude that $\widehat{M}$ is in fact a product $\widehat{M}=S^1\times {F}$.

The boundary components of $M$ are given by $T_i:=c_i\times S^1$ ($i=1,\dots,k$), in particular $\pi_1(T_i)$ is the free abelian group generated by $t$ and $c_i$. Also note that
$ \pi_1(T_i)\cap \ker(\b)$ is generated by $t$ and $c_i^q$.
Finally note that $\pi_1(\widehat{M})=\ll t\rr \times \pi_1(F)$. It is easy to see that the cover of $\widehat{M}$ given by $$\pi_1(\widehat{M})\to \ll t\rr \times \pi_1(F)\to \ll t\rr =\Z\to \Z/q\Z$$
gives rise to a cover of $M$ with the desired properties.
\end{proof}

We now let $q>5$ be a prime such that all hyperbolic manifolds $N_j$ admit a subgroup  $\widetilde{G_{v_j}}\leq G_{v_j}=\pi_1(N_{j})$
such that for any boundary component~$T_{i}$ of~$N_j$, $\pi_1(T_{i}) \cap \widetilde{G_{v_j}}$
is the unique characteristic subgroup  of $\pi_1(T_{i})$ of index $q^2$. We can find such $q$ and such subgroups by Proposition \ref{prop:virtphyp}.
Furthermore, given any $j\in \{1,\dots,s\}$ such that $N_j$ is  Seifert fibered  we now pick a subgroup $\widetilde{G_{v_j}}\leq G_{v_j}=\pi_1(N_{j})$ as in the above lemma.
The collection of subgroups $\widetilde{G_{v_j}}$ is clearly
a compatible collection of finite index subgroups of the graph of groups $\scrg$.
We now let $\widetilde{N}\xrightarrow{\pi} N$ be the finite cover arising from (the proof of) Proposition~\ref{prop:commoncover}.
Recall that the manifold $\widetilde{N}$ is in a canonical way the topological realization of a graph $\widetilde{\scrn}$ of based CW-complexes
where the vertex spaces are the components of $\pi^{-1}(N_j)$, $j=1,\dots,s$ and the edge spaces are the components of $\pi^{-1}(T_i)$, $i=1,\dots,r$.
It now follows that  $\widetilde{N}$ is Seifert simple.
Hence, after replacing $N$ by $\widetilde{N}$ if necessary, we may assume that $N$ is Seifert simple.

\medskip

\subsubsection{Conclusion of the proof} \label{sec:conclusion}
By the discussion above it  remains to prove the main theorem  in the case where $N$ is a closed orientable prime $3$-manifold which has a non-trivial JSJ decomposition and is Seifert simple. In this situation, by Propositions~\ref{prop:virtphyp} and \ref{prop:seifert simple lower p-central}, for all but finitely many $p$, each JSJ component of~$N$ admits a boundary compatible $p$-filtration of level $1$, and so by Theorem~\ref{thm:virt res p for graphs of linear groups, 1}, the group $G=\pi_1(\scrg)$ is virtually residually $p$. \qed



\section{A localization theorem}\label{sec:approx}

\noindent
In this section we finally give the proof of Theorem~\ref{thm:proprmi}. For the reader's convenience we repeat the statement. In the following, ``ring'' will mean ``commutative Noetherian ring with unit $1$,'' and
$R$ will always denote a ring with $1\neq 0$.

\begin{theorem}\label{Main-Theorem} \label{mt}
Suppose $R$ is a finitely generated integral domain of characteristic zero. Then for all but finitely many $p$ there exists a maximal ideal $\mm$ of $R$ such that $R_{\mm}$ is regular unramified of residue characteristic $p$.
\end{theorem}

\index{ring!regular}
\index{ring!unramified local}

We first recall the concepts mentioned in the conclusion of the theorem.
A sequence $(x_1,\dots,x_n)$ (where $n>0$) of elements of $R$ is said to be {\it regular}\/ if $(x_1,\dots,x_n)\neq R$ and for $i=0,\dots,n-1$, the canonical image of $x_{i+1}$ is not a zero divisor of the ring $R/(x_1,\dots,x_i)$. For example, the sequence $(X_1,\dots,X_n)$ of indeterminates in a polynomial ring $S[X_1,\dots,X_n]$ over a ring $S$ is a regular sequence. In general, regular sequences in $R$ do behave very much like indeterminates in a polynomial ring: if $(x_1,\dots,x_n)$ is a regular sequence in $R$, then the associated graded ring of the ideal $I:=(x_1,\dots,x_n)$,
$$\operatorname{gr}_I(R) = R/I \oplus I/I^2 \oplus \cdots \oplus I^d/I^{d+1}\oplus\cdots,$$
is a polynomial ring over $R/I$; more precisely, the  morphism of graded $R/I$-algebras
$$(R/I)[X_1,\dots,X_n] \to \operatorname{gr}_I(R)\qquad\text{with
$X_i\mapsto x_i+I^2$ for every $i$}$$ is an isomorphism; in particular, $x_i^d\notin I^{d+1}$ for every $i$ and $d$ \cite[Theorem~16.2]{Matsumura}.
If $(R,\mm)$ is a local ring, then $R$ is called regular if its maximal ideal $\mm$ can be generated by a regular sequence; 
in general, the ring $R$ is called regular if each localization $R_{\pp}$ of $R$ at one of its prime ideals $\pp$ is a regular local ring.

If $(R,\mm)$ is local of residue characteristic $p$ {\it unramified} if $p\notin\mm^2$. In general, we need a criterion to determine whether the localization $R_\mm$ of $R$ at one of its maximal ideals $\mm$ is regular unramified.
This is provided by
the following lemma (a special case of Theorem~14.2 in \cite{Matsumura}):

\begin{lemma}\label{Criterion}
Let $(R,\mm)$ be a regular local ring, and let $a\in\mm$, $a\neq 0$. Then the following are equivalent:
\begin{enumerate}
\item $\mm$ is generated by a regular sequence containing $a$;
\item $a\notin\mm^2$;
\item the local ring $R/aR$ is also regular.
\end{enumerate}
\end{lemma}

We equip, as usual, the set $\operatorname{Spec}(R)$ of prime ideals of the ring $R$ with its Zariski topology, with sub-basis given by the open sets of the form
$$D(f) := \big\{ \pp\in\operatorname{Spec}(R): f\notin\pp\big\} \qquad (f\in R).$$
Recall that $D(f)\cap D(g)=D(fg)$ for all $f,g\in R$; in particular any open subset of~$\operatorname{Spec}(R)$ contains an open set of the form $D(f)$, $f\in R$. Also, $D(f)=\emptyset$ if and only if $f$ is nilpotent.

We denote by $\Reg(R)$ the regular locus of $R$, that is, the set of all $\pp\in\Spec(R)$ such that the localization $R_{\pp}$ is regular.
It is well-known  that $\Reg(R)$ is open if $R$ is a finitely generated algebra over a field (Zariski) or over a Dedekind domain (e.g.,~$\Z$) of characteristic zero (Nagata); see (6.12.5) respectively (6.12.6) of \cite{EGA-IV-2}.

Recall that $R$ is said to be \emph{reduced} if $R$ contains no non-zero nilpotent elements; equivalently, if the intersection of all prime ideals of $R$ is trivial. 
Clearly a domain is in particular reduced.
We also need the notion of a {\it Jacobson ring}\/ (also sometimes called a {\it Hilbert ring}\/). This  is a ring each of whose prime ideals is the intersection of maximal ideals. Examples include all fields as well as every principal ideal domain with infinitely many pairwise non-associated primes (like the ring of integers). If $R$ is a Jacobson ring, then so is every finitely generated $R$-algebra $S$, and the pullback of each maximal ideal of $S$ is a maximal ideal of $R$. (See \cite[IV.3.4]{Bourbaki}.)

\begin{lemma}\label{Regular maximal ideal}
Suppose $R$ is a reduced ring whose regular locus is open.  Then $\Reg(R)$ is non-empty, and if $R$ is Jacobson, then $\Reg(R)$ contains a maximal ideal.
\end{lemma}
\begin{proof}
Since $R$ is reduced, for every minimal prime ideal $\pp$  of $R$, the localization $R_{\pp}$ is a field (isomorphic to the fraction field of $R/\pp$), hence $\pp\in\Reg(R)$. In particular, $\Reg(R)$ is a non-empty open subset of $\Spec(R)$, and therefore contains an open subset  $D(f)$, for some non-zero $f\in R$. If $R$ is Jacobson, then the intersection of all maximal ideals of $R$ is zero; hence there exists a maximal ideal $\mm$ of~$R$ not containing $f$, and thus $\mm\in\Reg(R)$.
\end{proof}

Our final ingredient is the following fact, whose proof we postpone:

\begin{proposition}\label{Non-trivial and reduced}
Suppose $R$ is finitely generated. Then $R$ has characteristic zero if and only if for all but finitely many primes $p$, we have
$R/pR\neq 0$. Moreover, if $R$ is reduced then for all but finitely many primes $p$ the ring $R/pR$ is reduced.
\end{proposition}

\begin{remarks}\mbox{}

\begin{enumerate}
\item The converse of the second statement in the proposition is false: if $R=\Z[X]/(4X)$, then $R$ is not reduced, yet $R/pR$ is reduced for every prime $p\neq 2$.
Note also that the analogous statement for the property of being an integral domain is also false, since there are polynomials in $\Z[X]$ (e.g., $X^4+1$) which are irreducible over $\Q$ but reducible modulo every prime.
(In fact, for each composite natural number $n$ there is a monic polynomial in $\Z[X]$ of degree $n$ with this property, cf.~\cite{Brandl}.)

\item After a first version of this manuscript had been completed, Pascal Adjamagbo informed us that in the course of investigations around the Dixmier Conjecture on Weyl algebras, he also established a version of Proposition~\ref{Non-trivial and reduced}.
\end{enumerate}
\end{remarks}

Assuming this proposition, let us outline how Theorem~\ref{Main-Theorem} follows:
Suppose~$R$ is a finitely generated integral domain of characteristic zero.
By Lemma~\ref{Regular maximal ideal} and the remark preceding it we know that $\Reg(R)$ is open and non-empty. In particular there exists a non-zero $f\in R$ such that $D(f)\subseteq \Reg(R)$.
The prime ideals in $D(f)$ are in one-to-one correspondence with the prime ideals of $S:=R[\frac{1}{f}]$, and the corresponding localizations of $R$ and $S$ are isomorphic. Hence $S$ is regular.  Thus, replacing $R$ by $S$, we may assume from the outset that $R$ is regular.

Let now $p$ be such that the ring $\overline{R}:=R/pR$ is non-trivial and reduced, as in the proposition above. Let $\mm$ be a maximal ideal of $R$ containing $p$ whose image $\overline{\mm}$ under the natural morphism $R\to\overline{R}$ is in $\Reg(\overline{R})$. (Such $\mm$ exists by Lemma~\ref{Regular maximal ideal} applied to the finitely generated $\mathbb F_p$-algebra $\overline{R}$.) Note that $R_{\mm}/p R_{\mm}\cong \overline{R}_{\overline{\mm}}$. Therefore  $R_{\mm}$ is unramified regular of residue characteristic $p$, by Lemma~\ref{Criterion}. \qed

\subsection{Proof of Proposition \ref{Non-trivial and reduced}}

In the proof of Proposition~\ref{Non-trivial and reduced}, we use results about dependence on parameters for certain properties of ideals in polynomial rings, formulated in Lemmas~\ref{Member} and \ref{Radical} below.
The elements of the boolean algebra of subsets of $\operatorname{Spec}(R)$ generated by the sets $D(f)$ with $f\in R$, i.e., the smallest collection of subsets of $\operatorname{Spec}(R)$ which contains the sub-basis of $\operatorname{Spec}(R)$ and which is stable under finite union and complementation (and hence, also
under finite intersection) are called {\it constructible.}\/  It is easy to see that each constructible subset of $\operatorname{Spec}(R)$ is a finite union of sets of the form
\begin{equation}\label{Simple constructible set}
\Sigma = D(f)\cap \big(\operatorname{Spec}(R)\setminus D(g_{1})\big) \cap \cdots \cap \big(\operatorname{Spec}(R)\setminus D(g_{m})\big)
\end{equation}
where $f,g_{i}\in R$. See \cite[Chapter~2]{Matsumura-1980} for more on constructible sets, and \cite[Section~12.2]{vdDries} on the connection with ``constructible properties'' exploited below.

\index{constructible}

\medskip

Fix a tuple $X=(X_1,\dots,X_N)$ of pairwise distinct indeterminates, and let $C=(C_1,\dots,C_M)$ be another pairwise distinct tuple of indeterminates (thought of as placeholders for parameters) distinct from $X_1,\dots,X_N$.
Given a field $K$ and $c=(c_1,\dots,c_M)\in K^M$ we let $\pp_c$ denote the kernel of the ring morphism $\Z[C]\to K$ with $C_i\mapsto c_i$ for every $i$.
Note that for $g\in\Z[C]$ we have $\pp_c\in D(g)$ if and only if $g(c)\neq 0$ in $K$. More generally, a constructible subset $\Sigma$ of $\operatorname{Spec}(\Z[C])$ expressed in the form \eqref{Simple constructible set} contains $\pp_c$ if and only if $c$ satisfies the corresponding system of inequalities and equalities
\begin{equation}\label{System}
f(C)\neq 0\ \& \ g_1(C)=0 \ \& \ \cdots \ \& \ g_m(C)=0.
\end{equation}
For $c\in\Z$ and a prime number $p$ we set $c(p):=c\bmod p\in\mathbb F_p$, and we extend this notation to $M$-tuples $c=(c_1,\dots,c_M)\in\Z^M$ in the natural way:
$$c(p)=\big(c_1(p),\dots,c_M(p)\big)\in ({\mathbb F}_p)^M.$$ Then $\pp_c\in D(g)$ if and only if
$\pp_{c(p)}\in D(g)$ for all but finitely many prime numbers $p$, and $\pp_c\notin D(g)$ if and only if
$\pp_{c(p)}\notin D(g)$ for all (equivalently, all but finitely many) primes $p$.
Thus, we obtain the following principle for the reduction of constructible properties from characteristic zero to finite characteristic:

\begin{lemma}\label{Reduction mod p}
Let $\Sigma$ be a constructible subset of $\operatorname{Spec}(\Z[C])$, and let $c\in\Z^M$. Then $\pp_c\in \Sigma$ if and only if for all but finitely many prime numbers $p$ we have $\pp_{c(p)}\in \Sigma$.
\end{lemma}

The following is a theorem by Chevalley, also discovered by Tarski (and known among logicians as ``quantifier elimination for algebraically closed fields'' \cite{Marker}).

\begin{theorem}\label{Chevalley}
Let $\Sigma$ be a constructible subset of $\Spec(\Z[C,D])$, where $D$ is a new indeterminate. Then there exists a constructible subset $\Pi$ of $\Spec(\Z[C])$ such that for every  field $K$ with algebraic closure $K^{\operatorname{alg}}$ and every $c\in K^M$,
$$\text{$\pp_{(c,d)}\in\Sigma$ for every $d\in K^{\operatorname{alg}}$} \qquad\Longleftrightarrow\qquad \pp_c\in\Pi.$$
\end{theorem}

See \cite[Theorem~12.2.10]{vdDries} for a proof. (The proof given there in fact shows how to construct, in an effective way, a description of the constructible set $\Pi$ from a description of $\Sigma$.)

\medskip

Let $f_0(C,X),\dots,f_n(C,X)\in \Z[C,X]$. For a tuple $c\in R^M$ we denote by $I(c,X)$ the ideal of $R[X]$ generated by the polynomials $f_1(c,X),\dots,f_n(c,X)$. The next lemma (applied with $f_0=1$) says in particular that the property of $I(c,X)$ being a proper ideal is a constructible property of the parameters $c$:

\begin{lemma}\label{Member}
There is a constructible subset $\Sigma$  of $\operatorname{Spec}(\Z[C])$ such that for every field $K$ and $c\in K^M$, we have
$$f_0(c,X)\in I(c,X) \qquad\Longleftrightarrow\qquad \pp_c\in\Sigma.$$
\end{lemma}

For the following lemma recall that a field $K$ with $\chr(K)=p>0$ is called \emph{perfect} if $K^p=K$. Note that $\F_p$ is perfect. A field of characteristic zero is always referred to as perfect.
We now have the following lemma.

\begin{lemma}\label{Radical}
There is a constructible subset $\Sigma'$  of $\operatorname{Spec}(\Z[C])$ such that for every perfect field $K$ and $c\in K^M$,
\begin{equation}\label{Radical-Equ}
\text{$I(c,X)$ is radical} \qquad\Longleftrightarrow\qquad \pp_c\in \Sigma'.\end{equation}
\end{lemma}

Let us first see how Proposition~\ref{Non-trivial and reduced} follows from the previous lemmas. Suppose~$R$ is finitely generated. We may assume $R=\Z[X]/I$ for an ideal $I=(f_1,\dots,f_n)$ of $\Z[X]=\Z[X_1,\dots,X_N]$, for some $N$.
Let $c\in\Z^M$ be the tuple of coefficients of the polynomials $f_1,\dots,f_n$ (in some order); so we may write $f_i=f_i(c,X)$ where $f_i(C,X)\in\Z[C,X]$ (with each non-zero coefficient of $f_i(C,X)$ equal to $1$). Now $R$ has characteristic zero if and only if $I\cap\Z=\{0\}$, if and only if $1\notin I\Q[X]=I(c,X)$. Moreover, $R$ is reduced if and only if $I$ is radical; we have
$\sqrt{I}\Q[X] = \sqrt{I\Q[X]}$, hence if $I$ is radical then so is $I\Q[X]$.
Similarly, for a prime $p$, we have
$R/pR\neq 0$ if and only if the ideal $I(c(p),X)$ of $\mathbb F_p[X]$ is proper, and $R/pR$ is reduced if and only if $I(c(p),X)$ is radical.
Proposition~\ref{Non-trivial and reduced} is now a consequence of Lemma~\ref{Reduction mod p} applied to the constructible sets $\Sigma$ and $\Sigma'$ associated to the polynomials $f_0(C,X)=1$ and $f_1(C,X),\dots,f_n(C,X)$ in Lemmas~\ref{Member} and \ref{Radical}, respectively. \qed

\begin{remark}
The first part of Proposition~\ref{Non-trivial and reduced}
(which is related to  Grothendieck's `generic freeness lemma' \cite[Theorem~24.1]{Matsumura})
may also be proved in different ways. For example, by the Noether normalization lemma, there exists a non-zero integer $\delta$ and an injective morphism $\iota\colon\Z[\frac{1}{\delta}][X] \to R[\frac{1}{\delta}]$, for some tuple $X=(X_1,\dots,X_N)$ of indeterminates, such that $R':=R[\frac{1}{\delta}]$ is a finite module over the image $S$ of $\iota$. Suppose $p$ does not divide $\delta$. Then the natural morphism $R\to R'$ induces an isomorphism $R/pR\to R'/pR'$. Moreover, $p$ generates a prime ideal of $S$; let $\pp$ be a prime ideal of $R'$ with $\pp\cap S=pS$. Then the natural inclusion $S\to R'$ yields an inclusion $\mathbb F_p[X]\cong S/pS \to R'/\pp$; in particular, $R'/\pp\neq 0$ and thus $R'/pR'\neq 0$.
Another (algorithmic) proof of this part of Proposition~\ref{Non-trivial and reduced}, employing Gr\"obner basis theory, can be deduced from \cite{Vasconcelos}. It is also an immediate consequence of the non-trivial fact that the unit group of a finitely generated integral domain is finitely generated \cite{Samuel}.
\end{remark}

We now turn to the proof of Lemmas~\ref{Member} and \ref{Radical}. These can probably be deduced from the results in of \cite[\S{}9]{EGA-IV-3}; however, we prefer to follow a slightly less sophisticated approach, based on:

\begin{theorem}\label{vdDries-Schmidt}
There exists an integer $\alpha$ \textup{(}depending on the $f_i(C,X)$\textup{)} with the following property: for every field $K$ and $c\in K^M$, we have $f_0(c,X)\in I(c,X)$ if and only if there are $y_1,\dots,y_n\in K[X]$ of degree at most $\alpha$ with
\begin{equation}\label{ideal membership}
f_0(c,X)=y_1f_1(c,X)+\cdots+y_nf_n(c,X).
\end{equation}
Moreover, there exists an integer $\beta$ such that for every field $K$ and $c\in K^M$, if the implication
$$f\in \sqrt{I(c,X)} \quad\Rightarrow\quad f\in I(c,X)$$
holds for all
$f\in K[X]$ of degree at most $\beta$, then the ideal $I(c,X)$ is radical.
\end{theorem}

For a simple model-theoretic proof of this fact see \cite{vdDriesSchmidt}. It is known, essentially since Hermann's work \cite{Hermann} (see also \cite{Seidenberg}), that $\alpha=(2d)^{2^N}$ suffices, where $d$ is the maximal degree of $f_0,\dots,f_n$ in $X$; in the important case $f_0=1$, by  \cite{Kollar}  we may take  $\alpha=D^N$ where $D:=\max\{ 3, d \}$ for $N>1$ and $D:=2d-1$ for $N=1$.
Analyzing one of the common algorithms to construct the radical of an ideal in a polynomial ring (see, e.g., \cite{GTZ}) also would allow one, in principle, to explicitly compute an integer $\beta$ (depending on $N$ and $d$) with the required property.

\begin{proof}[Proof of Lemma~\ref{Member}]
By Theorem~\ref{vdDries-Schmidt}, $f_0(c,X)\in I(c,X)$ if and only if a certain
(very large) system of linear equations over $K$, obtained (uniformly in~$K$ and $c$) from the equation \eqref{ideal membership} by comparing the coefficients corresponding to equal monomials in $X$ on both sides, has a solution in $K$ (namely, the coefficient tuple of the polynomials $y_1,\dots,y_n$ of degree at most $\alpha$).
That is,
there exists a $k\times l$-matrix $A=A(C)$ with entries in $\Z[C]$ and a column vector $b=b(C)\in\Z[C]^k$ (for some~$k$,~$l$) such that for every field $K$ and $c\in K^M$, the ideal $I(c,X)$ contains $f_0(c,X)$ if and only if the system $A(c)y=b(c)$ of linear equations has a solution in $K$,
that is, if and only if
\begin{equation}\label{Rank}
\operatorname{rank}(A(c))=\operatorname{rank}\big(A(c)|b(c)\big).
\end{equation}
The rank of a matrix over a field is the least $r\geq 0$ such that all of its $(r+1)\times (r+1)$-minors vanish. It is now routine to construct a finite family $(\Sigma_i)_{i\in I}$ of systems as in \eqref{System} such that  for given $K$ and $c$, the equality \eqref{Rank} holds if and only if $c$ satisfies one of the $\Sigma_i$.
\end{proof}

Using the ``Rabinovich trick,'' Lemma~\ref{Member} yields:

\begin{corollary}\label{Member-Cor}
There is a constructible subset $\Sigma''$  of $\operatorname{Spec}(\Z[C])$ such that for every field $K$ and $c\in K^M$, we have
$$f_0(c,X)\in \sqrt{I(c,X)} \qquad\Longleftrightarrow\qquad \pp_c\in\Sigma''.$$
\end{corollary}
\begin{proof}
Let $Y$ be an indeterminate distinct from $X_1,\dots,X_N$. Then for every~$K$ and $c$ we have
$f_0(c,X)\in \sqrt{I(c,X)}$ if and only if the ideal
$$\big(f_1(c,X),\dots,f_n(c,X),1-Yf_0(c,X)\big)$$ of $K[X,Y]$ contains the constant polynomial $1$.
\end{proof}

We now turn to:

\begin{proof}[Proof of Lemma~\ref{Radical}]
Given a field extension $L|K$ of a perfect field $K$, an ideal $J$ of $K[X]$ is radical if and only if $JL[X]$ is radical \cite[Chapitre~V, \S 15, Proposition~5]{BA}. Hence it is enough to show the existence of a constructible subset $\Sigma'$ of $\Spec(\Z[C])$ satisfying \eqref{Radical-Equ} for every $M$-tuple $c$ of elements of an {\it algebraically closed}\/ field $K$. Let $f(D,X)\in\Z[D,X]$ be a polynomial of degree at most $\beta$ with indeterminate coefficients (where $D$ is a tuple of new indeterminates). Using Lemma~\ref{Member} and Corollary~\ref{Member-Cor}, let $\Sigma_1$, $\Sigma_2$ be constructible subsets of $\Spec(\Z[C,D])$ such that for every field $K$ and tuples $c$, $d$ of elements of $K$ of the appropriate lengths,
$$f(d,X) \in \sqrt{I(c,X)} \qquad\Longleftrightarrow\qquad \pp_{(c,d)}\in\Sigma_1$$
and
$$f(d,X) \in I(c,X) \qquad\Longleftrightarrow\qquad \pp_{(c,d)}\in\Sigma_2.$$
Consider the constructible subset
$$\Sigma :=  \big(\Spec(\Z[C,D])\setminus\Sigma_1\big) \cup \Sigma_2$$
of $\Spec(\Z[C,D])$: by Theorem~\ref{vdDries-Schmidt},
the ideal $I(c,X)$ is radical if and only if for every tuple $d$ of elements of $K$ we have
$\pp_{(c,d)}\in\Sigma$. One now obtains $\Sigma'$ from $\Sigma$ by applying Theorem~\ref{Chevalley}.
\end{proof}

\section{Fibered $3$-manifolds}\label{sec:fibered}

\noindent
We show Proposition~1 in a more general form:

\begin{proposition}\label{prop:fibered}
Let $\scrg$ be a graph of finitely generated virtually residually~$p$ groups such that for each edge $e$ of its underlying graph, the morphism $f_e$ is bijective. Then the group $\pi_1(\scrg)$ is virtually residually $p$.
\end{proposition}

\index{$3$-manifold!fibered}
\index{mapping torus}

The \emph{mapping torus} of a group endomorphism $\varphi\colon G\to G$ is the group
$$\Z\ltimes G = \langle G, t: tgt^{-1}=\varphi(g),\ g\in G\rangle.$$
More generally, given group endomorphisms $\varphi_1,\dots,\varphi_r\colon G\to G$, the mapping torus of $\varphi_1,\dots,\varphi_r$ is the group
$$F_r\ltimes G = \langle G, t: tgt^{-1}=\varphi_i(g),\ i=1,\dots,r,\ g\in G\rangle.$$
Here and below, $F_r$ denotes the free group on $r$ generators.
The proposition above immediately yields:

\begin{corollary}\label{cor:fibered}
The mapping torus of finitely many automorphisms of a finitely generated virtually residually $p$ group is virtually residually $p$.
\end{corollary}

This corollary readily implies Proposition~1: Suppose $N$ is a fibered $3$-manifold, so $\pi_1(N)=\Z \ltimes G$ where $G=\pi_1(F)$ is the fundamental group of a surface $F$. Thus~$F$ is either the free group $F_r$ on $r$ generators (for some $r$) or a surface group (that is, the fundamental group of a closed surface).
Free groups are residually $p$ for all~$p$, and orientable surface groups are residually free \cite{Ba62} and hence residually~$p$ for all~$p$.
Every non-orientable surface group of genus $g$ contains a characteristic subgroup of index $2$ which is isomorphic to the orientable surface group of genus~$g-1$.
In any case,  the group $G$ is virtually residually $p$, for every $p$, and hence so is $\pi_1(N)$ by Corollary~\ref{cor:fibered}.

\index{surface group}

\begin{remark}
Corollary~\ref{cor:fibered} can be seen as a more general and precise version of the result \cite{Ba71} that the mapping torus of an automorphism of a free group of finite rank is residually finite.
(In fact, the mapping torus of an {\it endo}\/morphism of a linear group is always residually finite, as shown by Borisov and Sapir \cite{BS05}; they also showed that the mapping torus of an endomorphism of a finite-rank free group is virtually residually $p$ for all but finitely many $p$ \cite{BS09}.) As remarked in~\cite{DS05}, the linearity of the braid group $B_4$ \cite{Krammer} implies that
the mapping torus of an automorphism of $F_2$ is a linear group; in fact, the holomorph of $F_2$ is linear \cite{CMP}.
It is unknown whether a mapping torus of an automorphism of $F_r$ is always linear for $r>2$, cf.~\cite[Problem~5]{DS05}.
\end{remark}

For the proof of Proposition~\ref{prop:fibered} we first note that it suffices to treat the special case of mapping tori: given a graph of groups $\scrg$ such that for each edge $e$ of the underlying graph $Y$ of $\scrg$, the morphism $f_e$ is bijective, for every maximal subtree~$T$ of $Y$ and every $v\in V(Y)$ we have $\pi_1(\scrg|T)\cong G_v$  and thus
$\pi_1(\scrg) \cong F_r \ltimes G_v$, where  $r:=\frac{1}{2}|E(Y)\setminus E(T)|$.
So suppose now that $$G^*=\langle G, t: tgt^{-1}=\varphi_i(g),\ i=1,\dots,r,\ g\in G\rangle$$
is the mapping torus of the automorphisms $\varphi_1,\dots,\varphi_r$ of the finitely generated group $G$. Every finite index subgroup of the finitely generated group $G$ contains a finite index subgroup of $G$ which is fully characteristic (even verbal) in $G$. Thus, if $G$ is virtually residually $p$, then $G$ contains a finite index residually $p$ subgroup which is fully characteristic. Using Proposition~\ref{prop:commoncover} we thus reduce to proving that
``$G$ residually $p$'' implies
 ``$G^*$ virtually residually $p$.'' The next lemma shows that our best chance for showing that $G^*$ is residually $p$ is by showing that the $\varphi_i$ act unipotently on $H_1(G;\F_p)$:

\begin{lemma}\label{lem:unipotent on homology}
Let
$$H^*=\langle H, t: tht^{-1}=L^p_1(\varphi_i)(h),\ i=1,\dots,r,\ h\in H\rangle$$
be the mapping torus of the automorphisms $L_1^p(\varphi_1),\dots,L_1^p(\varphi_r)$ of $H:=L^p_1(G)=H_1(G;\F_p)$ induced by $\varphi_1,\dots,\varphi_r$.  Then
$H^*$ is residually $p$ if and only if the automorphisms $L_1^p(\varphi_1),\dots,L_1^p(\varphi_r)$ generate a $p$-subgroup of $\Aut(H)$. If $G$ and $H^*$ are residually $p$, then so is $G^*$.
\end{lemma}

\begin{proof}
The first statement follows from the equivalence of (1) and (2) in Corollary~\ref{cor:chatzidakis}. Suppose $G$ and $H^*$ are residually $p$. To show that $G^*$ is residually~$p$, it is enough to show, by Proposition~\ref{Hempel criterion residually p} and Lemma~\ref{lem:hnn}, that for each $n\geq 1$, the mapping torus of the automorphisms of $L^p_n(G)$ induced by the $\varphi_i$ is residually $p$. This follows from the first statement and  Lemma~\ref{lem:induced automorphisms} below.
\end{proof}

For every $n\geq 1$, by restriction we obtain a natural group morphism
$\Aut(G) \to \Aut(L^p_n(G))$. If  $\sigma\in\Aut(G)$ induces the identity on $L^p_1(G)=H_1(G;\F_p)$, then $\sigma$ induces the identity on $L^p_n(G)$ for each $n\geq 1$
\cite[Chapter~VIII, Theorem~1.7]{HB82}. Therefore:

\begin{lemma}\label{lem:induced automorphisms}
If the image of a subgroup $A$ of $\Aut(G)$ in $\Aut(L^p_1(G))$ is a $p$-group, then the image of $A$ in $\Aut(L^p_n(G))$ is a $p$-group for every $n\geq 1$.
\end{lemma}

Suppose now that $G$ is residually $p$. Then by Corollary~\ref{cor:unfolded, 1}  we see that the mapping torus $G^*$ has a finite index normal subgroup which is itself (isomorphic to) the mapping torus of finitely many automorphisms of $G$, and such that the mapping torus of the corresponding induced automorphisms of $L_1^p(G)$ is residually~$p$. Hence by Lemma~\ref{lem:unipotent on homology},  $G^*$ is virtually residually $p$. This finishes the proof of Proposition~\ref{prop:fibered}. \qed

\medskip

It may be instructive to contrast Corollary~\ref{cor:fibered} with the following proposition, which shows that fundamental groups of fibered $3$-manifolds always satisfy, without passing to a finite index subgroup, a very weak variant of being residually $p$. (This proposition also fills in the details for a claim made in the introduction.)

\begin{proposition}\label{prop:fibered implies weakly residually p}
Let $\Gamma$ be the fundamental group of a fibered $3$-manifold. Then
for every $\gamma\in\Gamma$ with $\gamma\neq 1$ there is a subgroup $\widetilde{\Gamma}$ of $\Gamma$ of $p$-power index such that $\gamma\notin \widetilde{\Gamma}$.
\end{proposition}

\begin{proof}
As above we write $\Gamma=\Z \ltimes G$ where $G$ is a free group of finite rank or a surface group. Then $G$ is residually $p$.
Now, for any non-zero integer $k$ we can consider  the subgroup of $\Gamma=\Z\ltimes G$ generated by $t^k$ and $G$, which we denote by
$k\Z \ltimes G$; this subgroup is normal in $\Gamma$ of index $|k|$. For any $n\geq 1$ we may also consider the subgroup of $\G$ generated by $t$ and $\gamma^p_n(G)$, denoted by $\Z\ltimes\gamma^p_n(G)$; note that $\Z\ltimes\gamma^p_n(G)$ is of $p$-power index in $\G$, and this subgroup is not normal in general.
Now suppose $\gamma=t^kg\in\Gamma$ where $k\in\Z$, $g\in G$. If $k\neq 0$, then there exists an~$n$ such that $\gamma\notin p^n\Z\ltimes G$. If $k=0$ and $g\neq 1$, then there exists an $n\geq 1$ such that~$\gamma\notin \Z\ltimes \gamma^p_n(G)$.
\end{proof}

\chapter{The Case of Graph Manifolds}\label{ch:graph manifolds}

\noindent
It is not surprising that in the special case of graph manifolds, our methods yield more refined results. We first restate and prove Proposition~2 from the introduction:

\begin{proposition} \label{prop:gm}
Let $N$ be a graph manifold. Then
for every $p$, the group $G=\pi_1(N)$ is virtually residually $p$.
\end{proposition}
\begin{proof}
If $N$ is Seifert fibered, then $\pi_1(N)$ has a finite index subgroup which is residually $p$ for every $p$, by Proposition~\ref{prop:Seifert res p}. So we may assume from now on that~$N$ is not Seifert fibered. The rest of the argument is as in Sections~\ref{sec:reduction to seifert simple}--\ref{sec:conclusion} in the proof of the main theorem:
After passing to a finite cover of $N$ if necessary, we may assume that $N$ is Seifert simple. 
Then, by Proposition~\ref{prop:seifert simple lower p-central}, for all  $p$, each JSJ component of $N$ admits a boundary compatible $p$-filtration of level~$1$, and so $G$ is virtually residually $p$ by Theorem~\ref{thm:virt res p for graphs of linear groups, 1}.
\end{proof}

\index{graph manifold}

In the rest of this chapter we prove Proposition~3  from the introduction. We also discuss $p$-cohomological completeness and obstacles towards a proof of virtual $p$-efficiency for arbitrary $3$-manifolds, and establish Proposition~4.

\section{$p$-efficiency}

\index{$p$-efficient}

\noindent
We first turn to Proposition~3, restated here:

\begin{proposition}\label{prop:prop 2 restated}
Let $N$ be a closed graph manifold. Then for every $p$ there is a finite cover of $N$ which is $p$-efficient.
\end{proposition}

\begin{remark}
We do not know how to arrange for the finite cover of $N$ in this proposition  to be regular.
\end{remark}

\subsection{$p$-efficiency of graphs of groups}
We say that a graph of groups $\scrg$ is \emph{$p$-efficient} if
\begin{enumerate}
\item $G=\pi_1(\scrg)$ is residually $p$;
\item for all vertices $v$ and edges $e$, the subgroups $G_v$ and $G_e$ of $G$ are closed in the pro-$p$ topology on $G$;
\item for all vertices $v$ and edges $e$, the pro-$p$ topology on $G$ induces the pro-$p$ topology on $G_v$ and on $G_e$.
\end{enumerate}
So a closed orientable prime $3$-manifold $N$ is $p$-efficient precisely if its associated graph of groups is.
For the proof of Proposition~\ref{prop:prop 2 restated}, we first need to establish a criterion for $p$-efficiency of a graph of groups, which is Corollary~\ref{cor:p-efficiency criterion} of the following sequence of lemmas.

Given a filtration $\mathbf G$ of a group $G$ by $p$-power index normal subgroups, we say that $\mathbf G$ induces the pro-$p$ topology on $G$ if the group topology on $G$ with fundamental system of neighborhoods of the identity given by $\mathbf G$ agrees with the pro-$p$ topology on $G$; equivalently, for every $p$-power index normal subgroup $H$ of~$G$ there is some $m\geq 1$ with $H\geq G_m$. (If $G_{\operatorname{ab}}$ is finitely generated, it suffices to require this for $H$ of the form $H=\gamma^p_n(G)$, $n\geq 1$.)

\index{graph of groups!$p$-efficient}

In the following lemmas and their corollary below we fix a filtration $\mathbf G$ of a graph $\scrg$ of finitely generated groups.

\begin{lemma}\label{lem:p-efficient}
Suppose $\mathbf G$ is a complete $p$-filtration, and
\begin{enumerate}
\item for each vertex $v$, the filtration $\mathbf G_v$ induces the pro-$p$ topology on $G_v$;
\item for each edge $e$, the filtration $\mathbf G_e$ induces the pro-$p$ topology on $G_e$; and
\item for every $n\geq 1$, the group $\pi_1(\scrg/\mathbf G_n)$ is residually $p$.
\end{enumerate}
Then the pro-$p$ topology on $G=\pi_1(\scrg)$ induces the pro-$p$ topology on each vertex and edge group of $\scrg$.
If in addition $\mathbf G$ is separating and separates the edge groups of $\scrg$, then $\scrg$ is $p$-efficient.
\end{lemma}
\begin{proof}
By condition (2), the vertex groups induce the pro-$p$ topology on the edge groups, so it is enough to show that $G$ induces the pro-$p$ topology on each vertex group.
Fix a vertex $v_0$ and a normal subgroup $H_{v_0}$ of $G_{v_0}$ of $p$-power index; we need to find a normal subgroup $H$ of $G$ of $p$-power index with $H\cap G_{v_0}\leq H_{v_0}$. By hypothesis (1) there is some $n\geq 1$ such that $H_{v_0}\geq G_{v_0,n}$. Denote the natural morphism $G=\pi_1(\scrg)\to\pi_1(\scrg/\mathbf G_n)$ by $\pi$. By (3) and since $\pi(G_{v_0})$ is finite, there exists a normal subgroup $K$  of $\pi_1(\scrg/\mathbf G_n)$ of $p$-power index with $\pi(G_{v_0})\cap K=1$. Then $H:=\pi^{-1}(K)$ is a $p$-power index normal subgroup of $G$ with $H\cap G_{v_0}=G_{v_0,n}\leq H_{v_0}$ as desired.
The last statement now follows from the first in combination with Corollaries~\ref{cor:Hempel, cor} and \ref{cor:vertex and edge groups closed}.
\end{proof}

For the proof of the next lemma we revisit the proof of Theorem~\ref{thm:reduction theorem, 3}:

\begin{lemma}\label{lem:p-efficiency and p-excellence}
Suppose that $\mathbf G$ is $p$-excellent, and
\begin{enumerate}
\item for each vertex $v$, the filtration $\mathbf G_v$ induces the pro-$p$ topology on $G_{v,1}$;
\item for each edge $e$, the filtration $\mathbf G_e$ induces the pro-$p$ topology on $G_{e,1}$.
\end{enumerate}
Then there exists a finite degree morphism $\widetilde{\scrg}\to\scrg$ of graphs of groups such that
$\widetilde{\scrg}$ is $p$-efficient.
\end{lemma}
\begin{proof}
We argue as in the proof of Theorem~\ref{thm:reduction theorem, 3}.
First we apply Proposition~\ref{prop:commoncover} to $\{H_v\}=\{G_{v,1}\}$ in order to reduce to the case that $\mathbf G$ is complete. From now on we assume this to be the case.
Next, we apply Corollary~\ref{cor:unfolded, 1} with $\mathcal H:=\mathbf G_2$ and let $\phi\colon\widetilde{\scrg}\to\scrg$ have properties (1)--(4) in that corollary. Let $\widetilde{\mathbf G}:=\phi^{-1}(\mathbf G)$. Then by Theorem~\ref{thm:reduction theorem}, the group~$\widetilde{G}=\pi_1(\widetilde{\scrg})$ is residually $p$, the vertex and edge groups of $\widetilde{\scrg}$ are closed in the pro-$p$ topology on $\widetilde{G}$, and for every $n\geq 1$, the group~$\pi_1(\widetilde{\scrg}/\widetilde{\mathbf G}_n)$ is residually $p$.
Moreover, by condition (1) in Corollary~\ref{cor:unfolded, 1}, for each vertex $\widetilde{v}$ of the graph underlying $\widetilde{\scrg}$, the filtration $\widetilde{\mathbf G}_{\widetilde{v}}$ induces the pro-$p$ topology on $\widetilde{G}_{\widetilde{v}}$, and
for each edge $\widetilde{e}$,  $\widetilde{\mathbf G}_{\widetilde{e}}$ induces the pro-$p$ topology on $\widetilde{G}_{\widetilde{e}}$.
Thus by Lemma~\ref{lem:p-efficient}, $\widetilde{\scrg}$ is $p$-efficient.
\end{proof}

\index{filtration!$p$-excellent}

The following immediate corollary of Lemma~\ref{lem:p-efficiency and p-excellence}
is a variant of Theorem~\ref{thm:virt res p for graphs of linear groups, 1}:

\begin{corollary}\label{cor:p-efficiency criterion}
Suppose that all edge groups $G_e$ are abelian $p$-torsion free.
Assume moreover that $\mathbf G$ is $p$-compatible of some level $\ell\geq 0$ and for each vertex $v$, the filtration $\mathbf G_v$ induces the pro-$p$ topology on the subgroup $G_{v,1}$ of $G_v$. Then there is a finite degree morphism $\widetilde{\scrg}\to\scrg$ of graphs of groups such that $\widetilde{\scrg}$ is $p$-efficient.
\end{corollary}

We can now prove Proposition~\ref{prop:prop 2 restated}. So let $N$ be a closed prime graph manifold. If $N$ Seifert fibered, then there is nothing to do except to apply Proposition~\ref{prop:Seifert res p}. Suppose $N$ is not Seifert fibered. Then, after passing to a finite cover of $N$ if necessary, we may assume that $N$ is Seifert simple. In the proof of Proposition~\ref{prop:seifert simple lower p-central} we showed that for every JSJ component $N_v$ of $N$, the filtration $\mathbf G_v=\{\gamma^p_{n+1}(G_v)\}_{n\geq 1}$ of $G_v=\pi_1(N_v)$ is a boundary compatible $p$-filtration of level $1$ for $N_v$. Clearly $\mathbf G_v$ induces the pro-$p$ topology on $\gamma^p_2(G_v)$.
Hence by the previous corollary, $N$ admits a finite $p$-efficient cover. \qed

\subsection{$p$-efficiency and the pro-$p$ fundamental group}
The usefulness of the notion of $p$-efficiency is to be found in connection with a description of the pro-$p$ completion of the fundamental group of a $p$-efficient graph of groups, which we give now. For simplicity we restrict our attention to graphs of finitely generated groups.

\begin{notation}
For a (discrete) group $G$ we denote by $\widehat{G}_p$ the pro-$p$ completion of $G$, that is, the
inverse limit of the system of $p$-group quotients of $G$. The group $\widehat{G}_p$ comes equipped with a topology making it a compact totally disconnected topological group, and we have a natural morphism $G\to\widehat{G}_p$ which is injective precisely if $G$ is  residually~$p$.
\end{notation}

\index{$\widehat{G}_p$}
\index{pro-$p$ completion|textbf}
\index{topology!pro-$p$}

Let first $\scrg$ be a graph of finitely generated pro-$p$ groups, i.e., a graph of groups all of whose vertex and edge groups are (topologically) finitely generated
pro-$p$ groups, and all of whose edge morphisms are continuous. In this situation, the \mbox{\it pro-$p$~fundamental~group $\widehat{\pi}_1(\scrg)$ of $\scrg$} is defined to be the pro-$p$ completion of $G=\pi_1(\scrg)$: $\widehat{\pi}_1(\scrg)=\widehat{G}_p$. We have a natural continuous group morphism $G\to\widehat{\pi}_1(\scrg)$, were $G$ is equipped with the pro-$p$ topology. Unlike the case of the ordinary fundamental group, the natural morphisms $G_v\to  \widehat{\pi}_1(\scrg)$ are usually not injective (cf.~\cite[Example~9.2.9]{RZ00}). One says that $\scrg$ is \emph{proper} if all morphisms $G_v\to  \widehat{\pi}_1(\scrg)$ are injective.

\index{graph of pro-$p$ groups!pro-$p$ fundamental group}
\index{graph of pro-$p$ groups!proper}

Now let $\scrg$ be an arbitrary graph of (abstractly) finitely generated groups. Suppose the pro-$p$ topology on each vertex group $G_v$ induces the pro-$p$ topology on~$f_e(G_e)$ for each edge $e$ with $t(e)=v$. (This will be satisfied if $\scrg$ is $p$-efficient.) Then the induced continuous group morphisms $\widehat{f_e}\colon\widehat{(G_e)}_p\to \widehat{(G_v)}_p$ between the pro-$p$ completions are injective, hence the graph of groups $\widehat{\scrg}$ with same underlying graph as $\scrg$, vertex groups $\widehat{(G_v)}_p$, edge groups $\widehat{(G_e)}_p$, and edge morphisms $\widehat{f_e}$, is a graph of finitely generated pro-$p$ groups as above. We have a natural morphism $\scrg\to\widehat{\scrg}$ of graphs of groups and hence a group morphism $G=\pi_1(\scrg)\to\pi_1(\widehat{\scrg})$.

\begin{lemma}\label{lem:p-efficient implies proper}
Suppose $\scrg$ is a $p$-efficient graph of finitely generated groups. Then applying the pro-$p$ completion functor to the morphism $G\to  \pi_1(\widehat{\scrg})$ yields a continuous isomorphism
$\widehat{G}_p \to \widehat{\pi}_1(\widehat{\scrg})$, and $\widehat{\scrg}$ is proper.
\end{lemma}
\begin{proof}
This is well-known in the profinite case; we only sketch the proof.
First note that $p$-efficiency implies that the natural morphism $G\to \pi_1(\widehat{\scrg})$ is injective.
Using the commutative diagram \eqref{eq:factorization} one sees that $\pi_1(\widehat{\scrg})$ induces the pro-$p$ topology on $G=\pi_1(\scrg)$, and $G$ is dense in $\pi_1(\widehat{\scrg})$. Hence the continuous morphism $\widehat{G}_p \to \widehat{\pi}_1(\widehat{\scrg})$ remains injective and has dense image in $\widehat{\pi}_1(\widehat{\scrg})$. Since $\widehat{G}_p$ is compact and $\widehat{\pi}_1(\widehat{\scrg})$ is Hausdorff, this image in fact equals $\widehat{\pi}_1(\widehat{\scrg})$. Hence $\widehat{G}_p \to \widehat{\pi}_1(\widehat{\scrg})$ is an isomorphism.
Since $G$ induces the pro-$p$ topology on each $G_v$, the natural injective morphisms $G_v\to G$ give rise to injective morphisms $\widehat{(G_v)}_p\to\widehat{G}_p=\widehat{\pi}_1(\widehat{\scrg})$. Thus $\widehat{\scrg}$ is proper.
\end{proof}

\section{Cohomological $p$-completeness}\label{sec:coh p-completeness}

\noindent
Following \cite{LS07} we call a group $G$ \emph{cohomologically $p$-complete} if 
the natural morphism $G\to \widehat{G}_p$ induces an isomorphism
$H^n_{\operatorname{cont}}(\widehat{G}_p;\F_p)\to H^n(G;\F_p)$ for each~$n\geq 1$.
Here, $H^\ast(G;\F_p)$ is the usual (discrete) cohomology of $G$, and $H^\ast_{\operatorname{cont}}(\widehat{G}_p;\F_p)$ is the continuous  cohomology (sometimes called Galois cohomology \cite{Se97}) of the pro-$p$ group $\widehat{G}_p$. In both cases the action on $\F_p$ is assumed to be trivial. Cohomological $p$-completeness can be characterized in different ways:

\begin{lemma}
Let $G$ be a group.
The following are equivalent:
\begin{enumerate}
\item $G$ is cohomologically $p$-complete;
\item for every finite $p$-primary $G$-module $M$ and
$n\geq 1$, the natural morphism $G\to \widehat{G}_p$ induces an isomorphism
$H^n_{\operatorname{cont}}(\widehat{G}_p;M)\to H^n(G;M)$;
\item for every finite $p$-primary $G$-module $M$ and
$n\geq 1$,
$$\lim_{N\leq_p G} H^n(N;M)=0;$$
\item for all $n\geq 1$,
$$\lim_{N\leq_p G} H^n(N;\F_p)=0.$$
\end{enumerate}
Here $N$ ranges over all subgroups of $G$ of $p$-power index.
\end{lemma}
\begin{proof}
For the implication (1)~$\Rightarrow$~(2) see \cite[Lemma~4.7]{LS07}.
The equivalences (1)~$\Longleftrightarrow$~(4) and (2)~$\Longleftrightarrow$~(3) follow from \cite[Exercise~1, Section~I.2.6]{Se97}. (This exercise is stated in terms of profinite completion, but extends easily to the pro-$p$ case.)
\end{proof}

\index{cohomologically $p$-complete}
\index{cohomologically complete}

See \cite{FAKRS08, KZ08, Lo10, LS07, Wei07} for more on this notion.
A group which is cohomologically $p$-complete for every $p$ is said to be \emph{cohomologically complete.} (This class of groups is of interest in connection with Atiyah's conjecture on the integrality of $L^2$-Betti numbers of compact manifolds with torsion free fundamental group \cite{LS07}.) Free groups are cohomologically complete, as are right-angled Artin groups \cite{Lo10} and primitive
link groups  \cite{BLS08}.
(The argument in \cite[Theorem 1.2.1]{HMM05} claiming that all link groups are cohomologically complete seems to have a gap, cf.~\cite{Bl07}.)
By induction on the nilpotency class, the following fact implies that finitely generated nilpotent groups are cohomologically complete:

\begin{proposition}\label{prop:cc and central extensions}
Let
$$1\to A\to G\to Q\to 1$$
be a central group extension with $A$ finitely generated. If $A$ and $Q$ are cohomologically $p$-complete, then so is $G$.
\end{proposition}

For a proof of this fact see \cite[Corollary~1.2]{Lo10} or \cite[Section~4]{LS07}. We use it to show:

\begin{proposition}\label{prop:seifert virtually cc}
Let $N$ be a Seifert fibered space. Then $\pi_1(N)$ is virtually cohomologically complete. If $N=S^1\times F$ for some orientable surface $F$, then $\pi_1(N)$ is  cohomologically complete.
\end{proposition}
\begin{proof}
As in the proof of Proposition~\ref{prop:Seifert res p}, after passing to a finite index subgroup if necessary, we may assume that $N$ is an $S^1$-bundle over an orientable surface $F$. If $N=S^1\times F$ for some orientable surface $F$, then $N$ is already of this form.
In any case, $\pi_1(N)$ fits into an exact sequence
$$1\to \Z\to \pi_1(N)\to \pi_1(F)\to 1$$
with (the image of) $\Z$ central in $\pi_1(N)$. (Cf.~proof of Lemma~\ref{lem:S1-bundles res p}.) Now $\pi_1(F)$ is either free or an orientable surface group; either way, $\pi_1(F)$ is a primitive one-relator group and hence cohomologically complete by \cite[Exemple (2), p. 144]{Lab67}. Thus $\pi_1(N)$ is also cohomologically complete, by Proposition~\ref{prop:cc and central extensions}.
\end{proof}

\index{$3$-manifold!Seifert fibered}

On the other hand, it does not seem to be known which fundamental groups of hyperbolic $3$-manifolds are cohomologically $p$-complete. In \cite{Re97} (with generalizations and different proofs in \cite{KZ08, Wei07}) it is shown that if $G$ is the fundamental group of a
closed hyperbolic $3$-manifold which is not virtually Haken
(i.e., a hypothetical counterexample to Thurston's virtually Haken conjecture), whose pro-$p$ completion is infinite, then $G$ is cohomologically $p$-complete.
We have:

\begin{proposition}\label{prop:virt cohomological p-completeness for graph manifolds}
Let $N$ be a closed graph manifold. Then for every $p$ the group $\pi_1(N)$ is virtually cohomologically $p$-complete.
\end{proposition}

This complements a result by Wilton and Zalesskii \cite{WZ10} who showed that every graph manifold group $G$ is \emph{good} in the sense of Serre \cite[Section~2.6, exercises]{Se97}, that is, with $\widehat{G}$ denoting the profinite completion of $G$, for every finite $G$-module $M$ and $n\geq 1$, the natural morphism $G\to\widehat{G}$ induces an isomorphism $H^n_{\operatorname{cont}}(\widehat{G},M)\to H^n(G,M)$. For the proof of Proposition~\ref{prop:virt cohomological p-completeness for graph manifolds}
we use the following well-known fact, a consequence of Lemma~\ref{lem:p-efficient implies proper} (cf.~\cite[proof of Proposition~4.3]{WZ10}):

\index{group!good}

\begin{lemma}
Suppose $\scrg$ is a $p$-efficient graph of finitely generated groups with underlying graph $Y$ and fundamental group $G=\pi_1(\scrg)$. Let $E_+$ be an orientation of $Y$ and $V=V(Y)$.
Then for every $p$-primary $G$-module there is a commutative diagram
$$\entrymodifiers={+!!<0pt,\fontdimen22\textfont2>}
\xymatrixcolsep{1pc}
\xymatrix{
\cdots\ar[r] & H^{n}(G;M) \ar[r] & \bigoplus\limits_{v\in V} H^n(G_v;M) \ar[r] & \bigoplus\limits_{e\in E_+} H^n(G_e;M) \ar[r] & H^{n+1}(G;M) \ar[r]&\cdots \\
\cdots\ar[r] & H^{n}(\widehat{G};M) \ar[r]\ar[u] & \bigoplus\limits_{v\in V} H^n(\widehat{G_v};M) \ar[r]\ar[u] & \bigoplus\limits_{e\in E_+} H^n(\widehat{G_e};M) \ar[r]\ar[u] & H^{n+1}(\widehat{G};M)\ar[u]
\ar[r]&\cdots}$$
with exact rows. Here the hats indicate pro-$p$ completion,
the groups in the bottom row are continuous cohomology groups, and
the vertical maps are induced by the natural embeddings of the groups into their pro-$p$ completions.
\end{lemma}

(The top row is simply the Mayer-Vietoris sequence associated to the graph of groups $\scrg$, cf.~Section~\ref{sec:MV}.)

\begin{corollary}
Suppose $\scrg$ is a $p$-efficient graph of finitely generated groups. Suppose all vertex and edge groups of $\scrg$ are cohomologically $p$-complete. Then $G=\pi_1(\scrg)$ is cohomologically $p$-complete.
\end{corollary}
\begin{proof}
Since for every finite $p$-primary $G$-module $M$ we have
$H^0(G;M)=M^G=M^{\widehat{G}_p}=H^0(\widehat{G}_p;M)$, this follows from the commutative diagram in the lemma above.
\end{proof}

Proposition~\ref{prop:virt cohomological p-completeness for graph manifolds} is now immediate from Propositions~\ref{prop:prop 2 restated} and \ref{prop:seifert virtually cc}, and the previous corollary.

\section{Virtual $p$-efficiency for arbitrary $3$-manifolds?}\label{sec:p-efficiency in general}

\noindent
The results for graph manifolds in the previous sections naturally raise the  question whether \emph{every} closed prime $3$-manifold is, for all but finitely many $p$, virtually $p$-efficient.
(If the answer was ``yes,'' then every closed prime $3$-manifold also would be, for all but finitely many $p$, virtually cohomologically $p$-complete.)
Now, by Corollary~\ref{cor:p-efficiency criterion}, a positive answer to this question is implied by a positive answer to the following:
\begin{quote}
{\it Let $N$ be an orientable hyperbolic $3$-manifold with non-empty to\-roi\-dal incompressible boundary. Does $N$, for all but finitely many~$p$, admit a boundary compatible $p$-filtration $\mathbf G=\{G_n\}$ \textup{(}of some level\textup{)} which induces the pro-$p$ topology on $G_{1}$?}
\end{quote}
We do not know the answer to this question, and in the remainder of this section we indicate some serious obstacles in answering it, which we were note able to overcome.

\index{$p$-efficient}
\index{$3$-manifold!hyperbolic}

\medskip

To investigate this issue more closely, let us return to the setting of the proof of Theorem~\ref{thm:virt res p for graphs of linear groups, 2}, in the case where $n=2$ and $G=\pi_1(N)$ is the fundamental group  of a $3$-manifold $N$ as in the question above, identified with a subgroup of $\SL(2,\C)$. We employ the modifications of this proof outlined in Section~\ref{sec:Kleinians}, in order to take advantage of the special properties of $G$:
After conjugating $G$ suitably in $\SL(2,\C)$ we may and shall assume that $G$ is defined over
a ring $R$ of $S$-integers of a number field.
Suppose $\mm$ is a maximal ideal of $R$ such that $R_\mm$ is regular unramified with finite residue field of characteristic $p$. For all but finitely many $p$ there will be such $\mm$. Then $\Lambda:=\widehat{R_{\mm}}$ is a finite ring extension of $\Z_p$. Let $P$ be the finite set of exceptional primes as defined in the proof of Theorem~\ref{thm:virt res p for graphs of linear groups, 2}. In particular,
$H=H_1(G;\Z)/\operatorname{tor}$ is $p$-torsion free for $p\notin P$.
We note that Long and Reid \cite{LR98} showed, as a consequence of Chebotarev's Density Theorem and the availability of a complete list of all subgroups of $\operatorname{SL}(2,\F_p)$ (cf.~\cite[Section~3.6]{Su82}):

\begin{proposition}\label{prop:Long-Reid}
The set $\mathcal P$ of primes $p\notin P$ such that $|R/\mm|=p$ and the morphism $$G\to\SL(2,R)\to\SL(2,R/\mm)$$ is onto is infinite \textup{(}in fact, has natural density $1$\textup{)}.
\end{proposition}

In the proof of Theorem~\ref{thm:virt res p for graphs of linear groups, 2}, given $p\notin P$, we constructed a boundary compatible $p$-filtration $\mathbf G=\{G_k\}_{k\geq 1}$ of level $1$ for $N$ by setting
$$  G_k = \ker\left( G\to \SL(2,\Lambda)\times H\to \SL(2,\Lambda/\mm^k\Lambda)\times H/p^k H\right).$$
In particular we have
\begin{align*}
G_1 &= \ker\left( G\to \SL(2,\Lambda)\times H\to \SL(2,\Lambda/\mm\Lambda)\times H/p H\right) \\
&= \SL^1(2,\Lambda)\cap \gamma^p_2(G).
\end{align*}
Now the following three possibilities for answering the question posed above present themselves:
\begin{enumerate}
\item {\it Show that $\mathbf G$ induces the pro-$p$ topology on $G_1$.}\/ (If so, we would be done.)
\item {\it Show that $G$ induces the pro-$p$ topology on $G_1$.}\/ (If so, by Remark~(2) following the proof of Theorem~\ref{thm:virt res p for graphs of linear groups, 2}, the filtration $\mathbf G^*=\{G^*_k\}$ of $G$ defined by
$G_k^*=\SL^k(2,\Lambda)\cap\gamma^p_{k+1}(G)$ would be a boundary compatible $p$-filtration of level $1$ for $N$ which clearly induces the pro-$p$ topology on $G_1=G_1^*$.)
\item {\it Show that $H_1(G_1;\Z)$ is $p$-torsion free.}\/ (If so, $G_1$ would be $p$-potent, hence by Lemma~\ref{lem:canonical filtration, 2}, $\gamma^p(G_1)$ would be a boundary compatible $p$-filtration of level~$1$ for $N$.)
\end{enumerate}
The first approach receives a resounding defeat; this is due to the following fact:

\begin{theorem}\label{thm:Lubotzky}
Let $M$ be an orientable hyperbolic $3$-manifold with non-empty to\-roi\-dal boundary and $\dim H_1(M;\F_p)>2$.
Then the pro-$p$ completion of $\pi_1(M)$ is not $p$-adic analytic.
\end{theorem}
\begin{proof}
This follows from an important result of Lubotzky \cite[Theorem~2.3~(b)]{Lu83}: if $\Gamma$ is a finitely presentable group whose $p$-adic completion is \mbox{$p$-adic} analytic and $\dim H_1(\Gamma;\F_p)\neq 1$, then its deficiency $\operatorname{def}(\Gamma)$ (number of generators $-$ number of relations) satisfies the inequality $\operatorname{def}(\Gamma)\leq d(1-d/4)$ where $d=\dim H_1(\Gamma;\F_p)$. On the other hand,
it is well-known (cf.~\cite{Ra87}) that 
$\pi_1(M)$ has deficiency $1$.  Thus $\pi_1(M)$ is not $p$-adic analytic.
\end{proof}

\index{Lubotzky}
\index{topology!pro-$p$}
\index{group!$p$-adic analytic}

The hypothesis on $\dim H_1(M;\F_p)$ is satisfied if $M$ has more than $2$ boundary components, by virtue of
the following well-known application of Poincar\'e-Lefschetz duality arguments:

\begin{lemma}
Let $M$ be a $3$-manifold with empty or toroidal boundary. Then for any field $\F$,
$$\operatorname{dim} \operatorname{Im}\big(H_1(\partial M;\F) \to H_1(M;\F)\big) \geq b_0(\partial M).$$
\end{lemma}

Hence we obtain:

\begin{corollary}
For every $p\in\mathcal P$, the filtration $\mathbf G$ of $G$ does not induce the pro-$p$ topology on $G_1$.
\end{corollary}
\begin{proof}
Suppose $p\in\mathcal P$. Let $N_1\xrightarrow{\pi} N$ denote the finite regular cover corresponding to $G_1\trianglelefteq G$, and let $T$ be a boundary component of $\partial N$.
Then
$$b_0(\pi^{-1}(T)) = \frac{[G:G_1]}{[T:G_1\cap T]} \geq \frac{p^2(p^2-1)}{p^2}\geq 3.$$
Hence $\dim H_1(N_1;\F_p)\geq 3$ by the previous lemma.
If  $\mathbf G$  induced the pro-$p$ topology on $G_1$, then the pro-$p$ completion of $G_1=\pi_1(N_1)$ would be $p$-adic analytic (cf.~the remark following the proof of Lemma~\ref{lem:linearity}), contradicting Theorem~\ref{thm:Lubotzky}.
\end{proof}

Let us now turn to the second approach. In order to show that $G$ induces the pro-$p$ topology on $G_1$, it is convenient to refine the normal series $G\trianglerighteq G_1$:
$$G \trianglerighteq G_{0}:=G\cap \SL^1(2,\Lambda) \trianglerighteq  G_1.$$ Note that $G_1$ has $p$-power index in
$G_0$ and $G_0$ is residually $p$; hence $G_0$ induces the pro-$p$ topology on $G_1$, by the following well-known fact (cf.~\cite[Lemma~3.1.4~(a)]{RZ00}):

\begin{lemma}
Let
$$1\to A\to \G\to Q\to 1$$
be a group extension where $A$ is finitely generated, $\G$ is residually $p$, and $Q$ is a $p$-group.
Then $\G$ induces the pro-$p$ topology on $A$.
\end{lemma}

Thus it remains to show that $G$ induces the pro-$p$ topology on $G_0$.
We don't know whether this is the case. We do, however, have the following general fact  (inspired by \cite[Proposition~4.23]{LS07}), which might be of general interest. For this proposition and the lemmas used in its proof below, we fix a group extension
$$1\to A\xrightarrow{\iota} \G\to Q\to 1.$$
As a matter of convenience we shall assume that $A\trianglelefteq\G$ and $\iota$ is just the natural inclusion map.
Note that $\G$ acts on $A$ by conjugation, and this action descends to an action of $Q$ on $H_1(A;\Z)$ and hence on $H_1(A;\F_p)$. We say that $Q$ acts unipotently on $H_1(A;\F_p)$ if the image of the associated morphism $Q\to\Aut(H_1(A;\F_p))$ is a $p$-group.

\begin{proposition}\label{prop:pro-p induced}
Suppose $A$ is finitely generated,
$Q$ is finite and perfect, and $|M(Q)|$ is coprime to $p$.
Suppose $Q$ acts unipotently on $H_1(A;\F_p)$. Then $\G$ induces the pro-$p$ topology on $A$.
\end{proposition}

Here $M(Q)=H_2(Q;\Z)$  denotes the Schur multiplier of a finite group $Q$.
This proposition applies to extensions of $Q=\SL(2,\F_p)$:
if the Sylow $p$-subgroups of a finite group $Q$ are cyclic for $p>2$ and each Sylow $2$-subgroup of $Q$ is a generalized quaternion $2$-group, then $M(Q)=1$ \cite[Section~2.9]{Su82}. This holds for $Q=\SL(2,\F_p)$, for every $p$ (cf.~\cite[Section~3.6]{Su82}).
Also recall that  if $p>3$ then the only proper normal subgroup of $\operatorname{SL}(2,\F_p)$ is its center, hence $\operatorname{SL}(2,\F_p)$ is perfect.

Before we give the proof of this proposition in general, we first show two special cases. {\it In the lemmas below, we assume that
$Q$ is finite and perfect, and $|M(Q)|$ is coprime to $p$.}

\begin{lemma}
Suppose $A$ is a $p$-group and $A\leq Z(\G)$.  Then $\G$ induces the pro-$p$ topology on $A$.
\end{lemma}

\begin{proof}
We need to show that there exists a normal subgroup $H$ of $\G$ of $p$-power index with $A\cap H=1$. We consider $A$ as a trivial $Q$-module. Since $Q$ is perfect, we have $H^2(Q;A)\cong\Hom(M(Q),A)$. Now $A$ is a $p$-group and $|M(Q)|$ is not divisible by $p$, hence $H^2(Q;A)=0$. Thus $\G\cong A\times Q$, and this yields the existence of $H$ as desired.  
\end{proof}

\begin{lemma}\label{lem:pro-p unipotent}
Suppose $A$ is a $p$-group and $Q$ acts unipotently on $H_1(A;\F_p)$. Then $\G$ induces the pro-$p$ topology on $A$.
\end{lemma}
\begin{proof}
Let $\G\to\Aut(A)$ be the morphism arising from the action of $\G$ on~$A$, and let $K$ be its kernel. Then Lemma~\ref{lem:induced automorphisms} and the assumption that $Q$ acts unipotently on $H_1(A;\F_p)$ shows that $K$ has $p$-power index in $\G$. The normal subgroup~$KA/A$ of $Q$ has $p$-power index, hence equals $Q$ since the latter group is perfect. Thus we have a central group extension
$$1 \to A\cap K \to K\to Q\to 1.$$
By the previous lemma there exists a normal subgroup $H$ of $K$ of $p$-power index such that $A\cap H=(A\cap K)\cap H=1$.
Since  $\G$ is finitely generated, after replacing~$H$ by its normal core in $\G$ we may achieve that $H$ is normal in $\G$.
\end{proof}

\begin{proof}[Proof of Proposition~\ref{prop:pro-p induced}]
We have to show: for every normal $p$-power index subgroup $B$ of $A$, there is a normal subgroup $H$ of $\G$ of $p$-power index with $H\cap A\leq B$. By replacing $B$ by its normal core in $\G$ we may reduce to the case that  $B$ is normal in $\G$. We thus obtain a group extension $1\to A/B\to \G/B\to Q\to 1$ where $A/B$ is a finite $p$-group and $Q$ acts unipotently on $H_1(A/B;\F_p)$. By Lemma~\ref{lem:pro-p unipotent} applied to this group extension there exists a $p$-power index subgroup $\ol{H}$ of $\G/B$ with $(A/B)\cap\ol{H}=1$. Then the preimage $H$ of $\ol{H}$ under the natural surjection $\G\to\G/B$ has the desired property.
\end{proof}

With $p\in\mathcal P$ as in Proposition~\ref{prop:Long-Reid}, $G$ and $G_0$ fit into a short exact sequence
$$1 \to G_0 \to G\to Q=\SL(2,\F_p)\to 1.$$
In order to be able to apply our Proposition~\ref{prop:pro-p induced}, we would need $Q$ to act unipotently, and hence trivially, on $H_1(G_0;\F_p)$.
However, this is not the case:

\begin{proposition}\label{prop:Q acts non-trivially}
Let $M$ be an orientable hyperbolic $3$-manifold with non-empty boundary, and suppose $M_0\to M$ be the cover corresponding to a surjective morphism $\pi_1(M)\to Q=\SL(2,\F_p)$, where $p>5$.
Then $Q$ acts non-trivially on $H_1(M_0;\F_p)$.
\end{proposition}

For the proof of this proposition we use the well-known fact
(cf.~\cite[Chapter~IV, Corollary~6.8]{AM94} and \cite{Sw60}) that for odd $p$, the mod $p$ cohomology ring of $Q=\SL(2,\F_p)$ has the form
$$H^*(Q;\F_p) = \F_p[X, Y]/(Y^{2}),$$
equipped with the grading given by $\deg(X)=p-1$, $\deg(Y)=2p-3$. Hence
$$H^n(Q;\F_p) = \begin{cases}
\F_p & \text{if $(p-1)|n$ or $(p-1)|(n+1)$} \\
0    & \text{otherwise.}
\end{cases}$$
\begin{proof}[Proof of Proposition~\ref{prop:Q acts non-trivially}]
We assume for a contradiction that $Q$ acts trivially on $H_1(M_0;\F_p)$. We have
$$H^1(M_0;\F_p) = \operatorname{Hom}(H_1(M_0),\F_p) = H_1(M_0;\F_p)^*$$
as $Q$-modules. Also, the natural morphism $H_1(M_0;\F_p)\to H_1(M_0,\partial M_0;\F_p)$ is a surjective morphism of $Q$-modules, hence $Q$ acts trivially on $H_1(M_0,\partial M_0;\F_p)$ and thus, by
Poincar\'e-Lefschetz Duality, on $H^2(M_0;\F_p)$.
Therefore the Hochschild-Serre spectral sequence of the short exact sequence
$$1\to \pi_1(M_0)\to \pi_1(M)\to Q\to 1$$
looks like
$$\begin{matrix}
0 & & &\\
H^*(Q; W) & & &\\
H^*(Q; V) & &\Rightarrow & H^{i+j}(M;\F_p) \\
H^*(Q; \F_p) & & &
\end{matrix}$$
where both $V=H^1(M_0;\F_p)$ and $W=H^2(M_0;\F_p)$ are trivial $Q$-modules. Using the description of $H^*(Q;\F_p)$ above we now see that
since $p>5$,  the $\F_p$ with the coordinates $(k(p-1)-1,0)$, $k>1$, in the bottom row  of the second sheet $E_2$ only get hit by the zero differentials, hence survive in $E_\infty$, contradicting the fact that $H^*(M;\F_p)$ is finite-dimensional.
\end{proof}

This proposition shows that in order to successfully complete approach~(2)  to answering our question, i.e., to show that $G$ induces the pro-$p$ topology on $G_1$, one would have to use other methods than Proposition~\ref{prop:pro-p induced}.
As to approach~(3), that is, to prove that $H_1(G_1;\Z)$ is $p$-torsion free, it may be useful to again refine $G\trianglerighteq G_1$, this time as
$$G\trianglerighteq \gamma^p_2(G) \trianglerighteq G_1.$$
This leads to two subquestions, given $N$ as above.
\begin{enumerate}
\item Are there only finitely many $p$ such that for the  cover $N_0\to N$ corresponding to the natural group morphism $G=\pi_1(N)\to H_1(N;\F_p)$,  the abelian group $H_1(N_0;\Z)$ has non-trivial $p$-torsion?
\item Are there only finitely many $p$ such that for the cover $N_1\to N$ corresponding to the group morphism $G\to\SL(2,R/\mm)$, $H_1(N_1;\Z)$ has non-trivial $p$-torsion?
\end{enumerate}
Concerning subquestion~(1): currently available technology allows one to extract information about the prime-to-$p$-torsion in the homology of abelian $p$-power covers from algebraic data associated to $G$, cf., e.g., \cite{MS02}; however, little seems to be known about $p$-torsion in the homology of such covers. It is known that if $M$ is a knot complement in $S^3$
and $\widetilde{M}\to M$ is a $p$-power cyclic cover of $M$, then $H_1(\widetilde{M};\Z)$ has trivial $p$-torsion \cite{Go78}.  Moreover, by \cite[Theorem~5.11]{SW02}, if $M$ is the complement of a $2$-component link  $l=l_1\cup l_2$ in $S^3$ and $p$ does not divide the linking number of $l_1$ and $l_2$,
then for every abelian $p$-power index cover $\widetilde{M}\to M$ of~$M$, again $H_1(\widetilde{M};\Z)$ has only trivial $p$-torsion.
This leads us to conjecture that at least for link complements, subquestion~(1) has a positive answer.

As to (2), we do not know the answer. We note that Reznikov \cite[Theorem~9.2]{Re97} proved a result pointing to a negative answer; however his result only holds for certain covers of \emph{closed} $3$-manifolds: if $M$ is a closed $3$-manifold, and $M_1\to M$ the cover corresponding to a surjective group morphism $\pi_1(M)\to\SL(2,\F_p)$ where $p>5$, and $b_1(M)=b_1(M_1)=0$, then $H_1(M_1;\Z)$ has non-trivial $p$-torsion.

\section{The mod $p$ homology graph}\label{sec:mod p homology graph}

\noindent
In this final section we record another, potentially useful, application of our methods (Proposition~4 in the introduction):

\begin{proposition}\label{prop:injective on mod p homology}
Every closed orientable prime $3$-manifold $N$ has, for all but finitely many $p$, a finite cover $\widetilde{N}$ such that for every JSJ component $\widetilde{N}_{\widetilde{v}}$ of $\widetilde{N}$, the natural map $H_1(\widetilde{N}_{\widetilde{v}};\F_p)\to H_1(\widetilde{N};\F_p)$ is injective.
\end{proposition}

\index{graph!mod $p$ homology|textbf}

The proof of this proposition involves a close examination of some of the arguments in Sections~\ref{sec:virtual res p} and \ref{sec:proof of the main theorem}.

\subsection{Definition of $H_1(\scrg;\F_p)$ and its basic properties}
In this subsection we
let $\scrg$ be a graph of finitely generated groups based on a graph $Y$, and we let $T$ be a maximal subtree of $Y$. We assume that for a given $p$, we have
\begin{equation}\label{eq:gammap2}
\gamma^p_2(G_v)\cap f_e(G_e) = \gamma^p_2(f_e(G_e))\qquad\text{for all $e\in E(Y)$, $v=t(e)$.}
\end{equation}
Then the collection $\gamma^p_2(\scrg):=\{\gamma^p_2(G_v)\}_{v\in V(Y)}$ of normal subgroups is compatible, and  for each edge $e$ of $Y$, the edge morphism $f_e\colon G_e\to G_{v}$ (where $v=t(e)$) induces an injective morphism $$(f_e)_*\colon H_1(G_e;\F_p)\to H_1(G_v;\F_p).$$
We call the graph of groups $\scrg/\gamma^p_2(\scrg)$ the {\it mod $p$ homology graph}\/ of $\scrg$, and denote it by $H_1(\scrg;\F_p)$.
It has  the same underlying graph $Y$ as $\scrg$, vertex and edge groups $H_1(G_v;\F_p)$ respectively $H_1(G_e;\F_p)$, and edge morphisms $(f_e)_*$.

The fundamental groups of $\scrg$ and of $H_1(\scrg;\F_p)$ have the same first mod $p$ homology:
the natural surjections $G_v\to H_1(G_v;\F_p)$ and $G_e\to H_1(G_e;\F_p)$ define a morphism $\scrg\to H_1(\scrg;\F_p)$ of graphs of groups, hence yield a surjective morphism
$$G=\pi_1(\scrg) \to \pi_1(H_1(\scrg;\F_p))$$
which induces an isomorphism in homology mod $p$, by the following lemma:

\begin{lemma}\label{lem:iso in homology}
Let $M$ be a trivial $G$-module, and let $\scrh=\{H_v\}_{v\in V(Y)}$ be a compatible collection of normal subgroups of $\scrg$ such that for each $v\in V(Y)$ and $e\in E(Y)$, the natural surjections
$$G_v\to G_v/H_v, \qquad G_e\to G_e/f_e^{-1}(H_{t(e)})$$
induce isomorphisms
$$H_1(G_v;M) \overset{\cong}{\longrightarrow} H_1(G_v/H_v;M),\qquad
H_1(G_e;M) \overset{\cong}{\longrightarrow} H_1(G_e/f_e^{-1}(H_{t(e)});M).$$
Then the natural surjective morphism $\pi_1(\scrg)\to\pi_1(\scrg/\scrh)$ induces an isomorphism
$$H_1(\pi_1(\scrg);M) \overset{\cong}{\longrightarrow} H_1(\pi_1(\scrg/\scrh);M).$$
\end{lemma}

This is shown by an easy application of the five-lemma in combination with the Mayer-Vietoris sequence of $\scrg$, cf.~Section~\ref{sec:MV}.
This fact together with the remark following Lemma~\ref{lem:homology and fiber sums} yields:

\begin{corollary}
Suppose there exists a morphism $\pi_1(H_1(\scrg;\F_p))\to \Sigma$ to an elementary abelian $p$-group which is injective on the vertex groups of $\scrg$. Then the natural inclusions $G_v\to G$ yield injective morphisms $H_1(G_v;\F_p) \to H_1(G;\F_p)$.
\end{corollary}

\subsection{Achieving injectivity of mod $p$ homology}
In this subsection we let
$\scrg$ be any graph of finitely generated groups (based on a graph $Y$).

\begin{lemma}\label{lem:injective on homology}
Suppose $\scrg$ satisfies \eqref{eq:gammap2}.
Then there exists a morphism of graphs of groups $\widetilde{\scrg}\to\scrg$ of finite degree such that for each vertex $\widetilde{v}\in V(\widetilde{Y})$, the natural inclusion $\widetilde{G}_{\widetilde{v}}\to\widetilde{G}=\pi_1(\widetilde{\scrg})$ yields an embedding   $H_1(\widetilde{G}_{\widetilde{v}};\F_p) \to H_1(\widetilde{G};\F_p)$.
\end{lemma}
\begin{proof}
We apply the same argument as in the proof of Corollary~\ref{cor:unfolded, 1} with $\scrh=\gamma^p_2(\scrg)$. We obtain a morphism
$\psi\colon\pi_1(Y)\to A$ to a finite group, with corresponding unfolding $\widetilde{\scrg/\scrh}\to\scrg/\scrh$ of $\scrg/\scrh$, and a morphism $\pi_1(\widetilde{\scrg/\scrh})\to\Sigma$ to an elementary abelian $p$-group which is injective when restricted to each vertex group of $\widetilde{\scrg/\scrh}$. Let $\phi\colon\widetilde{\scrg}\to\scrg$ be the unfolding to $\scrg$ along $\psi$.
Since $\phi^{-1}(\scrh)=\gamma^p_2(\widetilde{\scrg})$, the graphs of groups $\widetilde{\scrg/\scrh}$ and $H_1(\widetilde{\scrg};\F_p)=\widetilde{\scrg}/\gamma^p_2(\widetilde{\scrg})$
are naturally isomorphic, and
by Lemma~\ref{lem:separating, 3},  $\widetilde{\scrg}$ also satisfies \eqref{eq:gammap2}.
Thus the previous corollary applies to $\widetilde{\scrg}$ in place of $\scrg$.
\end{proof}

Similarly modifying the proof of Theorem~\ref{thm:reduction theorem, 3} shows:

\begin{lemma}
Suppose all edge groups of $\scrg$ are abelian $p$-torsion free, and suppose for some $\ell\geq 0$, $\scrg$ admits a $p$-compatible filtration of level $\ell$.
Then there exists a morphism $\widetilde{\scrg}\to\scrg$ of graphs of groups such that the image of $\widetilde{G}=\pi_1(\widetilde{\scrg})$ in $\pi_1(\scrg)$ has finite index, and for each vertex $\widetilde{v}\in V(\widetilde{Y})$, the natural inclusion $\widetilde{G}_{\widetilde{v}}\to\widetilde{G}=\pi_1(\widetilde{\scrg})$ yields an embedding   $H_1(\widetilde{G}_{\widetilde{v}};\F_p) \to H_1(\widetilde{G};\F_p)$.
\end{lemma}

\begin{proof}
Let $\mathbf G$ be a $p$-compatible filtration of $\scrg$ of level $\ell$; then $\mathbf G$ is $p$-excellent. Arguing as in the proof of Theorem~\ref{thm:reduction theorem, 3}, we let $\phi\colon\widetilde{\scrg}\to\scrg$ be a morphism of graphs of groups obtained by applying Proposition~\ref{prop:commoncover} to the compatible collection of normal subgroups $\{G_{v,1}\}_v$ of $\scrg$. Then the filtration $\widetilde{\mathbf G}:=\phi^{-1}(\mathbf G)$  of $\widetilde{\scrg}$ is complete;
in fact, $\widetilde{\mathbf G}$ is a central $p$-filtration of $\widetilde{\scrg}$. Now, in
the proof of Theorem~\ref{thm:reduction theorem, 3} we simply applied Lemma~\ref{lem:separating, 2} to show that  $\widetilde{\mathbf G}$ intersects to a uniformly $p$-potent filtration on the image of each edge group of $\widetilde{\scrg}$ under the edge morphisms. However, the somewhat more complicated Lemma~\ref{lem:separating, 4} shows that in our situation, these intersections are precisely
the lower central $p$-filtrations of the images of the edge groups under the edge morphisms. By Lemma~\ref{lem:canonical filtration, 2}, the lower central $p$-filtration $\gamma^p(\widetilde{\scrg})$ of~$\widetilde{\scrg}$ also has this property, i.e., $\widetilde{\scrg}$ satisfies \eqref{eq:gammap2}.
So may now apply Lemma~\ref{lem:injective on homology} to $\widetilde{\scrg}$ in place of $\scrg$ to complete the proof.
\end{proof}

\subsection{Proof of Proposition~\ref{prop:injective on mod p homology}}
Let $N$ be a closed orientable prime $3$-manifold $N$. Clearly we may assume that $N$ has a non-trivial JSJ decomposition.
As shown in Section~\ref{sec:reduction to seifert simple}, after replacing $N$ by a finite cover if necessary, we may then reduce to the case that $N$ is Seifert simple.
Then, by Propositions~\ref{prop:virtphyp} and \ref{prop:seifert simple lower p-central}, for all but finitely many $p$, each JSJ component of $N$ admits a boundary compatible $p$-filtration of level $1$, and so by the previous lemma, the proposition follows. \qed

\backmatter

\bibliographystyle{amsalpha}

\printindex

\end{document}